\documentclass{amsart}

\usepackage[utf8]{inputenc}
\usepackage[T1]{fontenc}
\usepackage[margin=1in]{geometry}
\usepackage[section]{placeins}

\usepackage{tikz}
\usepackage{multicol}
\usepackage{amsmath,amssymb,amsthm}
\usepackage{graphicx}
\usepackage{color}
\usepackage{ifthen}
\usepackage{booktabs}
\usepackage[english]{babel}
\usepackage{enumerate}
\usepackage{float}
\usepackage{bm}
\usepackage{mathtools}
\usepackage{algorithm,algorithmicx,algpseudocode}

\newcommand{\pvct}[1]{\bm{#1}}
\newcommand{\vct}[1]{\bm{\mathsf{#1}}}
\newcommand{\mtx}[1]{\bm{\mathsf{#1}}}
\newcommand{\pxx}{\pvct{x}}
\newcommand{\pyy}{\pvct{y}}

\newcommand{\cx}{\mathbb{C}}

\newcommand{\reals}{\mathbb{R}}

\newcommand{\alp}{\alpha}
\newcommand{\ba}{\beta}
\newcommand{\ta}{\theta}

\newcommand{\sa}{\sigma}

\newcommand{\ka}{\kappa}

\newcommand{\eps}{\varepsilon}
\newcommand{\Om}{\Omega}

\newcommand{\Da}{\Delta}

\newcommand{\mc}{\mathcal}
\newcommand{\wh}{\widehat}
\newcommand{\wt}{\widetilde}

\newtheorem{lemma}{Lemma}

\numberwithin{lemma}{section}

\theoremstyle{definition}
\newtheorem{remark}{Remark}
\numberwithin{remark}{section}

\begin{document}

\begin{center}
\textbf{\textsc{An accelerated, high-order accurate direct solver for the Lippmann-Schwinger equation
for acoustic scattering in the plane}}

\vspace{3mm}

\textit{{Abinand Gopal, Mathematical Institute, University of Oxford, Oxford, United Kingdom}}

\vspace{1mm}

\textit{{Per-Gunnar Martinsson, Department of Mathematics, University of Texas at Austin, USA}}

\vspace{3mm}

\begin{minipage}{140mm}
\textbf{Abstract:}
  An efficient direct solver for solving the Lippmann-Schwinger integral
  equation modeling acoustic scattering in the plane is presented.
  For a problem with
  $N$ degrees of freedom, the solver constructs an approximate inverse
  in $\mc{O}(N^{3/2})$ operations and then, given an incident field,
  can compute the scattered field in $\mc{O}(N \log N)$ operations.
  The solver is based on a previously published direct solver for integral
  equations that relies on rank-deficiencies in the off-diagonal blocks;
  specifically, the so-called Hierarchically Block Separable format is used.
  The particular solver described here has been reformulated in a way that
  improves numerical stability and robustness, and exploits the particular
  structure of the kernel in the Lippmann-Schwinger equation to accelerate
  the computation of an approximate inverse.
  The solver is coupled with a Nystr\"om discretization on a regular square
  grid, using a quadrature method developed by Ran Duan and Vladimir Rokhlin
  that attains high-order accuracy despite the singularity in the kernel of
  the integral equation.
  A particularly efficient solver is obtained when the direct solver is run at
  four digits of accuracy, and is used as a preconditioner to GMRES, with each forwards
  application of the integral operators accelerated by the FFT.
  Extensive numerical experiments are presented that illustrate the high performance
  of the method in challenging environments.
  Using the $10^{\rm th}$-order accurate version of the Duan-Rokhlin
  quadrature rule, the scheme is capable of solving problems on domains
  that are over 500 wavelengths wide to residual error below
  $10^{-10}$ in a couple of hours on a workstation, using 26M degrees of freedom.
\end{minipage}

\end{center}

\section{Introduction}
A recurring problem in scientific computing is the simulation of the
scattering of acoustic waves in media with variable wave speed, an
environment which arises in a wide variety of applications such as
seismology, ultrasound tomography, and sonar. We are specifically
interested in a situation in time-harmonic scattering with
\textit{wavenumber} $\ka > 0$, where a given \textit{incident field} $u^{\rm
inc}$ impinges on a medium with smooth wave speed $v = v(\pxx)$ that takes
on a constant value $v_{\rm free}$ everywhere outside of a finite domain
$\Omega$. The \textit{scattered field} can then be written
\begin{equation}
  u(\pxx) = \int_\Om G(\pxx,\pyy) \sa(\pyy)\,d\pyy,\quad x \in \reals^2,
  \label{eq:ansatz}
\end{equation}
where
\begin{equation}
  G(\pxx,\pyy) = \frac{i}{4} H_0(\ka\|\pxx-\pyy\|).
  \label{eq:kernel}
\end{equation}
The density $\sa$ is unknown a priori and can be obtained by solving the
\textit{Lippmann-Schwinger equation} given by
\begin{equation}
  \sa(\pxx) + \ka^2 b(\pxx) \int_{\Om} G(\pxx,\pyy)
  \sa(\pyy)\,d\pyy = -\ka^2 b(\pxx) u^{\rm inc}(\pxx), \quad \pxx
  \in \reals^2.
  \label{eq:ls}
\end{equation}
The function $b$, which we refer to as a \textit{scattering potential},
specifies how much the wave speed deviates from the free-space wave speed:
\[
  b(\pxx) = 1 - \frac{v_{\rm free}^2}{v(\pxx)^2}.
\]
As long as $b$ is bounded above by 1, existence and uniqueness of the
solution is guaranteed \cite[Theorem 8.7]{colton2012inverse}.

The presence of the free-space fundamental solution in \eqref{eq:ansatz}
makes the formulation immune to pollution error
(cf.~\cite{babuska1997pollution}) and guarantees that the scattered field
will satisfy the Sommerfeld radiation condition, avoiding the need for
special techniques, such as absorbing boundary conditions
\cite{engquist1977absorbing} or perfectly matched layers
\cite{berenger1994perfectly}, which would be required if a differential
equation formulation of the scattering problem were directly discretized.
Moreover, \eqref{eq:ls} is typically a Fredholm equation of the second kind
(for instance when viewed as an operator on $L^{2}(\Omega)$), which admits
well-conditioned high-order accurate discretizations.

\begin{remark}
  The equation \eqref{eq:ls} is obtained by inserting \eqref{eq:ansatz} as
  an ansatz into the variable-coefficient Helmholtz equation
  \begin{equation}
    -\Da u^{\rm tot}(\pxx) - \ka^2 (1 - b(\pxx)) u^{\rm tot}(\pxx) = 0,\quad
    \pxx \in \reals^2,
    \label{eq:tot}
  \end{equation}
where $u^{\rm tot} = u^{\rm inc} + u$ is the total field, and then applying
the fact that the incident field satisfies the homogeneous free-space
Helmholtz equation.
\end{remark}

Integral equation formulations of scattering problems do come with
certain trade-offs. Particularly, they upon discretization lead to linear
systems with dense coefficient matrices. Additionally, the fact that the
kernel in \eqref{eq:ls} is singular necessitates some care in order to
attain high-order accurate discretizations. Both of these issues were
addressed in \cite{duan2009high}, where the authors present a high-order
accurate quadrature rule that is tailored specifically for \eqref{eq:ls},
and use an FFT to rapidly apply the global operator. (This is possible
since the kernel $G(\pxx,\pyy)$ in \eqref{eq:ls} is translation invariant,
and a regular square mesh is used for the discretization.) This allows for
iterative solvers to be used, and the overall solver is shown to be highly
effective in many environments. However, it turns out that as the
wavenumber $\kappa$ grows, the iteration count required grows
prohibitively, especially when the scatterer is highly refractive.  This
may seem surprising given that \eqref{eq:ls} is a second kind Fredholm
equation, but it is explained by the fact that as $\kappa$ grows, the
magnitude of the second term in \eqref{eq:ls} grows as $\kappa^{2}$, which
means that it quickly comes to dominate the first term, and the overall
behavior of the integral equation tends towards the behavior of a first
kind equation.

To overcome the difficulties posed by the slow convergence of iterative methods
for problems with oscillatory solutions, one may look to utilize a direct solver,
leveraging the rank structure of the linear system. Many data-sparse
matrix formats and corresponding inversion algorithms have been proposed,
including $\mc{H}$- \cite{hackbusch1999sparse}, $\mc{H}^2$-
\cite{borm2010efficient}, Hierarchically Off-Diagonal Low Rank (HODLR)
\cite{ambikasaran2013fast} and Hierarchically Block/Semi- Separable
(HBS/HSS) matrices
\cite{martinsson2005fast,chandrasekaran2004divide,xia2010fast}.
These format are differentiated by whether or not the
interactions between adjacent patches in the domain are compressed and
whether or not the basis matrices used in the hierarchical structure are
taken to be nested \cite[Sec.~15.8]{martinsson2019book}.
In this work, we rely on the so called HBS format, which
uses nested bases ($\mathcal{H}^2$ rather than $\mathcal{H}$ structure),
and does compress adjacent patches (``weak admissibility''). These ideas
trace back to, e.g.,
\cite{martinsson2005fast,1996_mich_elongated,greengard1991numerical,starr_rokhlin}.

There are two general approaches to inverting HBS matrices. The first
approach is to compute the factors of a data-sparse representation of the
inverse through use of explicit formulas in terms of the factors of the
original HBS matrix. The prototypical example of this is reviewed in
\cite{gillman2012direct}. Often, these can be viewed as derived through the
recursive usage of variants of the Woodbury formula. Generally, such
methods achieve high practical speed and are relatively memory efficient.
The weakness of these methods is that they can run into stability issues if
ill-conditioned matrices are encountered during the inversion process. This
is especially prevalent for problems like the Lippmann-Schwinger equation
where the off-diagonal blocks of the discretized system can vary
drastically in magnitude due to the presence of the scattering potential.
The second approach is to use rank structure to create an equivalent sparse
linear system and then use black box sparse direct solvers to compute an LU
factorization. An example of this is presented in \cite{ho2012fast}. The
advantage of this approach is that pivots can be adaptively determined to
reduce fill-in while still avoiding the inversion of ill-conditioned
matrices, leading to a compromise between stability and speed. However,
finding the optimal pivots is generally unfeasible, and in our experience,
without careful tuning, this approach while very robust is significantly
more costly in both computational time and memory than the first approach.

In this paper, we present a new direct solver, following the first
approach, where the inversion formulas are intentionally derived to be
stable for the Lippmann-Schwinger problem and the same basis matrices are
used for both the original HBS matrix and the computed inverse. This solver
is based on discrete scattering matrices and is closely related to the
solver derived in \cite{bremer2015high} for scattering in homogeneous
media. We empirically find that this scheme avoids instabilities, while
maintaining, if not exceeding, the performance of other explicit
inversion formula methods. Furthermore, it compresses only the free-space
kernel which leads to speedups in environments where multiple scattering
potentials are of interest, such as inverse scattering. The direct solver
can also be used as a preconditioner and then combined with an iterative
solver. We demonstrate the stability and speed of our solver through
extensive numerical experiments on several challenging problems.

The paper is organized as follows: In Section \ref{s:prelim}, we
review some preliminaries necessary to describe our solver, namely the
interpolative decomposition and the high-order discretization of
\eqref{eq:ls}. Section \ref{s:hbscomp} defines the HBS matrix format and
discusses efficient techniques for computing the compressed version of our
linear system in this format. Section \ref{s:direct} describes our direct
solver, gives a brief complexity estimate, and a discussion on how this
method relates to existing approaches. In Section \ref{s:precon}, we
motivate using our direct solver as a preconditioner. Section
\ref{s:numerics} contains our numerical tests, and the final section
summarizes our findings and outlines some potential directions for future
work.

\section{Preliminaries}
\label{s:prelim}
\subsection{Interpolative decomposition (ID)}
\label{s:id}
The interpolative decomposition is a matrix factorization in which a subset of
the columns (rows) are used to approximately span the column (row) space of
the matrix. To be precise, for a matrix $\mtx{A} \in \cx^{m \times n}$, a rank
$k$ column ID takes the form
\begin{equation}
  \mtx{A} \approx \mtx{A}(:,\mc{I}) \mtx{W},
  \label{eq:id}
\end{equation}
where $\mc{I} \subset [1,2,\hdots,n]$ is a $k$-element index vector and
$\mtx{W} \in \cx^{k \times n}$ is a matrix for which
$\mtx{W}(\colon,\mc{I}) = \mtx{I}$ (in other words, $\mtx{W}$ contains
the $k \times k$ identity matrix as a submatrix). For \eqref{eq:id} to be a useful
approximation, we would like the bound
\begin{equation}
  \| \mtx{A} -  \mtx{A}(:,\mc{I}) \mtx{W} \|_2 \leq (1+C) \sa_{k+1}(\mtx{A})
  \label{eq:idapprox}
\end{equation}
to hold for some modest size constant $C$. In \eqref{eq:idapprox},
$\sa_{k+1}(\mtx{A})$ is the $(k+1)$th largest singular value of $\mtx{A}$.
We call the elements of $\mc{I}$ the \textit{skeleton indices} and the
remaining indices $[1,2,\cdots,n] \backslash \mc{I}$ the \textit{residual
indices}. The rank $k$ row ID, where all rows of $\mtx{A}$ are represented
as linear combinations of a select subset, can be defined analogously:
\begin{equation*}
  \mtx{A} \approx \mtx{X} \mtx{A}(\mc{J},:),
\end{equation*}
where now $\mc{J}$ is a $k$-element index vector and $\mtx{X} \in \cx^{m
\times k}$ satisfies $\mtx{X}(\mc{J},\colon) = \mtx{I}$.

The ID is well suited for our application since it is highly storage
efficient (observe that whenever the matrix entries of $\mtx{A}$ can be
computed on the fly, the factorization \eqref{eq:id} requires the storage
of only the index vector $\mc{I}$, and the $k \times (n-k)$ nonzero
entries of $\mtx{W}$, regardless of how large $m$ is). Moreover, the
ID enables us to construct low rank approximations using the method
of ``proxy sources'' (see \cite[Sec.~5.2]{cheng2005compression} and
\cite[Sec.~17.2]{martinsson2019book}), which is
crucial for achieving a complexity better than $\mc{O}(N^2)$,
cf.~Section \ref{s:compression}. An additional
benefit is that the presence of the identity submatrix in $\mtx{W}$ and
$\mtx{X}$ can be
leveraged for slightly more efficient matrix-matrix multiplication.

In addition to satisfying \eqref{eq:idapprox}, it is for purposes of
numerical stability important that the
entries of $\mtx{W}$ in \eqref{eq:id} are not large in magnitude.
It can be shown that there exists an ID such that all
entries of $\mtx{W}$ have magnitude at most 1, but we are not aware of a
polynomial-time algorithm for computing a factorization with this property
guaranteed. While there do exist polynomial-time algorithms when this
condition is slightly relaxed (see for example \cite{gu1996efficient}), in
practice we find that it is more efficient and effectively as robust to
compute an ID using the standard column-pivoted QR algorithm (CPQR). We
refer the reader to \cite{cheng2005compression} for more details. For a
rank $k$ approximation, only $\mc{O}(mnk)$ operations are required. This
can be further reduced to $\mc{O}(m n \log k + k^2 n)$ through the use of
randomized algorithms \cite{woolfe2008fast}.

\subsection{High-order accurate discretization}
\label{s:quadcorr}

To discretize \eqref{eq:ls}, we place a uniform grid $\{\pxx_i\}_{i=1}^N$
with mesh width $h$ on the rectangle $\Omega$. (If needed, we can slightly
increase the length of one side of the rectangle to ensure that it is a multiple of $h$.)
We then use a Nystr\"om discretization which collocates \eqref{eq:ls} at the
nodes $\{\pxx_i\}_{i=1}^N$, while replacing the integrand by a quadrature.
We could in this step use a trapezoidal rule quadrature, and simply omit one
point to avoid the singularity in the kernel. Such a ``punctured trapezoidal rule''
Nystr\"om discretization would take the form
\begin{equation} \sa(\pxx_i) + \ka^2 h^2
  b(\pxx_i) \sum_{\substack{j=1 \\ j \neq i}}^N
G(\pxx_i,\pxx_j)\sa(\pxx_j) = -\ka^2 b(\pxx_i) u^{\rm
inc}(\pxx_i),
\qquad i \in \{1,2,\dots,N\}.
\label{eq:nystrom}
\end{equation}
This crude treatment of the singularity results in roughly $2^{\rm
nd}$-order convergence.

To improve the convergence order, we employ a correction technique proposed by
Duan and Rokhlin \cite{duan2009high} that modifies the quadrature weights
for a small number of elements near the diagonal. These corrections are derived by
exploiting the form of the singularity of the Hankel function in \eqref{eq:kernel}
and smoothness in $\sa$, and result in a modified linear system
\begin{multline}
  \sa(\pxx_i) + \ka^2 h^2 b(\pxx_i) \sum_{\substack{j=1 \\ j \neq i}}^N
  G(\pxx_i,\pxx_j)\sa(\pxx_j) + \ka^2 h^2 b(\pxx_i) \sum_{j \in \mc{S}_i}
  \tau_{ij} \sa(\pxx_j)\\ = -\ka^2 b(\pxx_i)
  u^{\rm inc}(\pxx_i),
  \qquad i \in \{1,2,\dots,N\},
  \label{eq:corrected}
\end{multline}
where $\mc{S}_{i}$ identifies the set of indices to be corrected.
Remarkably, convergence of order $\mc{O}(h^4 \log (1/h)^2)$ can be
attained by correcting just a single point (so that $\mc{S}_{i} = \{i\}$).
To attain higher orders, slightly larger stencils are required, as shown
in Figure \ref{fig:qc}; for instance order $\mc{O}(h^{10} \log (1/h)^2)$
requires 25 points. Experimentally,
we do not observe the effect of the logarithm terms in these convergence
orders, so we informally refer to the order of these rates in the sequel by
just the polynomial part (e.g. we refer to the $\mc{O} (h^4 \log (1/h)^2)$
scheme as a $4^{\rm th}$-order scheme) for ease of exposition.
The quadrature corrections $\tau_{ij}$ depend only on the quantity $\ka h$, in
addition to the order of the correction, so after a
precomputation they can be cheaply interpolated for different values of
$\ka$ and $h$. We refer the reader to \cite{duan2009high} for the details
of computing the quadrature corrections.

\begin{figure}
  \centering
  \,\hspace{2mm} $O(h^4 \log(1/h)^2)$ \hspace{7mm} $O(h^6 \log(1/h)^2)$ \hspace{7mm}
  $O(h^8 \log(1/h)^2)$ \hspace{7mm} $O(h^{10} \log(1/h)^2)$ \vspace{.5em}

  \scalebox{1.2}{
  \begin{tikzpicture}
    \draw[step=0.3,black,thick] (.1,.1) grid (2.3,2.3);
    \fill[red, draw=black] (1.2,1.2) circle (1mm);
  \end{tikzpicture}
  \hspace{1.5em}
  \begin{tikzpicture}
    \draw[step=0.3,black,thick] (.1,.1) grid (2.3,2.3);
    \fill[red, draw=black] (1.2,1.2) circle (1mm);
    \fill[red, draw=black] (.9,1.2) circle (1mm);
    \fill[red, draw=black] (1.5,1.2) circle (1mm);
    \fill[red, draw=black] (1.2,1.5) circle (1mm);
    \fill[red, draw=black] (1.2,.9) circle (1mm);
  \end{tikzpicture}
  \hspace{1.5em}
  \begin{tikzpicture}
    \draw[step=0.3,black,thick] (.1,.1) grid (2.3,2.3);
    \fill[red, draw=black] (1.2,1.2) circle (1mm);
    \fill[red, draw=black] (.9,1.2) circle (1mm);
    \fill[red, draw=black] (1.5,1.2) circle (1mm);
    \fill[red, draw=black] (1.2,1.5) circle (1mm);
    \fill[red, draw=black] (1.2,.9) circle (1mm);
    \fill[red, draw=black] (1.5,1.5) circle (1mm);
    \fill[red, draw=black] (.9,.9) circle (1mm);
    \fill[red, draw=black] (.9,1.5) circle (1mm);
    \fill[red, draw=black] (1.5,.9) circle (1mm);
    \fill[red, draw=black] (1.8,1.2) circle (1mm);
    \fill[red, draw=black] (1.2,1.8) circle (1mm);
    \fill[red, draw=black] (.6,1.2) circle (1mm);
    \fill[red, draw=black] (1.2,.6) circle (1mm);
  \end{tikzpicture}
  \hspace{1.5em}
  \begin{tikzpicture}
    \draw[step=0.3,black,thick] (.1,.1) grid (2.3,2.3);
    \fill[red, draw=black] (1.2,1.2) circle (1mm);
    \fill[red, draw=black] (.9,1.2) circle (1mm);
    \fill[red, draw=black] (1.5,1.2) circle (1mm);
    \fill[red, draw=black] (1.2,1.5) circle (1mm);
    \fill[red, draw=black] (1.2,.9) circle (1mm);
    \fill[red, draw=black] (1.5,1.5) circle (1mm);
    \fill[red, draw=black] (.9,.9) circle (1mm);
    \fill[red, draw=black] (.9,1.5) circle (1mm);
    \fill[red, draw=black] (1.5,.9) circle (1mm);
    \fill[red, draw=black] (1.8,1.2) circle (1mm);
    \fill[red, draw=black] (1.2,1.8) circle (1mm);
    \fill[red, draw=black] (.6,1.2) circle (1mm);
    \fill[red, draw=black] (1.2,.6) circle (1mm);
    \fill[red, draw=black] (.6,1.5) circle (1mm);
    \fill[red, draw=black] (.9,1.8) circle (1mm);
    \fill[red, draw=black] (1.5,1.8) circle (1mm);
    \fill[red, draw=black] (1.8,1.5) circle (1mm);
    \fill[red, draw=black] (.6,.9) circle (1mm);
    \fill[red, draw=black] (.9,.6) circle (1mm);
    \fill[red, draw=black] (1.5,.6) circle (1mm);
    \fill[red, draw=black] (1.8,.9) circle (1mm);
    \fill[red, draw=black] (2.1,1.2) circle (1mm);
    \fill[red, draw=black] (1.2,2.1) circle (1mm);
    \fill[red, draw=black] (.3,1.2) circle (1mm);
    \fill[red, draw=black] (1.2,.3) circle (1mm);
  \end{tikzpicture}
  }
  \caption{The ``stencils'' of the Duan-Rokhlin quadrature corrections
  described in Section \ref{s:quadcorr}. The index set $\mc{S}_{i}$ is
  shown with red dots on a local grid centered around $\pxx_{i}$. The
  $4^{\rm th}$-order stencil (leftmost) requires only a single correction point, while
  the $6^{\rm th}$-, $8^{\rm th}$-, and $10^{\rm th}$-order stencils
  require 5, 13, and 25 points, respectively.}
  \label{fig:qc}
\end{figure}
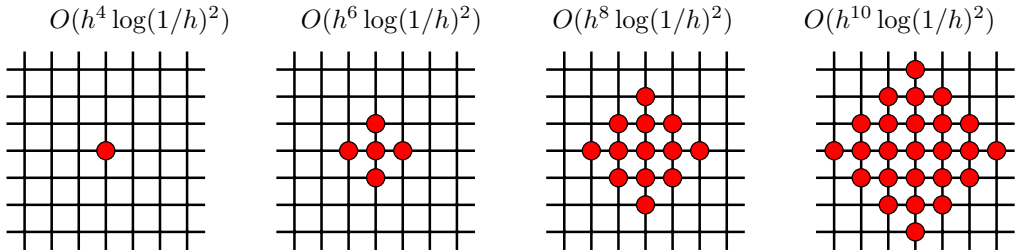

For notational convenience, we hereafter restrict our focus to the case of
$4^{\rm th}$-order (i.e. diagonal) corrections and outline as we go any
relevant additional considerations for larger stencils. Equation
\eqref{eq:corrected} can now be expressed as the linear system
\begin{equation}
  (\mtx{I} + \mtx{B}\mtx{G}) \vct{q}  = \vct{f},
  \label{eq:linsys}
\end{equation}
where $\mtx{I}$ is the $N \times N$ identity matrix, $\mtx{B}$ is a
diagonal matrix with $\mtx{B}(i,i) = \ka^2 b(\pxx_i)$, $\vct{q}(i) =
\sa(\pxx_i)$, $\vct{f}(i) = -\ka^2 b(\pxx_i) u^{\rm inc}(\pxx_i)$, and
$\mtx{G}$ is the matrix with entries given by
\begin{equation}
  \mtx{G}(i,j) = \begin{cases} h^2 G(\pxx_i,\pxx_j), &
  i \neq j \\ h^2 \tau_{ii}, & i = j \end{cases},
  \label{eq:Gmat}
\end{equation}
for $i,j=1,2,\hdots,N$. Incorporating higher-order corrections requires
adding $\mtx{B}$ times a Toeplitz matrix to the left-hand side of
\eqref{eq:linsys}, so the coefficient matrix in \eqref{eq:linsys} can
always be applied to vectors in quasilinear time using the FFT.

\section{HBS compression}
\label{s:hbscomp}
The linear system \eqref{eq:linsys} involves a coefficient matrix that is
dense, but rank structured. In this section, we describe the so-called
\textit{Hierarchically Block Separable (HBS)} format that we use to
represent this matrix. In Section \ref{s:hbs}, we briefly introduce our
notation, referring to \cite[Ch.~14]{martinsson2019book} for additional details.
Section \ref{s:compression} describes how to efficiently construct the low rank
factorizations that are required in what is sometimes referred to as the
``compression step'' in a rank-structured algorithm.

\subsection{HBS matrices}
\label{s:hbs}

To define an $N \times N$ HBS matrix, we first need a binary tree structure
on the row and column indices of the matrix. We label the root of the tree
by 1, and form the corresponding index vector $\mc{I}_1 = [1,2,\cdots,N]$.
We partition the indices into two disjoint index vectors $\mc{I}_2$ and
$\mc{I}_3$ (i.e. $\mc{I}_1 = \mc{I}_2 \cup \mc{I}_3$ and $\mc{I}_2 \cap
\mc{I}_3 = \varnothing$), which correspond to nodes on the tree which are
the children of the root node and which we label as 2 and 3. We similarly
split $\mc{I}_2$ into index vectors $\mc{I}_4$ and $\mc{I}_5$, which then
correspond to the children of node 2 on the tree and are accordingly
labeled nodes 4 and 5. We then do the analogous partitioning of $\mc{I}_3$
to obtain $\mc{I}_6$ and $\mc{I}_7$. We continue this process until all
nodes with no children have only a handful of indices. We refer to nodes
without children as \textit{leaves} and all other nodes in the tree as
\textit{parents}.

The HBS format is flexible and can be used with a variety of different
trees, but for simplicity we restrict our attention to uniform trees that
have $2^{L}$ leaf boxes for some positive integer $L$. We label the
\textit{levels} in the tree so that $\ell=0$ denotes the root, level $\ell$
has $2^{\ell}$ nodes, and the leaves are at level $\ell = L$.  Figure
\ref{fig:tree} illustrates our node labeling convention for a tree with $L
= 3$.

\begin{figure}
  \centering
  \begin{tikzpicture}

    \draw (0pt,0pt) -- (-80pt,-30pt);
    \draw (0pt,0pt) -- (80pt,-30pt);
    \draw[fill=white] (0pt,0pt) circle (10pt);
    \node[black] at (0pt,0pt) {\large 1};

    \draw (-80pt,-30pt) -- (-115pt,-65pt);
    \draw (-80pt,-30pt) -- (-45pt,-60pt);
    \draw (80pt,-30pt) -- (45pt,-60pt);
    \draw (80pt,-30pt) -- (115pt,-60pt);

    \draw[fill=white] (-80pt,-30pt) circle (10pt);
    \node[black] at (-80pt,-30pt) {\large \textcolor{black}{$2$}};
    \draw[fill=white] (80pt,-30pt) circle (10pt);
    \node[black] at (80pt,-30pt) {\large \textcolor{black}{$3$}};

    \draw (-115pt,-65pt) -- (-135pt,-105pt);
    \draw (-115pt,-65pt) -- (-95pt,-105pt);
    \draw (-45pt,-65pt) -- (-65pt,-105pt);
    \draw (-45pt,-65pt) -- (-25pt,-105pt);
    \draw (115pt,-65pt) -- (135pt,-105pt);
    \draw (115pt,-65pt) -- (95pt,-105pt);
    \draw (45pt,-65pt) -- (65pt,-105pt);
    \draw (45pt,-65pt) -- (25pt,-105pt);

    \draw[fill=white] (-115pt,-65pt) circle (10pt);
    \node[black] at (-115pt,-65pt) {\large \textcolor{black}{$4$}};
    \draw[fill=white] (-45pt,-65pt) circle (10pt);
    \node[black] at (-45pt,-65pt) {\large \textcolor{black}{$5$}};
    \draw[fill=white] (45pt,-65pt) circle (10pt);
    \node[black] at (45pt,-65pt) {\large \textcolor{black}{$6$}};
    \draw[fill=white] (115pt,-65pt) circle (10pt);
    \node[black] at (115pt,-65pt) {\large \textcolor{black}{$7$}};

    \draw[fill=white] (-135pt,-105pt) circle (10pt);
    \node[black] at (-135pt,-105pt) {\large \textcolor{black}{$8$}};
    \draw[fill=white] (-95pt,-105pt) circle (10pt);
    \node[black] at (-95pt,-105pt) {\large \textcolor{black}{$9$}};
    \draw[fill=white] (-65pt,-105pt) circle (10pt);
    \node[black] at (-65pt,-105pt) {\large \textcolor{black}{$10$}};
    \draw[fill=white] (-25pt,-105pt) circle (10pt);
    \node[black] at (-25pt,-105pt) {\large \textcolor{black}{$11$}};
    \draw[fill=white] (65pt,-105pt) circle (10pt);
    \node[black] at (65pt,-105pt) {\large \textcolor{black}{$13$}};
    \draw[fill=white] (25pt,-105pt) circle (10pt);
    \node[black] at (25pt,-105pt) {\large \textcolor{black}{$12$}};
    \draw[fill=white] (135pt,-105pt) circle (10pt);
    \node[black] at (135pt,-105pt) {\large \textcolor{black}{$15$}};
    \draw[fill=white] (95pt,-105pt) circle (10pt);
    \node[black] at (95pt,-105pt) {\large \textcolor{black}{$14$}};

  \end{tikzpicture} \\[1em]
  \caption{A tree with four levels, corresponding to the partitioning of the
  physical domain in Figure \ref{fig:boxes}.}
  \label{fig:tree}
\end{figure}
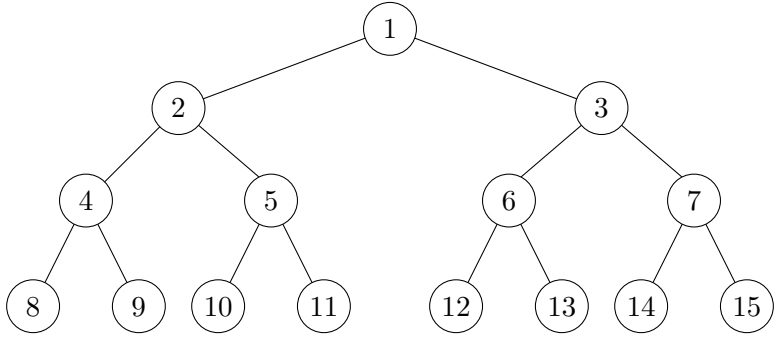

Given such a tree structure, we say that a matrix $\mtx{A}$ is an HBS
matrix if the following two properties hold:
\begin{enumerate}
  \item For any two sibling nodes (i.e. nodes with the same parent) in the
    tree $\alp$ and $\ba$, there exist matrices $\wh{\mtx{U}}_\alp \in
    \cx^{n_\alp \times k_\alp}$, $\wh{\mtx{V}}_\ba \in \cx^{n_\ba \times
    k_\ba}$, and $\mtx{A}_{\alp,\ba} \in \cx^{k_\alp \times k_\ba}$ such
    that
    \[
      \mtx{A}(\mc{I}_\alp,\mc{I}_\ba) = \wh{\mtx{U}}_\alp \mtx{A}_{\alp,\ba}
      \wh{\mtx{V}}_\ba^*,
    \]
  with $k_\alp < n_\alp = |\mc{I}_\alp|$ and  $k_\ba < n_\ba =
  |\mc{I}_\ba|$.
  \item For any parent node $\tau$ with children $\alp$ and $\ba$ there
    exist matrices $\mtx{U}_\tau, \mtx{V}_\tau \in \cx^{(k_\alp + k_\ba)
    \times k_\tau}$, such that
    \[
      \wh{\mtx{U}}_\tau = \begin{bmatrix} \wh{\mtx{U}}_\alp & \mtx{0} \\
      \mtx{0} & \wh{\mtx{U}}_\ba \end{bmatrix} \mtx{U}_\tau, \quad
      \wh{\mtx{V}}_\tau = \begin{bmatrix} \wh{\mtx{V}}_\alp & \mtx{0} \\
      \mtx{0} & \wh{\mtx{V}}_\ba \end{bmatrix} \mtx{V}_\tau.
    \]
    For notational convenience, we also define $\wh{\mtx{U}}_\tau = \mtx{U}_\tau$
    and $\wh{\mtx{V}}_\tau = \mtx{V}_\tau$ when $\tau$ is a leaf.
\end{enumerate}

The first condition imposes that $\mtx{A}$ can be tessellated in a fashion
with numerically low-rank off-diagonal blocks. This condition on its own is
the definition of the well-known Hierarchically Off-Diagonal Low-Rank
(HODLR) format \cite{ambikasaran2013fast}. The second condition
distinguishes the HBS format from the HODLR format by specifying that the
basis matrices can be taken to be nested. We refer to
\cite[Ch.~14]{martinsson2019book} for more details.

\begin{remark}
Vital to the efficacy of the HBS format is that the only matrices that need
to be stored are the small basis matrices $\mtx{U}_\tau$ and $\mtx{V}_\tau$
for each box $\tau$, the matrices $\mtx{A}(\mc{I}_\tau,\mc{I}_\tau)$ for
all leaves $\tau$, and the sibling-sibling interaction matrices
$\mtx{A}_{\alp,\ba}$ and $\mtx{A}_{\ba,\alp}$ for all sibling pairs $\alp$
and $\ba$. All necessary matrix algebra can then be done with these factors
alone, without explicitly forming large blocks of the full matrix or the
extended basis matrices $\wh{\mtx{U}}_\tau$ and $\wh{\mtx{V}}_\tau$.
\label{rmk:hbsfactors}
\end{remark}

\subsection{Computing an HBS representation}
\label{s:compression}

The coefficient matrix in \eqref{eq:linsys} can be well-approximated by an
HBS matrix with modest ranks if a tree structure is chosen based on a
division of physical space. We will compress only the matrix
$\mtx{G}$, as opposed to the full operator $\mtx{I} + \mtx{B} \mtx{G}$,
since this makes the compression independent of the scattering potential,
which is useful in a variety of applications and allows us to leverage
translation invariance of the kernel to accelerate the computation.

We let the root of our tree correspond to the entire domain $\Om$, and let
$\mc{I}_1 = [1,2,\hdots,N]$. We then bisect $\Om$ vertically in half and
associate the nodes 2 and 3 with the resulting left and right boxes,
respectively. The vectors $\mc{I}_2$ and $\mc{I}_3$ will contain indices
that correspond to points in the left and right half of the domain,
respectively. We now split the box corresponding to the node 2 in half
horizontally, and associate the upper box with node 4 and lower box with
node 5, populating the index vectors $\mc{I}_4$ and $\mc{I}_5$ with the
indices of the points $\pxx_i$ that lie in their corresponding box.
Similarly, boxes for nodes 6 and 7 are obtained by the analogous splitting
of the box for node 3. We continue this procedure, alternating between
vertical and horizontal cuts at each level until all undivided boxes
contain only a few of the $\pxx_i$. We will slightly abuse notation and
refer to a node $\tau$ in the tree and its corresponding box
$\Omega_{\tau}$ interchangeably. Figure \ref{fig:boxes} illustrates the
partition of physical space that corresponds with the tree in Figure
\ref{fig:tree}.

\begin{figure}
  \centering
  \scalebox{1.2}{
\begin{tikzpicture}
  \draw[black,thick] (0pt,0pt) rectangle (64pt,64pt);

  \draw[black,thick] (25pt+64pt,0pt) rectangle (25pt+1.5*64pt,64pt);
  \draw[black,thick] (25pt+1.5*64pt,0pt) rectangle (25pt+2*64pt,64pt);

  \node[black] at (25pt+8pt,32pt) {\large 1};

  \draw[black,thick] (2*25pt+2*64pt,0pt) rectangle (2*25pt+2.5*64pt,.5*64pt);
  \draw[black,thick] (2*25pt+2*64pt,.5*64pt) rectangle (2*25pt+2.5*64pt,64pt);
  \draw[black,thick] (2*25pt+2.5*64pt,0pt) rectangle (2*25pt+3*64pt,.5*64pt);
  \draw[black,thick] (2*25pt+2.5*64pt,.5*64pt) rectangle (2*25pt+3*64pt,64pt);

  \node[black] at (2*25pt+46pt+8pt,32pt) {\large 2};
  \node[black] at (2*25pt+77pt+8pt,32pt) {\large 3};

  \draw[black,thick] (3*25pt+3*64pt,0pt) rectangle (3*25pt+3.25*64pt,.5*64pt);
  \draw[black,thick] (3*25pt+3.25*64pt,0pt) rectangle (3*25pt+3.5*64pt,.5*64pt);
  \draw[black,thick] (3*25pt+3.5*64pt,0pt) rectangle (3*25pt+3.75*64pt,.5*64pt);
  \draw[black,thick] (3*25pt+3.75*64pt,0pt) rectangle (3*25pt+4*64pt,.5*64pt);

  \node[black] at (3*25pt+111pt+8pt,32pt+15pt) {\large 4};
  \node[black] at (3*25pt+142pt+8pt,32pt+15pt) {\large 6};
  \node[black] at (3*25pt+111pt+8pt,32pt-15pt) {\large 5};
  \node[black] at (3*25pt+142pt+8pt,32pt-15) {\large 7};

  \draw[black,thick] (3*25pt+3*64pt,.5*64pt) rectangle (3*25pt+3.25*64pt,1*64pt);
  \draw[black,thick] (3*25pt+3.25*64pt,.5*64pt) rectangle (3*25pt+3.5*64pt,1*64pt);
  \draw[black,thick] (3*25pt+3.5*64pt,.5*64pt) rectangle (3*25pt+3.75*64pt,1*64pt);
  \draw[black,thick] (3*25pt+3.75*64pt,.5*64pt) rectangle (3*25pt+4*64pt,1*64pt);
  \node[black] at (4*25pt+168pt+8pt-1pt,32pt+15pt) {\large 8};
  \node[black] at (4*25pt+168pt+8pt+15pt-1pt,32pt+15pt) {\large 9};
  \node[black] at (4*25pt+168pt+8pt-2pt,32pt+15pt-32pt) {\large 10};
  \node[black] at (4*25pt+168pt+8pt+15pt,32pt+15pt-32pt) {\large 11};
  \node[black] at (4*25pt+168pt+8pt+32pt-2pt,32pt+15pt) {\large 12};
  \node[black] at (4*25pt+168pt+8pt+15pt+32pt-1pt,32pt+15pt) {\large 13};
  \node[black] at (4*25pt+168pt+8pt-2pt+32pt,32pt+15pt-32pt) {\large 14};
  \node[black] at (4*25pt+168pt+8pt+15pt+32pt,32pt+15pt-32pt) {\large 15};
\end{tikzpicture}
}
\caption{Division of $\Om$ into boxes corresponding to the tree in Figure
\ref{fig:tree}.}
\label{fig:boxes}
\end{figure}
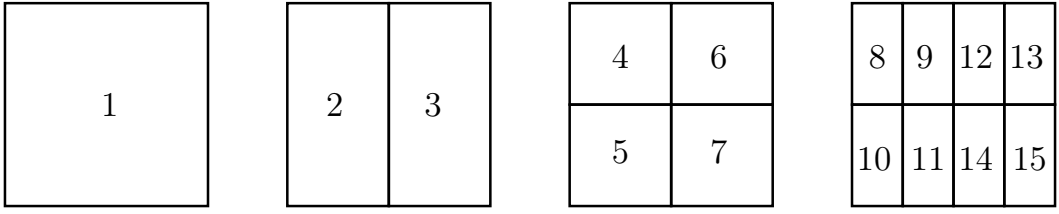

For such a tree and a compression tolerance $\eps \in (0,1)$, we expect there to
be an HBS matrix $\mtx{G}^\eps$, such that
\begin{equation}
  \frac{\|\mtx{G} - \mtx{G}^{\eps}\|_2}{\|\mtx{G}\|_2} \leq \eps,
  \label{eq:matapprox}
\end{equation}
where the ranks $k_\alp$ and $k_\ba$ at each level in the definition of the
HBS matrix increase only gradually as $\eps \to 0$. After, the HBS matrix
$\mtx{G}^\eps$ is obtained, it can be inverted efficiently using the
algorithm in Section \ref{s:solver}. In the remainder of this section we
discuss how to efficiently compute such a $\mtx{G}^\eps$.

For the reasons mentioned in Section \ref{s:id}, we use the ID for low-rank
factorization. Let $\tau$ be any leaf node. We will obtain the factor
$\mtx{U}_\tau$ from a row ID of the form
\begin{equation}
  \mtx{G}(\mc{I}_\tau,\mc{I}_\tau^c) \approx \mtx{U}_\tau
  \mtx{G}(\mc{I}_\tau^{\rm skel},\mc{I}_\tau^c).
  \label{eq:leafid}
\end{equation}
Since $|\mc{I}_\tau^c| = \mc{O}(N)$, directly applying the CPQR algorithm
would be expensive for large $N$, and ultimately yield an $\mc{O}(N^2)$
algorithm. Fortunately, this cost can be reduced
using the method of proxy sources \cite{cheng2005compression}. The idea
here is that the submatrix $\mtx{G}(\mc{I}_\tau,\mc{I}_\tau^c)$ can be
interpreted as encoding the interactions induced by the kernel
\eqref{eq:kernel} of charges located at $\{ \pxx_i \}_{i \in
\mc{I}_\tau^c}$ on the target points $\{ \pxx_i \}_{i \in
\mc{I}_\tau}$. The proxy source technique is based on a result from
potential theory that states that every solution to the Helmholtz
equation generated by sources outside of the box $\Omega_{\tau}$
can be exactly replicated by layer potentials on the boundary
of $\Omega_{\tau}$ \cite[Sec.~17.2]{martinsson2019book}. Numerically,
we approximate the effect of a continuum layer potential on
$\partial\Omega_{\tau}$ by placing sources on a thin ``ring'' of
collocation points that are just outside of $\Omega_{\tau}$, as
shown in Figure \ref{fig:proxy}(b), where the target points in
$\mc{I}_{\tau}$ are drawn in blue and the proxy points in
$\mc{I}_{\tau}^{\rm proxy}$ are drawn in red.
In other words, any field on the box $\tau$ induced by charges at
$\{\pxx_i \}_{i \in \mc{I}_\tau^c}$ can be well-approximated by
charges places at the proxy points
$\{\pxx_i \}_{i \in \mc{I}_{\tau}^{\rm proxy}}$. Expressed in the language of
linear algebra, our claim is that the relation
\begin{equation}
  \operatorname{colspan}(\mtx{G}(\mc{I}_\tau,\mc{I}_\tau^c))
  \subseteq
  \operatorname{colspan}(\mtx{G}(\mc{I}_\tau,\mc{I}^{\rm
  proxy}_{\tau}))
  \label{eq:proxy}
\end{equation}
holds to high accuracy.
Thus, if we compute $\mtx{U}_\tau$ and an index vector $\mc{I}_\tau^{\rm
skel}$ such that $\mtx{G}(\mc{I}_\tau,\mc{I}^{\rm proxy}_{\tau}) \approx
\mtx{U}_\tau \mtx{G}(\mc{I}_\tau^{\rm skel},\mc{I}_{\tau}^{\rm proxy})$,
then we expect that the same $\mtx{U}_\tau$ and $\mc{I}_\tau^{\rm skel}$
satisfy \eqref{eq:leafid} up to whatever accuracy \eqref{eq:proxy} holds to.

The accuracy of the proxy source technique depends on how ``thick'' one
takes the ring of proxy sources to be. Using just a single wide line of
sources, as shown in Figure \ref{fig:proxy}(b), typically leads to 4 or 5 accurate digits,
while using a double wide ring, as shown in Figure \ref{fig:proxy}(c) is almost always
enough to attain 10 accurate digits. If a triple wide ring is used, full
double precision accuracy typically results. See Table \ref{table:proxy} for details.

By the symmetry of $\mtx{G}$, we can set $\mtx{V}_\tau = \mtx{U}_\tau$ and
use the column skeleton indices $\mc{I}_\tau^{\rm skel}$ also for the row
skeletons.

\begin{figure}
\begin{center}
\begin{tabular}{ccc}
\includegraphics[width=50mm]{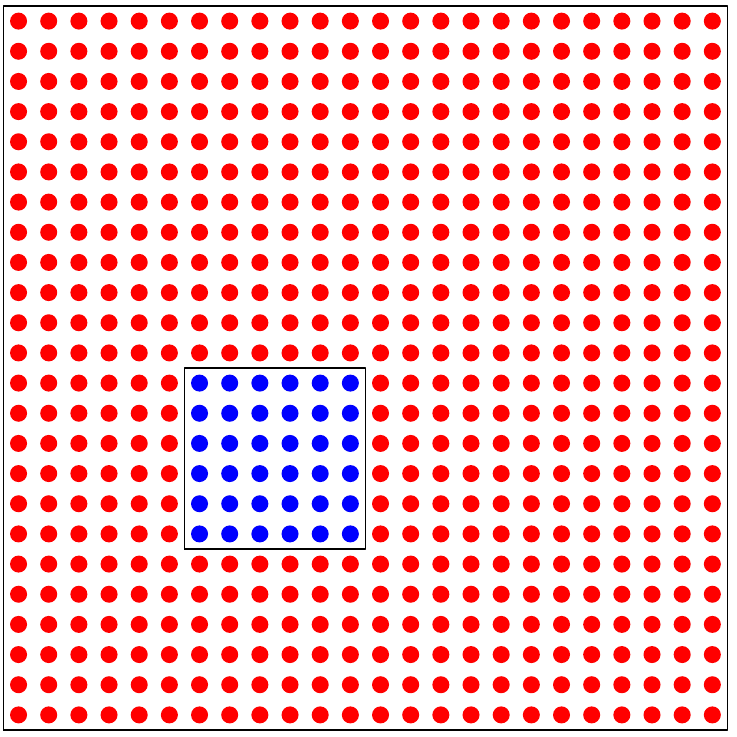} &
\includegraphics[width=50mm]{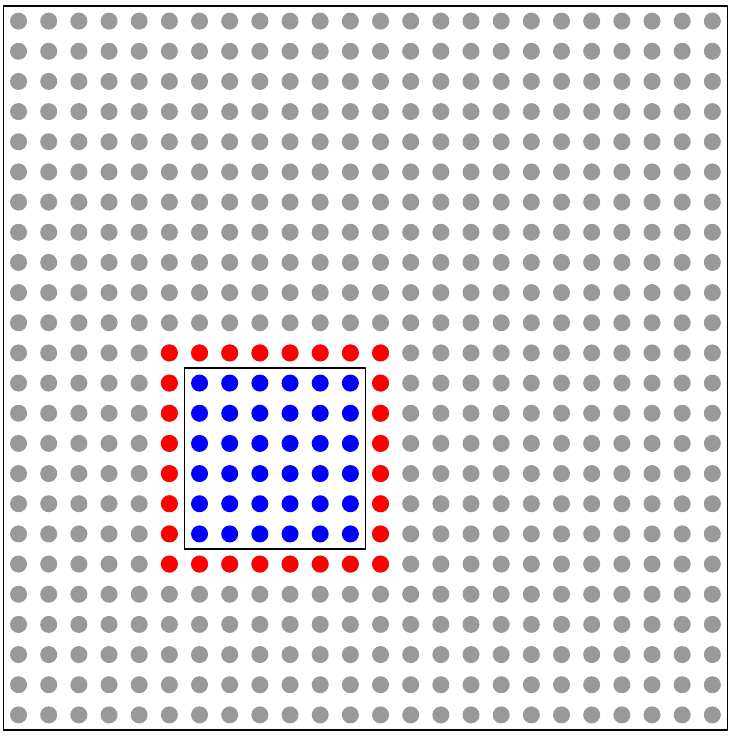} &
\includegraphics[width=50mm]{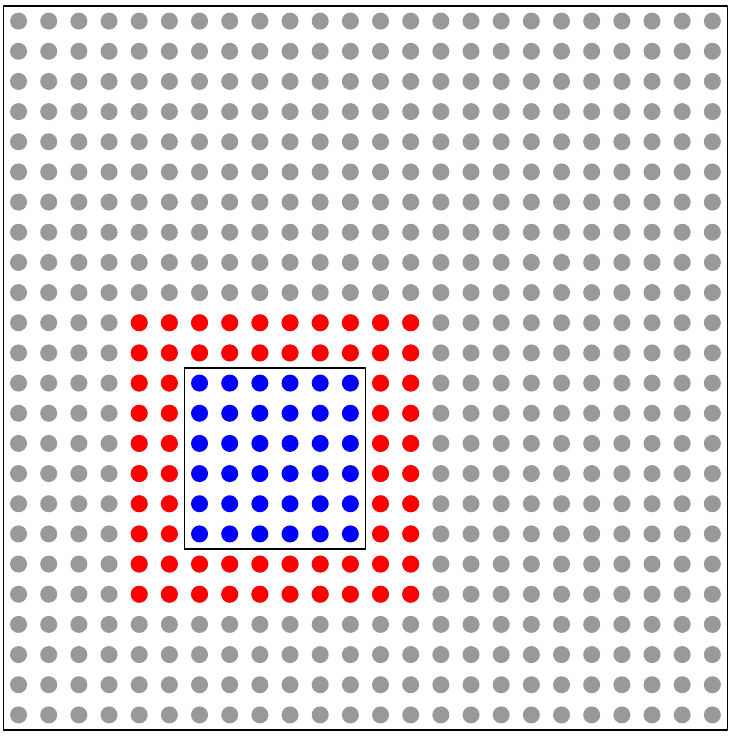} \\
(a) & (b) & (c)
\end{tabular}
\end{center}
\caption{The proxy source compression technique, cf.~Section \ref{s:compression}.
(a) We seek to build a basis for the range of the
matrix $\mtx{G}(\mc{I}_\tau,\mc{I}_\tau^c)$ that maps sources on the red
points ($\mc{I}_{\tau}^{\rm c}$) to targets on the blue points ($\mc{I}_{\tau}$).
This is expensive since $\mtx{G}(\mc{I}_\tau,\mc{I}_\tau^c)$ has $O(N)$ columns.
(b) The proxy source techniques relies on the observation that any
solution to the homogeneous Helmholtz equation on $\Omega_{\tau}$
(the box holding the blue dots) can be approximated to high precision
by placing sources \textit{on the red ``proxy points'' ($\mc{I}_{\tau}^{\rm proxy}$) only}. This means that
we can work with the small matrix $\mtx{G}(\mc{I}_\tau,\mc{I}_\tau^{\rm proxy})$.
(c) Using a thicker ring of proxy sources, higher accuracy can be attained.
}
\label{fig:proxy}
\end{figure}

\begin{table}
\begin{center}
\begin{tabular}{l|r|r|r}
& \multicolumn{3}{c}{Size of $\Omega_{\tau}$} \\
 & $0.25\lambda\times 0.25\lambda$ & $1\lambda \times 1\lambda$ & $4\lambda \times 4\lambda$ \\ \hline
width=1 &  8.7e-05  &    1.6e-04  &    9.6e-04 \\
width=2 &  1.1e-10  &    1.8e-10  &    6.2e-10  \\
width=3 &  5.6e-15  &    5.2e-15  &    4.0e-15
\end{tabular}
\end{center}
\caption{Errors incurred when using the method of proxy sources
to accelerate the compression stage. The target box $\Omega_{\tau}$
holds $20 \times 20$ points. The error reported is the relative
error in maximum norm. To be precise, using the notation in Section
\ref{s:compression}, the error reported is $\|\mtx{B} - \mtx{Q}\mtx{Q}^{*}\mtx{B}\|_{\infty}/
\|\mtx{B}\|_{\infty}$ where $\mtx{B} = \mtx{G}(\mc{I}_\tau,\mc{I}_\tau^c)$ is the
matrix to be approximated, where $\mtx{Q}$ is a matrix whose columns form
an orthornormal basis for the columns of
$\mtx{G}(\mc{I}_\tau,\mc{I}^{\rm proxy}_{\tau})$, and where
$\|\mtx{B}\|_{\infty} = \max|\mtx{B}(i,j)|$. }
\label{table:proxy}
\end{table}

Another advantage of the proxy source technique is that the matrix
$\mtx{G}(\mc{I}_{\tau},\mc{I}_{\tau}^{\rm proxy})$ that needs to get
factorized is \textit{identical} for every node $\tau$ on a given level.
This means that only one box per level needs to actually be processed.

\begin{remark}
A slight disadvantage of the proxy source technique is that it slightly overestimates
the interaction rank for boxes that are located at the boundary of the domain.
To see why, consider a box in the southwest corner of the computational domain.
The optimal set of proxy nodes need only capture a field generated by
sources to the north and to the east. But the proxy sources will completely
surround the box nevertheless, and will force the inclusion of skeleton nodes
along the outer boundaries.
\end{remark}

\begin{remark}
The compression stage is in a sense a ``universal'' computation, since
it depends only on the size of a box, and on the parameter $h\kappa$.
This means that if one is willing to live with mild restrictions on
the grids that can be used, the bases and the skeleton sets could be
precomputed once and for all, and be stored.
\end{remark}

After the relevant factors have been computed on the leaf level, the nodes
on level $L-1$ can be processed, picking the skeleton indices for each box
to be a subset of the skeleton indices of its children. If $\tau$ is now a
parent node with children $\alp$ and $\ba$, this can be accomplished by
computing a row ID of the form
\begin{equation}
  \mtx{G}([\mc{I}_\alp^{\rm skel},\mc{I}_\ba^{\rm skel}],\mc{I}_\tau^c)
  \approx
  \mtx{U}_\tau \mtx{G}(\mc{I}_\tau^{\rm skel}, \mc{I}_\tau^c).
  \label{eq:parentid}
\end{equation}
Again, the method of proxy sources is employed to reduce the cost of
computing this factorization. The sibling-sibling interaction matrices are
also simply given by
\begin{equation}
  \mtx{G}_{\alp,\ba} = \mtx{G}(\mc{I}_\alp^{\rm skel}, \mc{I}_\ba^{\rm  skel})
  \qquad  \text{and} \qquad
  \mtx{G}_{\ba,\alp} = \mtx{G}(\mc{I}_\ba^{\rm skel},  \mc{I}_\alp^{\rm skel}),
    \label{eq:sibint}
\end{equation}
so these dense matrices themselves do not need to be stored. Translation
invariance is again exploited so it suffices to factor just a single box on
this level. The equations \eqref{eq:parentid} and \eqref{eq:sibint} can be
viewed as a recursion relation, which can be applied to move up the tree,
processing one level at a time, employing the method of proxy sources and
translation invariance on each level.

This procedure carries over straightforwardly when higher-order quadrature
corrections are included. However, the required ranks in each ID and the
necessary thickness of the ring of proxy sources for a fixed accuracy
increases. The former effect is illustrated in Figures \ref{fig:4thid} and
\ref{fig:10thid}, which show the skeleton points and ranks of the ID when
the interaction through the kernel \eqref{eq:kernel} with $\ka = 10$ and
$4^{\rm th}$- and $10^{\rm th}$-order quadrature respectively between a box
of 20 by 20 points in a 40 by 40 uniform lattice of points is compressed to
various accuracies.

\begin{figure}
  \centering
  \includegraphics[scale=0.35]{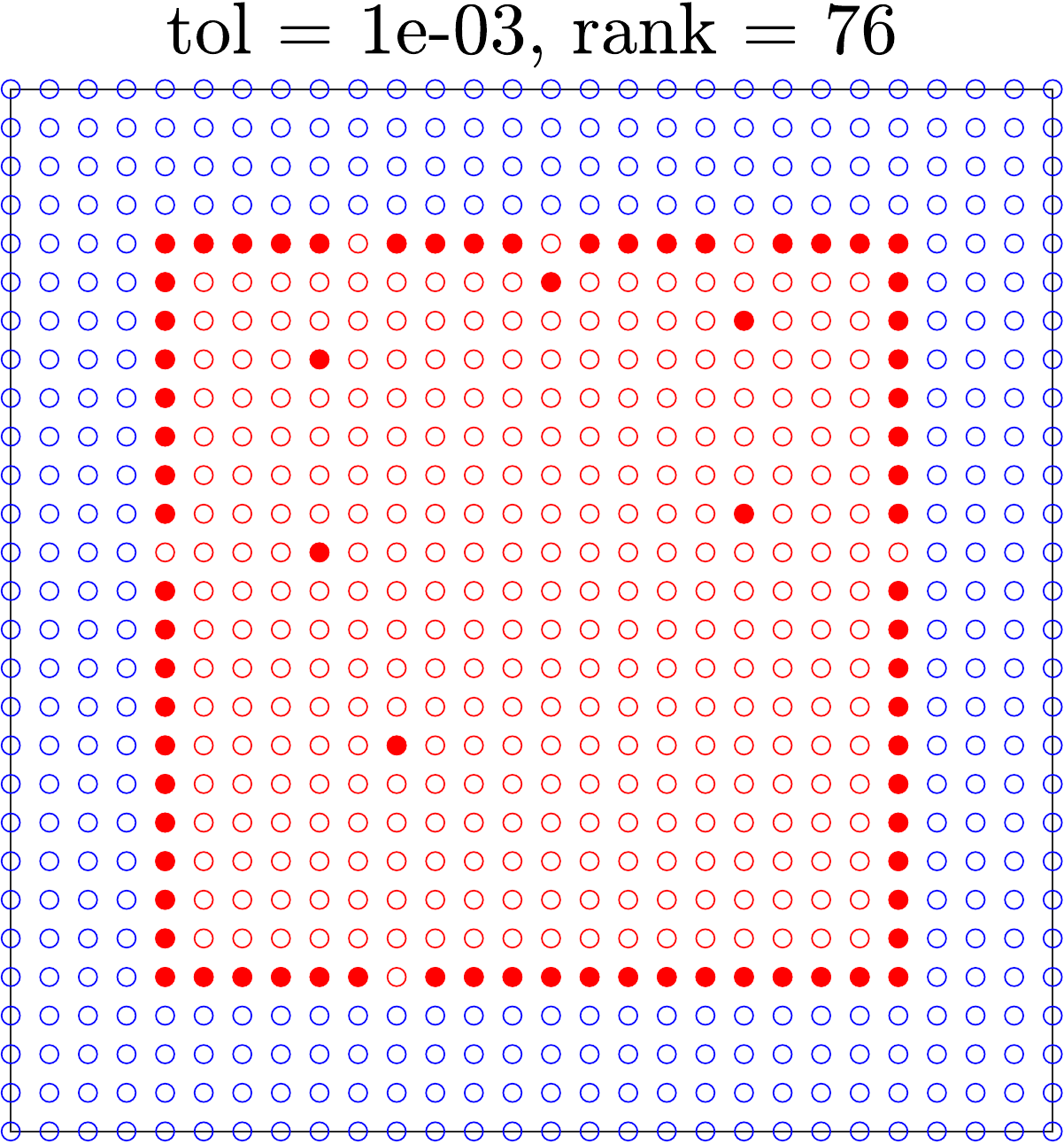} \hspace{1.5em}
  \includegraphics[scale=0.35]{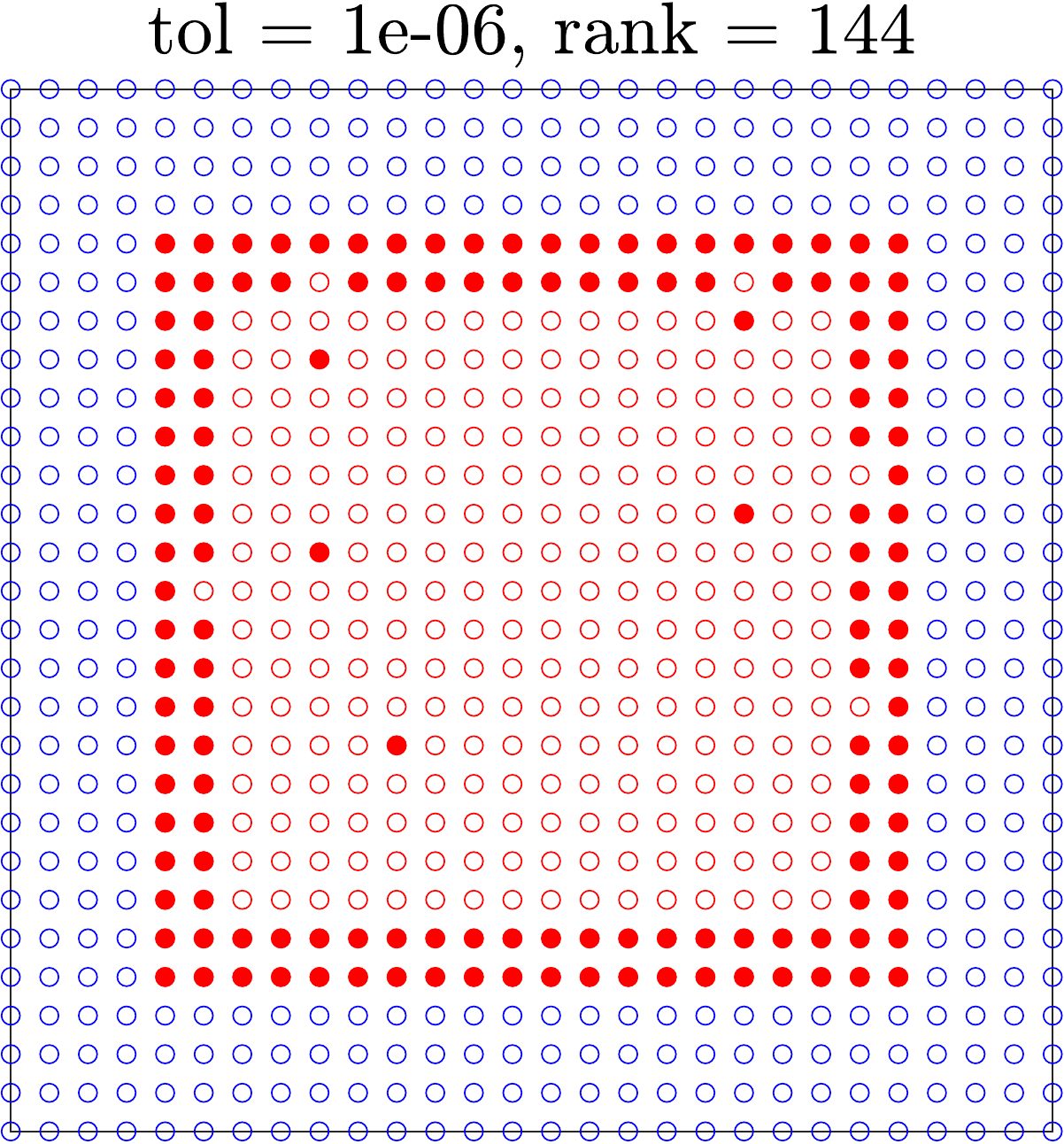} \hspace{1.5em}
  \includegraphics[scale=0.35]{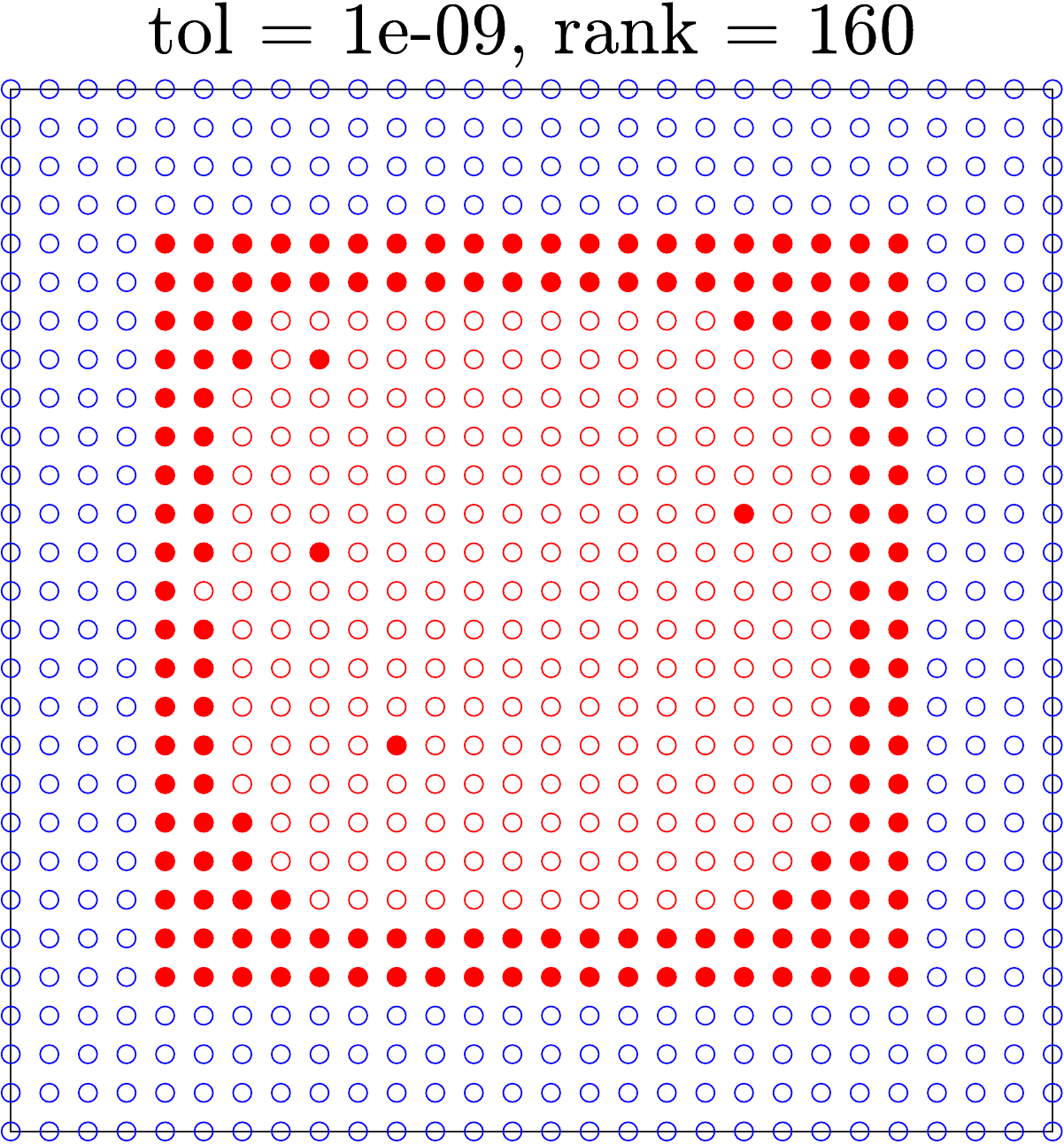}
  \caption{Illustration of ID with $4^{\rm th}$-order corrections described
    in Section \ref{s:compression}. The red circles denote points in the
    box $\tau$ and the blue points denote the points in the proxy ring
    used.  The solid red dots denote the points that correspond to skeleton
    indices, as chosen using CPQR as explained in Section \ref{s:id}. Each
    image corresponds to a specific compression tolerance, which is
    specified in the title, and also reported is the rank (i.e.  the number
  of solid red circles) for that tolerance.}
  \label{fig:4thid}
\end{figure}

\begin{figure}
  \centering
  \includegraphics[scale=0.35]{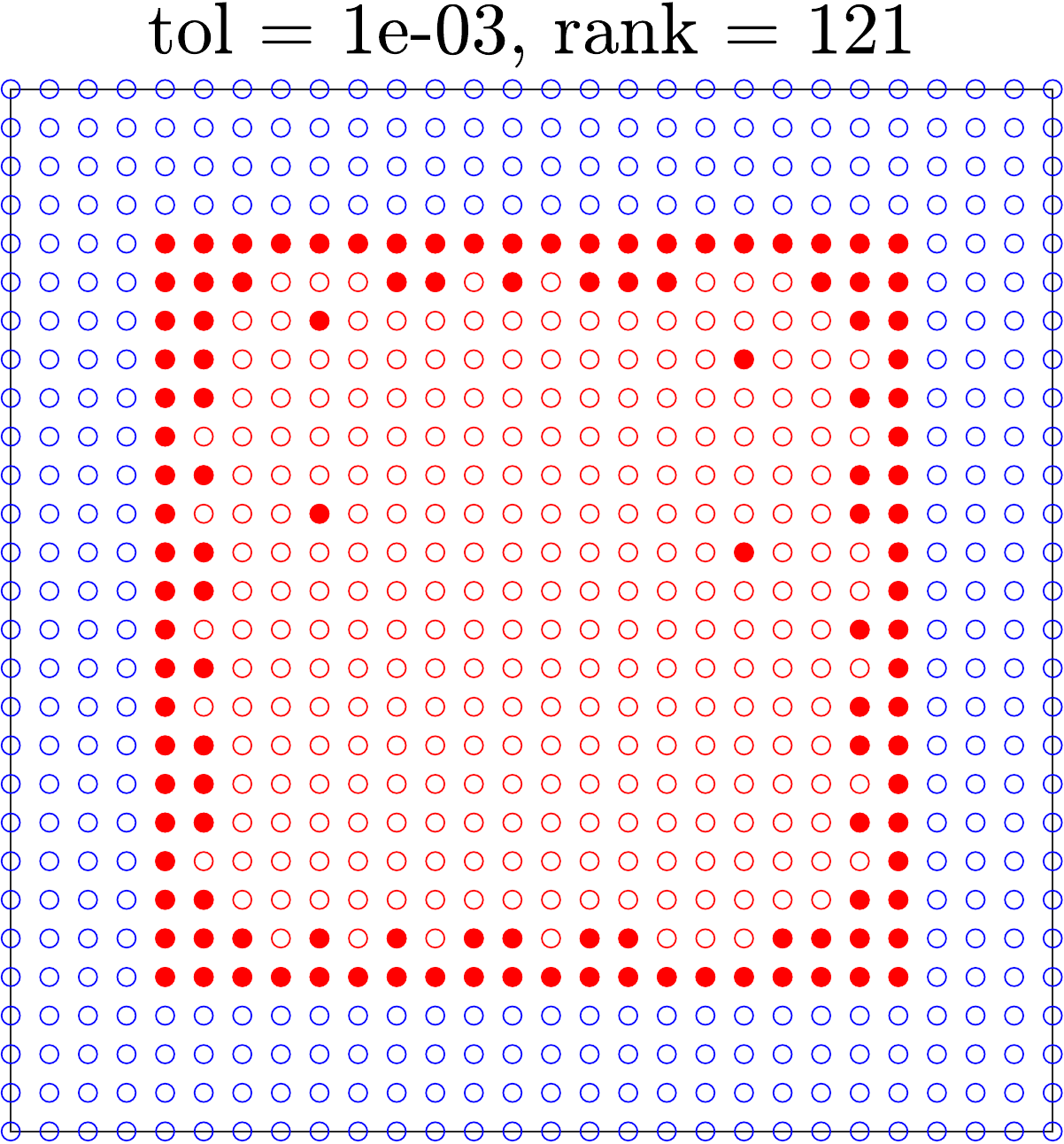} \hspace{1.5em}
  \includegraphics[scale=0.35]{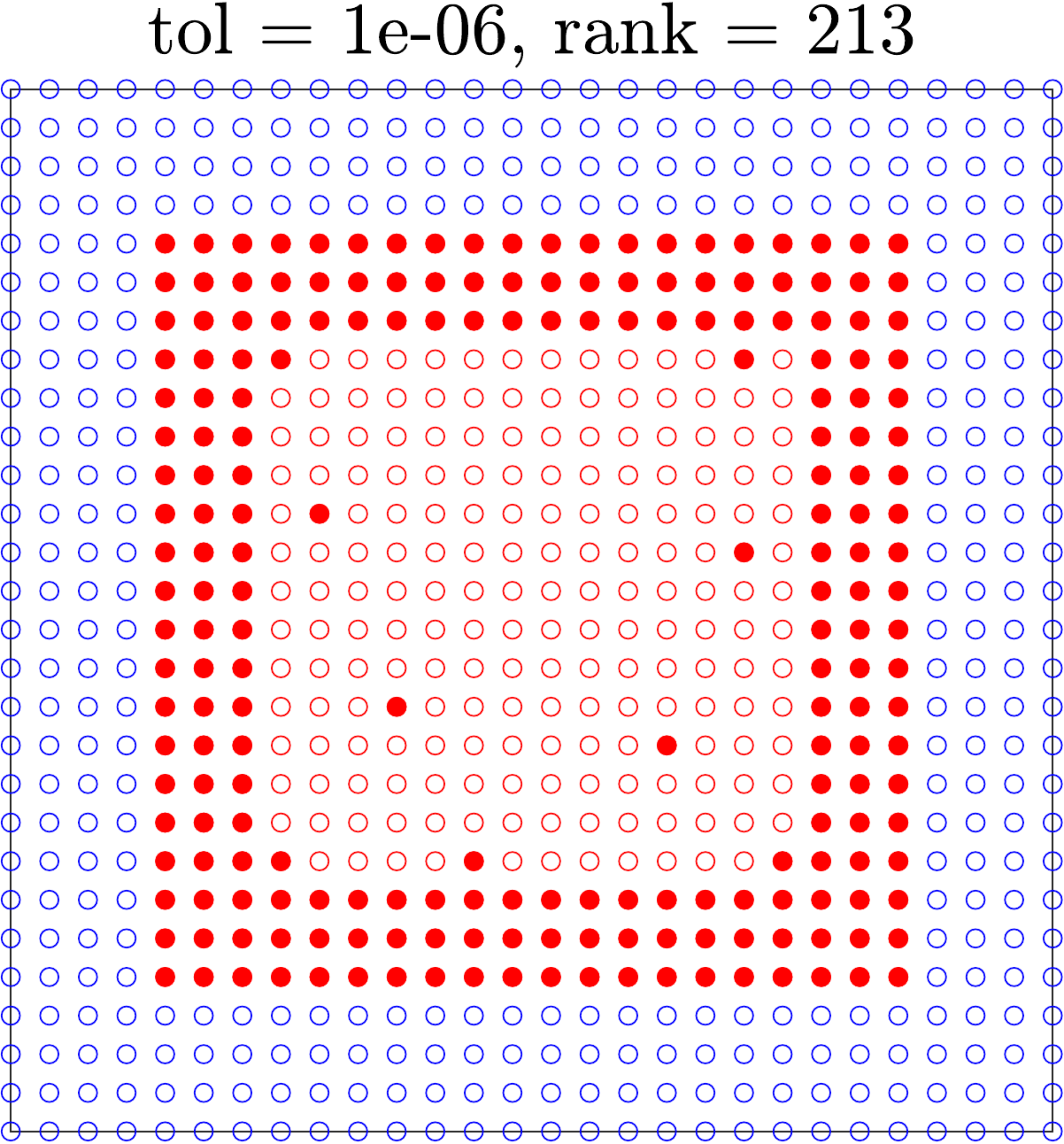} \hspace{1.5em}
  \includegraphics[scale=0.35]{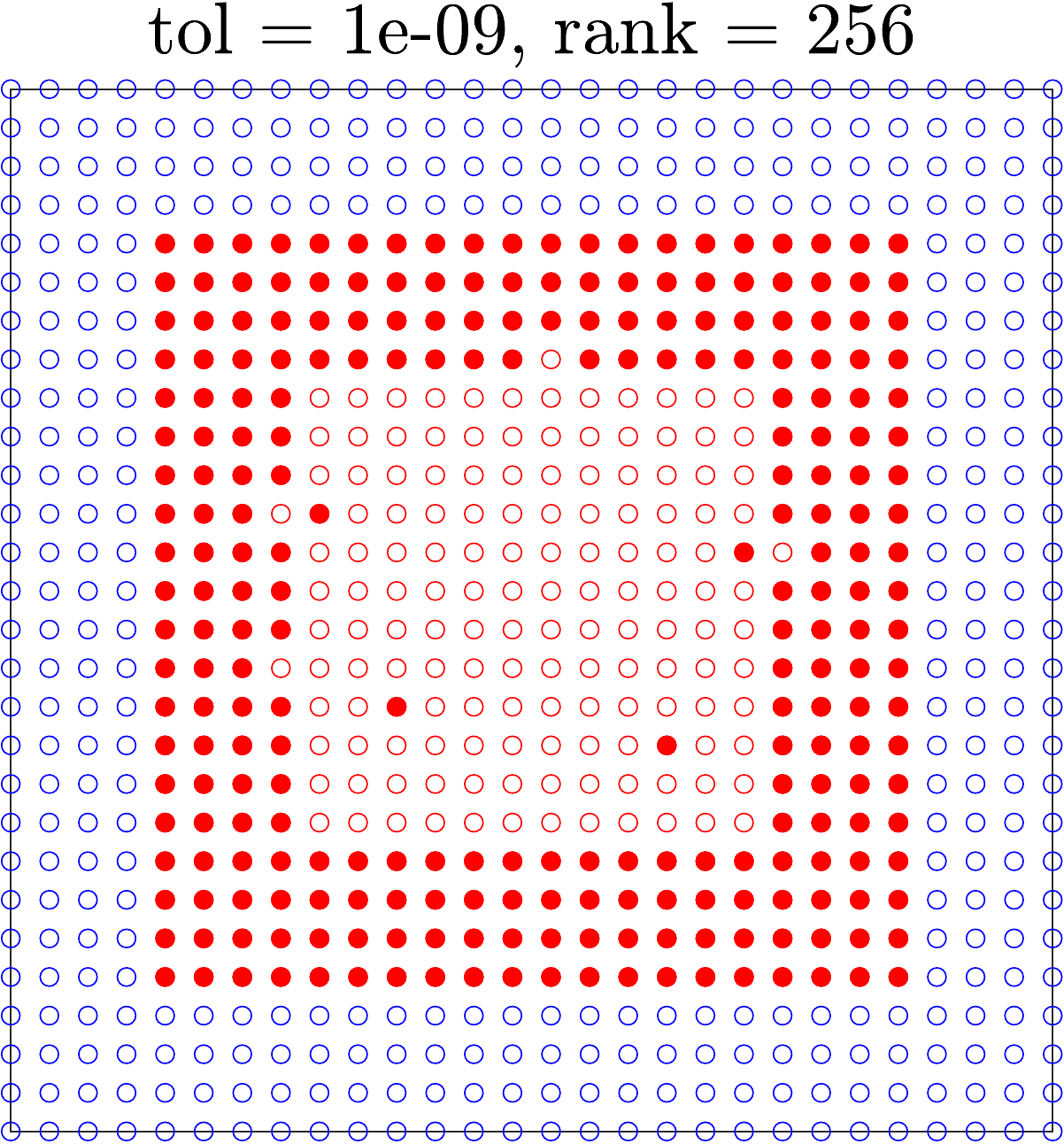}
  \caption{Analogue of Figure \ref{fig:4thid} for the $10^{\rm th}$-order
  quadrature corrections.}
  \label{fig:10thid}
\end{figure}

\begin{remark}
\label{remark:howpicklocaltolerance}
A delicate question in building a rank-structured representation of a matrix
is how to choose the local compression tolerance. Since errors do to some degree
aggregate in the pass from smaller to larger blocks, it is advantageous to
use a smaller (absolute) tolerance for boxes on the finest level, and then to
gradually increase the tolerance parameter in the upwards pass. Deriving rigorous
bounds on the error propagation is difficult, in particular when non-orthonormal
basis matrices are used. Moreover, such bounds tend to be pessimistic since
the worst-case alignment of the all error vectors is statistically extremely unlikely.
In the numerical experiments reported in Section \ref{s:numerics}, we picked a local
compression tolerance by multiplying a given global tolerance parameter $\eps$ by
(an estimate of) the Frobenius norm of the block to be compressed. This simple strategy
automatically enforces more precise approximation of blocks at the finer levels, and
has empirically proven to work well.
\end{remark}

\section{A direct solver based on discrete scattering operators}
\label{s:direct}

We are now ready to describe the inversion procedure for \eqref{eq:linsys}.
In the following, we will assume that an HBS approximation to the matrix
$\mtx{G}$ has been computed using the procedure described in Section
\ref{s:compression} and therefore all necessary matrices, as described in
Remark \ref{rmk:hbsfactors}, have been computed.  In Section
\ref{s:solver}, we give a formal description of our solver.  This is
followed by Section \ref{s:complexity} which gives a brief outline of the
solver's complexity. Section \ref{s:comparison} provides some discussion
contextualizing our work in relation to some existing direct solvers.

\subsection{Description of solver}
\label{s:solver}
For each node $\tau$, we introduce the notation
\begin{equation}
  \vct{q}_\tau = \vct{q}(\mc{I}_\tau), \quad \vct{w}_\tau =
  \mtx{G}(\mc{I}_\tau,\mc{I}_\tau^c) \vct{q}(\mc{I}_\tau^c), \quad
  \vct{f}_\tau = \vct{f}(\mc{I}_\tau).
  \label{eq:notation}
\end{equation}
The vector $\vct{q}_\tau$ represents the restriction of $\vct{q}$ to the
points in $\tau$ and $\vct{w}_\tau$ represents the incoming field on the
points in $\tau$ generated from all of the points outside of $\tau$.

The key observation that enables our direct solver is that to represent the
interactions between nodes, it suffices to work with compressed versions of
these vectors.  To this end, we define an \textit{outgoing expansion}
$\wt{\vct{q}}_{\tau}$ and an \textit{incoming expansion}
$\wt{\vct{w}}_{\tau}$ through the relations
\begin{equation}
  \wt{\vct{q}}_\tau = \wh{\mtx{V}}_\tau^* \vct{q}_\tau, \quad \vct{w}_\tau
  = \wh{\mtx{U}}_\tau \wt{\vct{w}}_\tau,
  \label{eq:expansions}
\end{equation}
respectively, where $\wh{\mtx{U}}_\tau$ and $\wh{\mtx{V}}_\tau$ are the
extended basis matrices defined in Section \ref{s:hbs}. The expansions
$\wt{\vct{q}}_\tau$ and $\wt{\vct{w}}_\tau$ serve as proxies for
$\vct{q}_\tau$ and $\vct{w}_\tau$, and are analogous to the multipole and
local expansions in the fast multipole method \cite{greengard1987fast}. The
key to relating the two expansions is the following lemma, which states
that the incoming expansion of a non-root node can be split into a
sibling-sibling interaction plus a contribution from the incoming field of
the parent:

\begin{lemma}
  Let $\tau$ be a parent node with children $\alp$ and $\ba$. Then,
  with $\mtx{G}_{\alpha,\beta}$ and $\mtx{G}_{\beta,\alpha}$ defined by
  \eqref{eq:sibint}, we have
  \begin{equation}
    \begin{bmatrix}
      \wt{\vct{w}}_\alp \\ \wt{\vct{w}}_\ba
    \end{bmatrix}
    =
    \begin{bmatrix}
      \mtx{0} & \mtx{G}_{\alp,\ba} \\ \mtx{G}_{\ba,\alp} & \mtx{0}
    \end{bmatrix}
    \begin{bmatrix} \wt{\vct{q}}_\alp \\ \wt{\vct{q}}_\ba \end{bmatrix}
    + \mtx{U}_\tau \wt{\vct{w}}_\tau,
    \label{eq:lemma}
  \end{equation}
  whenever $\wh{\mtx{U}}_\alp$ and $\wh{\mtx{U}}_\ba$ have full column rank.
  \label{thm:lemma}
\end{lemma}
\begin{proof}
  By definition of the extended basis matrices, the full incoming field on
  $\alp$ and $\ba$ can be expressed by
  \begin{align*}
    \begin{bmatrix}
      \vct{w}_\alp \\ \vct{w}_\ba
    \end{bmatrix}
    &=
    \begin{bmatrix}
      \mtx{0} & \mtx{G}(\mc{I}_\alp,\mc{I}_\ba) \\
      \mtx{G}(\mc{I}_\ba,\mc{I}_\alp) & \mtx{0}
    \end{bmatrix}
    \begin{bmatrix}
      \vct{q}_\alp \\ \vct{q}_\ba
    \end{bmatrix}
    + \wh{\mtx{U}}_\tau \vct{w}_\tau \\
    &=
    \begin{bmatrix}
      \wh{\mtx{U}}_\alp & \mtx{0} \\ \mtx{0} & \wh{\mtx{U}}_\ba
    \end{bmatrix}
    \begin{bmatrix}
      \mtx{0} & \mtx{G}_{\alp,\ba} \\
      \mtx{G}_{\ba,\alp} & \mtx{0}
    \end{bmatrix}
    \begin{bmatrix}
      \wh{\mtx{V}}_\alp^* & \mtx{0} \\
      \mtx{0} & \wh{\mtx{V}}_\ba^*
    \end{bmatrix}
    \begin{bmatrix}
      \vct{q}_\alp \\
      \vct{q}_\ba
    \end{bmatrix}
    +
    \begin{bmatrix}
      \wh{\mtx{U}}_\alp & \mtx{0} \\ \mtx{0} & \wh{\mtx{U}}_\ba
    \end{bmatrix}
    \mtx{U}_\tau
    \wt{\vct{w}}_\tau \\
    &= \begin{bmatrix}
      \wh{\mtx{U}}_\alp & \mtx{0} \\ \mtx{0} &
      \wh{\mtx{U}}_\ba
    \end{bmatrix}
      \left(
        \begin{bmatrix}
          \mtx{0} & \mtx{G}_{\alp,\ba} \\ \mtx{G}_{\ba,\alp} & \mtx{0}
        \end{bmatrix}
        \begin{bmatrix}
          \wt{\vct{q}}_\alp \\ \wt{\vct{q}}_\ba
        \end{bmatrix}
        + \mtx{U}_\tau \wt{\vct{w}}_\tau
      \right).
  \end{align*}
  The result then follows from the fact that $\wh{\mtx{U}}_\alp$ and
  $\wh{\mtx{U}}_\ba$ have full column rank and by \eqref{eq:expansions}
 \[
    \begin{bmatrix}
      \wt{\vct{w}}_\alp \\ \wt{\vct{w}}_\ba
    \end{bmatrix}
    =
    \begin{bmatrix}
      \wh{\mtx{U}}_\alp^{\dagger} & \mtx{0} \\ \mtx{0} &
      \wh{\mtx{U}}_\ba^\dagger
    \end{bmatrix}
    \begin{bmatrix} \vct{w}_\alp \\ \vct{w}_\ba \end{bmatrix}.
 \]
\end{proof}

We are now in position to derive the solver, which follows the pattern of
the solver in \cite[Ch.~18]{martinsson2019book}, but with some
accelerations that are enabled by the particular nature of the equilibrium
equation \eqref{eq:linsys}.

The first stage consist of a sweep upwards through the hierarchical
tree, going from smaller to larger boxes. To start, let $\tau$ denote
a leaf box. Recalling the definition of local potentials
in \eqref{eq:notation}, and restricting the equilibrium equation \eqref{eq:linsys}
to $\tau$, we obtain the local equilibrium equation
\begin{equation}
  (\mtx{I} + \mtx{B}_\tau \mtx{G}_\tau) \vct{q}_\tau + \mtx{B}_\tau \vct{w}_\tau =
  \vct{f}_\tau,
  \label{eq:equi}
\end{equation}
where $\mtx{B}_\tau = \mtx{B}(\mc{I}_\tau,\mc{I}_\tau)$ and $\mtx{G}_\tau =
\mtx{G}(\mc{I}_\tau,\mc{I}_\tau)$. Setting
\begin{equation}
\mtx{X}_\tau = (\mtx{I} +
\mtx{B}_\tau \mtx{G}_\tau)^{-1}
\label{eq:localsol_leaf}
\end{equation}
and then multiplying both sides of
\eqref{eq:equi} by $\mtx{V}_\tau^* \mtx{X}_\tau$ produces
\begin{equation}
  \mtx{V}_\tau^* \vct{q}_\tau + \mtx{V}_\tau^* \mtx{X}_\tau
  \mtx{B}_\tau\vct{w}_\tau = \mtx{V}_\tau^* \mtx{X}_\tau \vct{f}_\tau.
  \label{eq:comp}
\end{equation}
Using the fact that by definition $\wh{\mtx{V}}_\tau = \mtx{V}_\tau$ for
leaves, \eqref{eq:comp} can be rewritten using the notation in
\eqref{eq:expansions} as
\begin{equation}
  \wt{\vct{q}}_\tau + \mtx{S}_\tau \wt{\vct{w}}_\tau = \wt{\vct{r}}_\tau,
  \label{eq:leaf}
\end{equation}
where
\begin{equation}
\mtx{S}_\tau = \mtx{V}_\tau^* \mtx{X}_\tau \mtx{B}_\tau \mtx{U}_\tau
\label{eq:smat_leaf}
\end{equation}
and
\begin{equation}
\wt{\vct{r}}_\tau = \mtx{V}_\tau^* \mtx{X}_\tau \vct{f}_\tau.
\label{eq:load_leaf}
\end{equation}
The matrix $\mtx{S}_\tau$ can be viewed as a discrete
analogue of a scattering matrix, which is a mathematical operator that
encodes the relationship between the incoming field and the effective
charge on a box in the absence of a source \cite{bremer2015high}.

We next move one level up the tree and consider a parent node $\tau$ with
leaf children $\alp$ and $\ba$. Combining Lemma \ref{thm:lemma} and the
fact that both $\alp$ and $\ba$ satisfy \eqref{eq:leaf}, we have the
equilibrium equation
\begin{equation}
  \begin{bmatrix}
    \wt{\vct{q}}_\alp \\ \wt{\vct{q}}_\ba
  \end{bmatrix}
  +
  \begin{bmatrix}
    \mtx{S}_\alp & \mtx{0} \\
    \mtx{0} & \mtx{S}_\ba
  \end{bmatrix}
  \left(
    \begin{bmatrix}
      \mtx{0} & \mtx{G}_{\alp,\ba} \\
      \mtx{G}_{\ba,\alp} & \mtx{0}
    \end{bmatrix}
    \begin{bmatrix}
      \wt{\vct{q}}_\alp \\ \wt{\vct{q}}_\ba
    \end{bmatrix}
    +
    \mtx{U}_\tau \wt{\vct{w}}_\tau
  \right)
  =
  \begin{bmatrix}
    \wt{\vct{r}}_\alp \\ \wt{\vct{r}}_\ba
  \end{bmatrix}.
  \label{eq:parent}
\end{equation}
Setting
\begin{equation}
  \mtx{X}_\tau
  =
  \begin{bmatrix}
    \mtx{I} & \mtx{S}_{\alp} \mtx{G}_{\alp,\ba} \\
    \mtx{S}_\ba \mtx{G}_{\ba,\alp} & \mtx{I}
  \end{bmatrix}^{-1}
  \label{eq:localsol}
\end{equation}
and then multiplying both sides of \eqref{eq:parent} by $\mtx{V}_\tau^*
\mtx{X}_\tau$ produces
\[
  \mtx{V}_\tau^*
  \begin{bmatrix}
    \wt{\vct{q}}_\alp \\ \wt{\vct{q}}_\ba
  \end{bmatrix}
  +
  \mtx{V}_\tau^* \mtx{X}_\tau
  \begin{bmatrix}
    \mtx{S}_\alp & \mtx{0} \\
    \mtx{0} & \mtx{S}_\ba
  \end{bmatrix}
  \mtx{U}_\tau
  \wt{\vct{w}}_\tau = \mtx{V}_\tau^* \mtx{X}_\tau
  \begin{bmatrix}
    \wt{\vct{r}}_\alp \\ \wt{\vct{r}}_\ba
  \end{bmatrix}.
\]
Using \eqref{eq:expansions} and defining
\begin{equation}
  \mtx{S}_\tau =  \mtx{V}_\tau^* \mtx{X}_\tau
  \begin{bmatrix}
    \mtx{S}_\alp & \mtx{0} \\
    \mtx{0} & \mtx{S}_\ba
  \end{bmatrix}
  \mtx{U}_\tau,
  \label{eq:smat}
\end{equation}
and
\begin{equation}
  \quad \wt{\vct{r}}_\tau = \mtx{V}_\tau^* \mtx{X}_\tau
  \begin{bmatrix}
    \wt{\vct{r}}_\alp \\ \wt{\vct{r}}_\ba
  \end{bmatrix}
  \label{eq:load}
\end{equation}
gives us exactly \eqref{eq:leaf} again, where now \eqref{eq:smat} allows us
to compute the discrete scattering matrix for $\tau$ using the scattering
matrices for its children, $\alp$ and $\ba$. This allows us to recursively
traverse up the tree, obtaining a relation identical to \eqref{eq:leaf} at
each level until we get to the root.

Now, let $\tau$ be the root and let $\alp$ and $\ba$ be its children. Since
there is no incoming field on $\tau$, the only incoming field on $\alp$ is
due to charges in $\ba$ and vice versa. Therefore, the equation analogous
to \eqref{eq:parent} simplifies in this case to
\begin{equation}
  \begin{bmatrix}
    \mtx{I} & \mtx{S}_\alp \mtx{G}_{\alp,\ba} \\
    \mtx{S}_\ba \mtx{G}_{\ba,\alp} & \mtx{I}
  \end{bmatrix}
  \begin{bmatrix}
    \wt{\vct{q}}_\alp \\ \wt{\vct{q}}_\ba
  \end{bmatrix}
  =
  \begin{bmatrix}
    \wt{\vct{r}}_\alp \\ \wt{\vct{r}}_\ba
  \end{bmatrix},
  \label{eq:root}
\end{equation}
which can be inverted directly. Then, the incoming fields on $\alp$ and
$\ba$ can be computed by the matrix-vector product
\[
  \begin{bmatrix}
    \wt{\vct{w}}_\alp \\ \wt{\vct{w}}_\ba
  \end{bmatrix}
  =
  \begin{bmatrix}
    \mtx{0} & \mtx{G}_{\alp,\ba} \\
    \mtx{G}_{\ba,\alp} & \mtx{0}
  \end{bmatrix}
  \begin{bmatrix}
    \wt{\vct{q}}_\alp \\ \wt{\vct{q}}_\ba
  \end{bmatrix}.
\]
The incoming and outgoing expansions of the nodes on other levels can be
computed using a downward sweep, as follows. Let $\tau$ be a node, such
that its parent has been processed and so the incoming expansion
$\wt{\vct{w}}_\tau$ has been computed. If $\alp$ and $\ba$ are children of
$\tau$, then using \eqref{eq:parent} and \eqref{eq:localsol} we can obtain
the equation for the outgoing expansions on $\alp$ and $\ba$
\[
  \begin{bmatrix}
    \wt{\vct{q}}_\alp \\ \wt{\vct{q}}_\ba
  \end{bmatrix}
  =
  \mtx{X}_\tau
  \left(
  \begin{bmatrix}
    \wt{\vct{r}}_\alp \\ \wt{\vct{r}}_\ba
  \end{bmatrix}
  -
  \begin{bmatrix}
    \mtx{S}_\alp & \mtx{0} \\
    \mtx{0} & \mtx{S}_\ba
  \end{bmatrix}
  \mtx{U}_\tau \wt{\vct{w}}_\tau
  \right).
\]
The incoming expansions $\wt{\vct{w}}_\alp$ and $\wt{\vct{w}}_\ba$ can then
be computed using \eqref{eq:lemma}. We can proceed to compute the
incoming and outgoing expansions for the children of $\alp$ and $\ba$ and
so on, all the way down the tree.

After the incoming expansion for a leaf node $\tau$ is determined,
\eqref{eq:equi} and \eqref{eq:expansions} are used to determine the
solution of the original system through the equation
\[
  \vct{q}_\tau = \mtx{X}_\tau \left(\vct{f}_\tau - \mtx{B}_\tau \mtx{U}_\tau
  \wt{\vct{w}}_\tau\right).
\]

It is customary to split the solution process into a \textit{build stage}
in which the matrices that define the solution operator are all
constructed, and a \textit{solve stage} in which the computed approximate
solution operator is applied to a given incoming field. In our setting, the
build stage consists of a single pass up the tree in which the local
solution operator $\mtx{X}_\tau$ and the local scattering matrix
$\mtx{S}_\tau$ is computed for each node $\tau$. This is illustrated in
Algorithm \ref{alg:build}, which computes the local solution operators and
scattering matrices on all nodes, first for the leaves using
\eqref{eq:localsol_leaf} and \eqref{eq:smat_leaf}, and then for all other
nodes in an upward pass using \eqref{eq:localsol} and \eqref{eq:smat}.
Then, given a right-hand side vector, the solve stage consists of a pass up
the tree, where the $\wt{\vct{r}}_\tau$ are computed using
\eqref{eq:load_leaf} and \eqref{eq:load}, and then a downward pass where
the incoming expansions $\wt{\vct{w}}_\tau$ are computed. This procedure is
summarized in Algorithm \ref{alg:apply}. In an environment where an inverse
needs to be applied to multiple different vectors, such as in
preconditioning or when one is interested in evaluating the scattered field
for multiple incident fields, the build stage which is the dominant cost
needs to be executed just once.

\begin{remark}
A dominant cost of the build stage is the evaluation of the matrices
$\mtx{X}_{\tau}$, as defined by \eqref{eq:localsol}. This computation
can be accelerated by exploiting the presence of the two identity matrices
on the diagonal and the fact that the matrices $\mtx{G}_{\alp,\ba}$ and
$\mtx{G}_{\ba,\alp}$ tend to be substantially rank deficient. A heuristic
argument for this rank deficiency is that the matrices $\mtx{G}_{\alp,\ba}$
and $\mtx{G}_{\ba,\alp}$ encode the interactions between the skeleton
indices of two adjacent boxes. While these points typically live on the
entire perimeter of each box (see Figures \ref{fig:4thid} and
\ref{fig:10thid}), the interactions can be accurately represented using
mostly points on the single shared side. Thus, we expect the ranks of the
off-diagonal blocks to be a little over 25\% of their sizes, which can be
leveraged for more efficient inversion.

To this end, suppose we have low rank factorizations $\mtx{G}_{\alp,\ba} =
\mtx{L}_{\alp,\ba} \mtx{R}_{\alp,\ba}^*$ and $\mtx{G}_{\ba,\alp} =
\mtx{L}_{\ba,\alp} \mtx{R}_{\ba,\alp}^*$, which can be computed once per
level during the HBS compression at a modest cost. We can apply the
Woodbury formula to \eqref{eq:localsol} to get the equivalent formula for
the local solution operator of a non-leaf node
\begin{equation}
  \mtx{X}_\tau = \mtx{I}
  -
  \begin{bmatrix}
    \mtx{S}_\alp \mtx{L}_{\alp,\ba} & \mtx{0} \\ \mtx{0} & \mtx{S}_\ba
    \mtx{L}_{\ba,\alp}
  \end{bmatrix}
  \begin{bmatrix}
    \mtx{I} & \mtx{R}_{\ba,\alp}^* \mtx{L}_{\ba,\alp} \\
    \mtx{R}_{\alp,\ba}^* \mtx{L}_{\alp,\ba} & \mtx{I}
  \end{bmatrix}^{-1}
  \begin{bmatrix}
    \mtx{0} & \mtx{R}_{\alp,\ba}^* \\
    \mtx{R}_{\ba,\alp}^* & \mtx{0}
  \end{bmatrix}.
  \label{eq:blockinv}
\end{equation}
Note that the matrix that needs to be inverted in \eqref{eq:blockinv} is
substantially smaller than that in \eqref{eq:localsol}.
\label{rmk:woodbury}
\end{remark}

\begin{algorithm}
  \caption{Build inverse as described in Section \ref{s:solver}}
  \textbf{Input}: HBS compression of coefficient matrix specified by
  $(\mtx{U}_\tau, \mtx{V}_\tau)$ for all non-root nodes, $(\mtx{B}_\tau,
  \mtx{G}_\tau$) for all leaves, and $(\mtx{G}_{\alp,\ba},
  \mtx{G}_{\ba,\alp})$ for all parents with children $\alp$ and
  $\ba$ \\
    \textbf{Output}: Inverse of HBS matrix specified by $\mtx{X}_\tau$ for
    all nodes and $\mtx{S}_\tau$ for all non-root nodes \\[-.5em]
    \hphantom\,\hrulefill
  \begin{algorithmic}[1]
    \ForAll {$\tau$ a leaf}
      \State $\mtx{X}_\tau = (\mtx{I} + \mtx{B}_\tau \mtx{G}_\tau)^{-1}$
      \State $\mtx{S}_\tau = \mtx{V}_\tau^* \mtx{X}_\tau \mtx{B}_\tau \mtx{U}_\tau$
    \EndFor
    \For{$\ell$ = $L-1$, $L-3$, \dots, 0}
      \ForAll {$\tau$ on level $\ell$}
        \State $[\alp,\ba] \leftarrow $ children of $\tau$
        \State $\mtx{X}_\tau = \begin{bmatrix} \mtx{I} & \mtx{S}_\alp
        \mtx{G}_{\alp,\ba} \\ \mtx{S}_{\ba} \mtx{G}_{\ba,\alp} &
      \mtx{I}\end{bmatrix}^{-1}$
      \If {$\ell \neq  0$}
      \State $\mtx{S}_\tau = \mtx{V}_\tau^* \mtx{X}_\tau \begin{bmatrix}
        \mtx{S}_\alp &  \mtx{0} \\ \mtx{0} & \mtx{S}_\ba \end{bmatrix}
        \mtx{U}_\tau$ \EndIf
      \EndFor
    \EndFor
  \end{algorithmic}
  \label{alg:build}
\end{algorithm}

\begin{algorithm}
  \caption{Apply inverse to right-hand side $\vct{f}$ as described in
  Section \ref{s:solver}}
  \textbf{Input}: HBS matrix, its inverse as in Algorithm \ref{alg:build}, and right-hand
  side $\vct{f}$ \\
  \textbf{Output}: Solution vector $\vct{q}$  \\[-.5em]
  \hphantom\,\hrulefill \\[-2em]
  \setlength\columnsep{4pt}
  \begin{multicols}{2}
    \begin{algorithmic}[1]
    \vspace{.2em}
    \State // Upwards pass
    \ForAll {$\tau$ a leaf}
    \State $\vct{r}_\tau = \mtx{X}_\tau \vct{f}_\tau$
    \State $\wt{\vct{r}}_\tau = \mtx{V}_\tau^* \vct{r}_\tau$
    \EndFor
    \vspace{.1em}
    \For{$\ell$ = $L-1$, $L-3$, \dots, 1}
      \ForAll {$\tau$ on level $\ell$}
        \State $[\alp,\ba] \leftarrow $ children of $\tau$
        \State $\wt{\vct{r}}_\tau = \mtx{V}_\tau^* \mtx{X}_\tau
        \begin{bmatrix} \wt{\vct{r}}_\alp \\ \wt{\vct{r}}_\ba
        \end{bmatrix}$
      \EndFor
    \EndFor
    \vspace{.5em}
    \State // Top-level solve
    \State $\tau \leftarrow $ root
    \State $[\alp,\ba] \leftarrow $ children of $\tau$ \vspace{.2em}
    \State $
  \begin{bmatrix}
    \wt{\vct{w}}_\alp \\ \wt{\vct{w}}_\ba
  \end{bmatrix}
  =
  \begin{bmatrix}
    \mtx{0} & \mtx{G}_{\alp,\ba} \\
    \mtx{G}_{\ba,\alp} & \mtx{0}
  \end{bmatrix} \mtx{X}_\tau \begin{bmatrix} \wt{\vct{r}}_\alp \\
  \wt{\vct{r}}_\ba \end{bmatrix}$
    \vspace{1em} \columnbreak
    \State // Downwards pass
    \For{$\ell = 2, 3,\hdots, L-1$}
      \ForAll {$\tau$ on level $\ell$}
        \State $[\alp,\ba] \leftarrow $ children of $\tau$
        \vspace{.15em}
        \State $\begin{bmatrix} \wt{\vct{q}}_\alp \\ \wt{\vct{q}}_\ba
        \end{bmatrix} = \mtx{X}_\tau \left( \begin{bmatrix}
        \wt{\vct{r}}_\alp \\ \wt{\vct{r}}_\ba \end{bmatrix} -
        \begin{bmatrix} \mtx{S}_\alp & \mtx{0} \\ \mtx{0} & \mtx{S}_\ba
      \end{bmatrix} \mtx{U}_\tau \wt{\vct{w}}_\tau \right)$
        \vspace{.15em}
        \State $\begin{bmatrix} \wt{\vct{w}}_\alp \\ \wt{\vct{w}}_\ba
        \end{bmatrix} =  \begin{bmatrix} \mtx{0} & \mtx{G}_{\alp,\ba} \\
        \mtx{G}_{\ba,\alp} & \mtx{0} \end{bmatrix}  \begin{bmatrix}
        \wt{\vct{q}}_\alp \\ \wt{\vct{q}}_\ba \end{bmatrix} + \mtx{U}_\tau
      \wt{\vct{w}}_\tau$
        \EndFor
    \EndFor
    \vspace{.5em}
    \State // Form solution vector
    \ForAll {$\tau$ a leaf}
    \State $\vct{q}_\tau = \vct{r}_\tau - \mtx{X}_\tau \mtx{B}_\tau
    \mtx{U}_\tau \wt{\vct{w}}_\tau$
    \EndFor
  \end{algorithmic}
  \end{multicols}
  \label{alg:apply}
\end{algorithm}

\subsection{Complexity estimate}
\label{s:complexity}

The computational complexity of our solver depends on the numerical ranks
of the off-diagonal blocks in the HBS format.
For a 2D volume kernel like \eqref{eq:kernel}, the interaction rank of a
box grows like the square root of the number of points it contains when the
wavenumber and compression tolerance are fixed \cite[Sec.~4]{ho2012fast}.
Since the HBS compression and inversion stages use cubic dense linear
algebra, the cost of each is dominated by the cost of processing the root
level where the rank of interaction grows as $\mc{O}(N^{1/2})$, and thus
both stages require $\mc{O}(N^{3/2})$ operations.
Applying the inverse to a vector requires an operation that grows
quadratically with the rank of each box, so each level requires $\mc{O}(N)$
operations, for a total of $\mc{O}(N \log N)$ operations. As for memory,
besides the HBS factors the only matrices that need be stored for the
inverse are the local solution operators $\mtx{X}_\tau$, and so the cost of
storage is still $\mc{O}(N)$. For a more detailed analysis of the costs of
solvers for HBS matrices, we refer the reader to \cite[Sec.~4]{ho2012fast},
which applies in our setting as well.

\begin{remark}[Complexity for a fixed PDE]
In the case where $\ka$ is kept fixed as $N$ is increased,
it should be possible to reduce the $\mc{O}(N^{3/2})$ cost of
the inversion stage to $\mc{O}(N)$
using techniques such as HIF \cite{ho2016hierarchical},
nested trees \cite{corona2015direct}, or
strong recursive skeletonization \cite{2017_ho_ying_strong_RS}.
This is the subject of ongoing work.
\label{rmk:fast}
\end{remark}

\begin{remark}[Complexity for a fixed number of points per wavelength]
In the case where $\ka$ is increased as $N$ grows to keep a fixed number
of discretization points per wavelength (so that $\ka \sim N^{1/2}$),
the overall complexity of the solver is $\mc{O}(\ka^3)$. In this case,
the acceleration techniques referenced in Remark \ref{rmk:fast} do not
apply, and we are not aware of any direct solvers that scale better than
$\mc{O}(\ka^3)$ in the general case.
\end{remark}

\subsection{Comparison with existing solvers}
\label{s:comparison}
The connection between the procedure described in Section \ref{s:solver}
and some existing work on HBS-based direct solvers can be elucidated through
the use of the sparse matrix embeddings introduced in \cite{ho2012fast}.
For the equation \eqref{eq:linsys} and an HBS tree with only two levels, the
corresponding matrix embedding is given by
\begin{equation}
\begin{bmatrix}
  (\mtx{I}+\mtx{B}_{\alp} \mtx{G}_{\alp}) &  \mtx{0} &  \mtx{B}_\alp
  \mtx{U}_\alp & \mtx{0} & \mtx{0} & \mtx{0} \\
  \mtx{0} & (\mtx{I} + \mtx{B}_\ba \mtx{G}_{\ba}) & \mtx{0} &
  \mtx{B}_\ba \mtx{U}_\ba & \mtx{0} & \mtx{0}\\
  -\mtx{V}_\alp^* & \mtx{0} & \mtx{0} & \mtx{0} & \mtx{I} & \mtx{0}\\
  \mtx{0} & -\mtx{V}_\ba^* & \mtx{0} & \mtx{0} & \mtx{0} & \mtx{I} \\
  \mtx{0} & \mtx{0} & - \mtx{I} & \mtx{0} & \mtx{0} & \mtx{G}_{\alp,\ba} \\
  \mtx{0} & \mtx{0} & \mtx{0} & -\mtx{I} & \mtx{G}_{\ba,\alp} & \mtx{0}
\end{bmatrix}
\begin{bmatrix}
  \vct{q}_\alp \\
  \vct{q}_\ba \\
  \wt{\vct{w}}_\alp \\
  \wt{\vct{w}}_\ba \\
  \wt{\vct{q}}_\alp \\
  \wt{\vct{q}}_\ba \\
\end{bmatrix}
=
\begin{bmatrix}
  \vct{f}_\alp \\
  \vct{f}_\ba \\
  \vct{0} \\
  \vct{0} \\
  \vct{0} \\
  \vct{0} \\
\end{bmatrix}.
\label{eq:sparse}
\end{equation}
Upon eliminating $\vct{q}_\alp$ and $\vct{q}_\ba$ in the equations for
$\wt{\vct{w}}_\alp$ and $\wt{\vct{w}}_\ba$, the system
\begin{equation}
  \begin{bmatrix}
    \mtx{S}_\alp & \mtx{0} & \mtx{I} & \mtx{0} \\
    \mtx{0} & \mtx{S}_\ba  & \mtx{0} & \mtx{I} \\
    -\mtx{I} & \mtx{0} & \mtx{0} & \mtx{G}_{\alp,\ba} \\
    \mtx{0} & -\mtx{I} & \mtx{G}_{\ba,\alp} & \mtx{0}
  \end{bmatrix}
\begin{bmatrix}
  \wt{\vct{w}}_\alp \\
  \wt{\vct{w}}_\ba \\
  \wt{\vct{q}}_\alp \\
  \wt{\vct{q}}_\ba \\
\end{bmatrix}
=
\begin{bmatrix}
  \wt{\vct{r}}_\alp \\
  \wt{\vct{r}}_\ba \\
  \vct{0} \\
  \vct{0} \\
\end{bmatrix}
\label{eq:sparse2}
\end{equation}
is obtained, where we continue to use the notation of Section
\ref{s:solver}. The classical Woodbury inversion formula for HBS matrices
as discussed in \cite{gillman2012direct} corresponds to using the (1,\,1) and
(2,\,2) blocks in \eqref{eq:sparse2} to eliminate $\wt{\vct{w}}_\alp$ and
$\wt{\vct{w}}_\ba$ to obtain
\begin{equation}
\label{eq:dvorakbad}
\begin{bmatrix}
  \mtx{S}_{\alp}^{-1} & \mtx{G}_{\alp,\ba} \\ \mtx{G}_{\ba,\alp} &
  \mtx{S}_\ba^{-1}
\end{bmatrix}
\begin{bmatrix}
  \wt{\vct{q}}_\alp \\
  \wt{\vct{q}}_\ba \\
\end{bmatrix}
=
\begin{bmatrix}
  \mtx{S}_\alp^{-1} \wt{\vct{r}}_\alp \\
  \mtx{S}_\ba^{-1} \wt{\vct{r}}_\ba \\
\end{bmatrix}.
\end{equation}
This requires inverting both $\mtx{S}_\alp$ and $\mtx{S}_\ba$, which in our
application is often highly ill-conditioned due to $\mtx{B}$ in
\eqref{eq:linsys} containing small diagonal entries. Indeed, for the
problems in Section \ref{s:numerics}, we frequently encounter scattering
matrices with condition numbers far greater than the reciprocal of machine
precision in double-precision arithmetic.

The scheme described in Section \ref{s:solver}, on the other hand, uses the
(3,\,1) and (4,\,2) blocks in \eqref{eq:sparse2} as pivots. This
leads to the mathematically equivalent system
\begin{equation}
\label{eq:dvorak}
\begin{bmatrix}
  \mtx{I} & \mtx{S}_\alp \mtx{G}_{\alp,\ba}  \\
  \mtx{S}_\ba \mtx{G}_{\ba,\alp}  & \mtx{I}
\end{bmatrix}
\begin{bmatrix}
  \wt{\vct{q}}_\alp \\
  \wt{\vct{q}}_\ba \\
\end{bmatrix}
=
\begin{bmatrix}
  \wt{\vct{r}}_\alp \\
  \wt{\vct{r}}_\ba \\
\end{bmatrix}.
\end{equation}
Solving (\ref{eq:dvorak}) is often a more benign operation, since the coefficient matrix
is a perturbation to the identity.
In particular, in regions where the scattering potential is small in magnitude,
the ``perturbation terms'' $\mtx{S}_{\alpha}\mtx{G}_{\alpha\beta}$
and $\mtx{S}_{\beta}\mtx{G}_{\beta\alpha}$ will be of small magnitude too, which
ensures that (\ref{eq:dvorak}) is well-conditioned.
In other words, the formulation (\ref{eq:dvorak})
avoids the ``spurious'' ill-conditioning issues in (\ref{eq:dvorakbad}).
In addition, the resulting solvers require fewer arithmetic operations.

The approach suggested in \cite{ho2012fast} is to apply a sparse direct
solver package, such as UMFPACK, to the entire system \eqref{eq:sparse}.
Such packages have highly optimized algorithms for selecting pivots to
minimize fill-in, while avoiding the inversion of ill-conditioned matrices.
In our experiments, we find this strategy is as robust as our solver, but
we empirically find the approach of \cite{ho2012fast} to be slower and more memory-intensive
than the technique described here.

\section{Using the solver as a preconditioner}
\label{s:precon}

The computational cost of the direct solver that we have described
depends strongly on the requested accuracy, for two reasons.
First, as the compression tolerance goes to zero, all interaction ranks
increase, and since the solver relies on dense linear algebra, its overall
run time scales cubically with the numerical ranks.
Second, in order to attain high accuracy, the high-order accurate version
of Duan-Rokhlin quadrature must be used, which as discussed in Section
\ref{s:compression} increases the ranks even further.

When high accuracy is desired, the best approach is often
to run the direct solver that we have described at low accuracy, and then
use the computed rough inverse as a preconditioner. This turns out to be
remarkably effective due to the fact that the $4^{\rm th}$-order accurate
Duan-Rokhlin quadrature involves only the diagonal entries of the matrix,
which means that they do not impact the compression of the off-diagonal
blocks at all.

To describe the idea in more detail, suppose that we seek to solve a
high-order discretization such as, cf.~\eqref{eq:linsys},
\begin{equation*}
  (\mtx{I} + \mtx{B}\mtx{G} + \mtx{B} \mtx{T}) \vct{q}  = \vct{f},
\end{equation*}
where $\mtx{G}$ is defined by \eqref{eq:Gmat} (representing the
$O(h^{4})$ order accurate discretization of the integral operator),
and where $\mtx{T}$ is the necessary ``correction'' matrix for a $10^{\rm
th}$-order discretization. In other words, $\mtx{T}$ contains the
$10^{\rm th}$-order corrections minus the $4^{\rm th}$-order corrections.
The idea is now to build a rough approximate inverse $\mtx{M}$ to the more
compressible $4^{\rm th}$-order discretization of the operator, so that
$$
\mtx{M} \approx \bigl(\mtx{I} + \mtx{B}\mtx{G}\bigr)^{-1}.
$$
We then apply GMRES to the preconditioned system
\begin{equation*}
  \mtx{M}(\mtx{I} + \mtx{B}\mtx{G} + \mtx{B} \mtx{T}) \vct{q}  = \mtx{M}\vct{f},
\end{equation*}
exploiting that the original operator $\mtx{I} + \mtx{B}\mtx{G} + \mtx{B}
\mtx{T}$ can be applied rapidly to vectors since $\mtx{B}$ is sparse, and
$\mtx{G} + \mtx{T}$ is a convolution operator that be applied through the
FFT (see \cite{duan2009high} for details).

The left image in Figure \ref{fig:spectra} displays the spectrum in the
complex plane of a discretization of the left-hand side of \eqref{eq:ls}
using 1600 degrees of freedom with $10^{\rm th}$-order quadrature corrections, where
$\ka = 8 \pi$, so that each side of $\Om$ is 4 vacuum wavelengths across,
and the scattering potential was chosen to be the lens described in Section
\ref{s:hps_paper}.
While most of the eigenvalues cluster near 1 reflecting the fact that
\eqref{eq:ls} is a second-kind Fredholm integral equation, there are many
eigenvalues that are located far from 1.  As a result, an iterative method,
such as GMRES, will converge quite slowly. For example, for a right-hand
side corresponding to a plane wave, GMRES takes 62 iterations to reach a
tolerance of $10^{-5}$ and 80 iterations to reach a tolerance of
$10^{-10}$.

The image on the right of Figure \ref{fig:spectra} shows the spectrum after
preconditioning with the direct solver with compression tolerance $10^{-2}$
and the $4^{\rm th}$-order quadrature corrections. Now, all of the eigenvalues are
clustered closely to 1, which means that GMRES will converge in just a few
iterations, even if high accuracy is required. (In fact, all
eigenvalues are within distance 0.06 from 1!) For the same plane
wave as in the previous paragraph, GMRES now only takes 3 iterations to
converge to a tolerance of $10^{-5}$ and 6 iterations to converge to a
tolerance of $10^{-10}$. Furthermore, for this example, building the
preconditioner takes only a fraction of a second on a laptop computer and
each preconditioned iteration is only twice as slow as an
iteration without any preconditioning.

\begin{figure}
  \centering
  \includegraphics[scale=0.40]{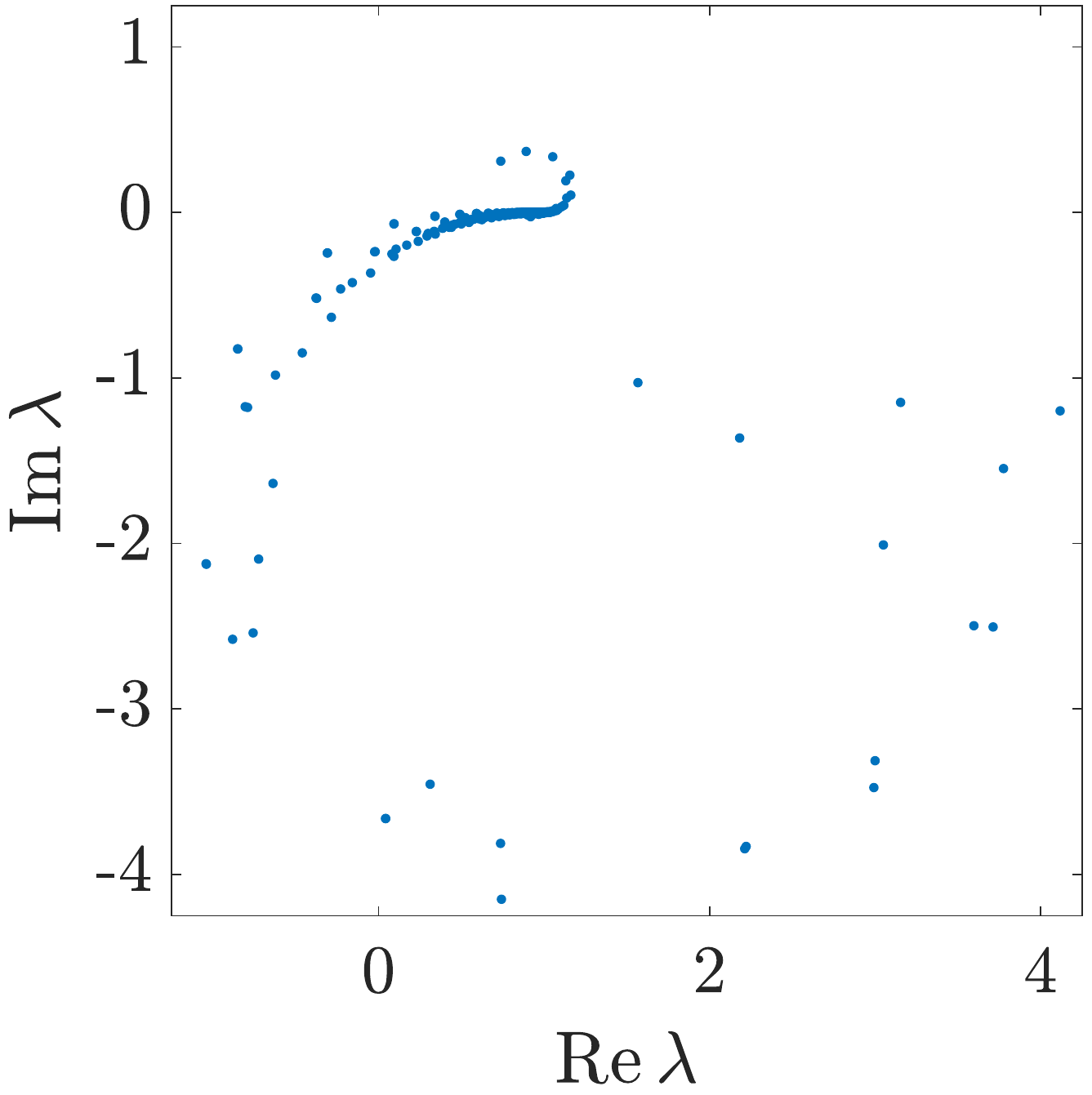} \hspace{4em}
  \includegraphics[scale=0.40]{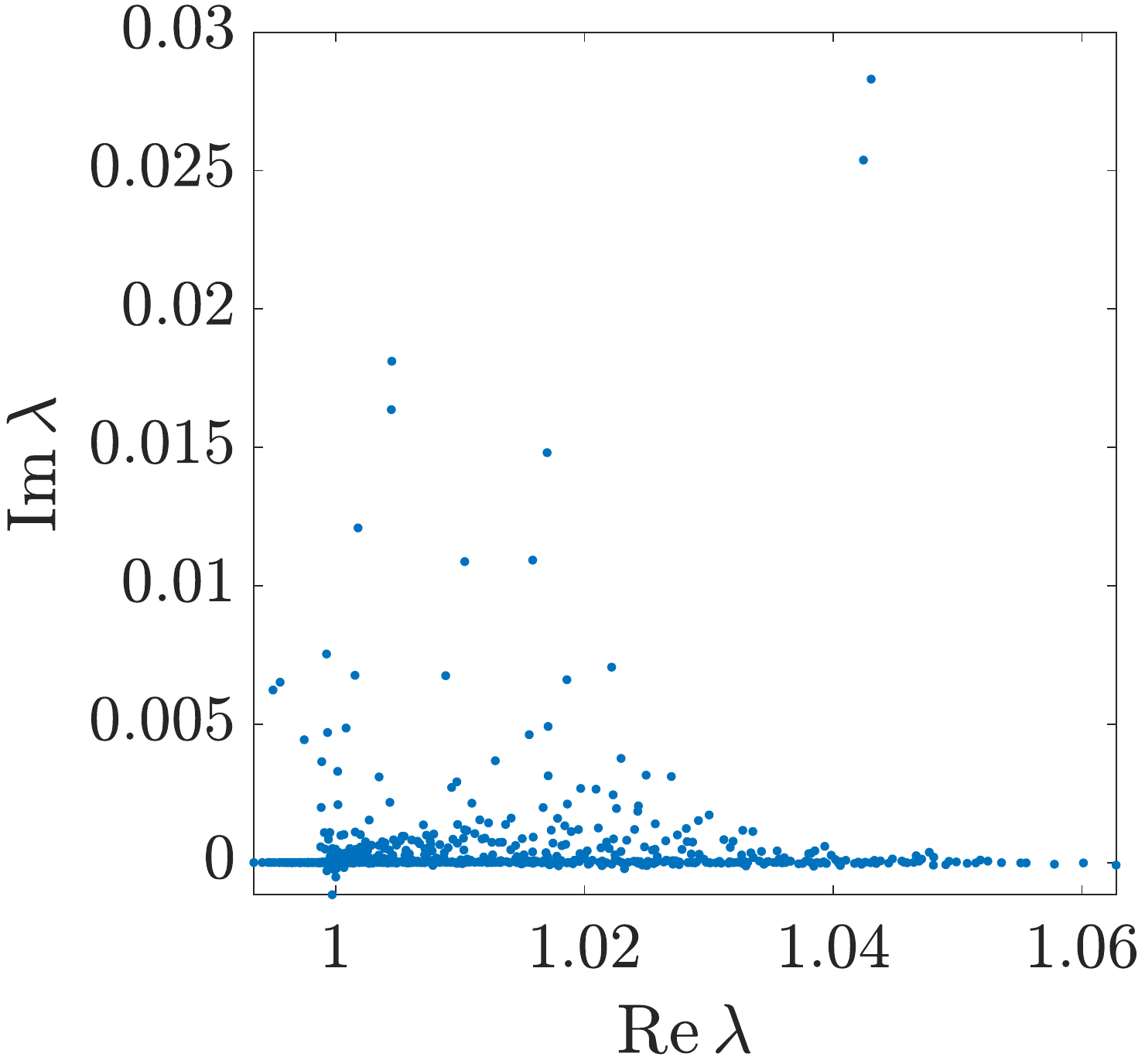}
  \caption{Spectra of unpreconditioned (left) and preconditioned (right)
  discretization of Lippmann-Schwinger equation with $\ka = 8\pi$, 1600
  degrees of freedom, $10^{\rm th}$-order quadrature corrections, and the lens
  potential as described in Section \ref{s:hps_paper}.
  Observe the scale in the figure to the right; all eigenvalues are \textit{very}
  close to one.
  }
  \label{fig:spectra}
\end{figure}

\section{Numerical experiments}
\label{s:numerics}

We present three sets of numerical experiments:
In Section
\ref{s:gaussian}, we study the scaling of our solver with respect to mesh
refinement for a scattering potential given by a Gaussian bump. We
empirically verify the theoretical estimates in Section \ref{s:complexity}.
This is followed by Section \ref{s:cavity}, where we examine the case of
scattering against a cavity. We end on Section \ref{s:hps_paper} which
demonstrates how our solver can be used as a preconditioner as motivated in
Section \ref{s:precon}. Here we consider scattering against a cavity in the
regime where the wavenumber is increased with the number of points per
wavelength held constant, and also four challenging problems introduced in
\cite{gillman2015spectrally}. All timings are reported on a workstation
with 2 Intel Xeon Gold 6254 (18 cores at 3.1GHz base frequency) processors
with 768 GB of RAM.

\subsection{The direct solver applied to a Gaussian scattering potential}
\label{s:gaussian}

We first study the case where the scattering potential is given by $b(\pxx)
= 1.5\exp(-160 \|\pxx\|^2)$ as shown in the left plot in Figure
\ref{fig:gaussian}. The wavenumber is set to $\ka = 25$, so that $\Om$
is approximately 4 wavelengths wide, and the incident field is taken to be
the plane wave $u^{\rm inc}(\pxx) = \exp(i \ka x_1 )$. We run the direct
solver with compression tolerances $10^{-3}$, $10^{-6}$, $10^{-9}$ and
$10^{-12}$ with the $10^{\rm th}$-order quadrature corrections. For each
choice of compression tolerance, we keep the number of degrees of freedom
in a leaf node of the HBS tree constant at 100, increase the number of
levels in the tree, and note the following quantities:

\begin{tabular}{p{2cm}p{13cm}}
   $N$ & The total number of degrees of freedom. \\
   $h$ & The mesh width. Note that $h = 1/\sqrt{N}$. \\
   $T_{\rm skel}$ & Time in seconds to compute the HBS approximation. \\
   $T_{\rm build}$ & Time in seconds to invert the HBS approximation. \\
   $T_{\rm apply}$ & Time in seconds to apply the inverse a single
    time. \\
   mem & Total memory in GB used to store both the HBS approximation
   and its inverse. \\
   res & The relative residual error in the computed solution.
\end{tabular}

\noindent The results are summarized in Tables
\ref{t:gaussian1}--\ref{t:gaussian4}.
The total field is shown in the right plot of Figure \ref{fig:gaussian}.

\begin{figure}
  \centering
  \includegraphics[width=70mm]{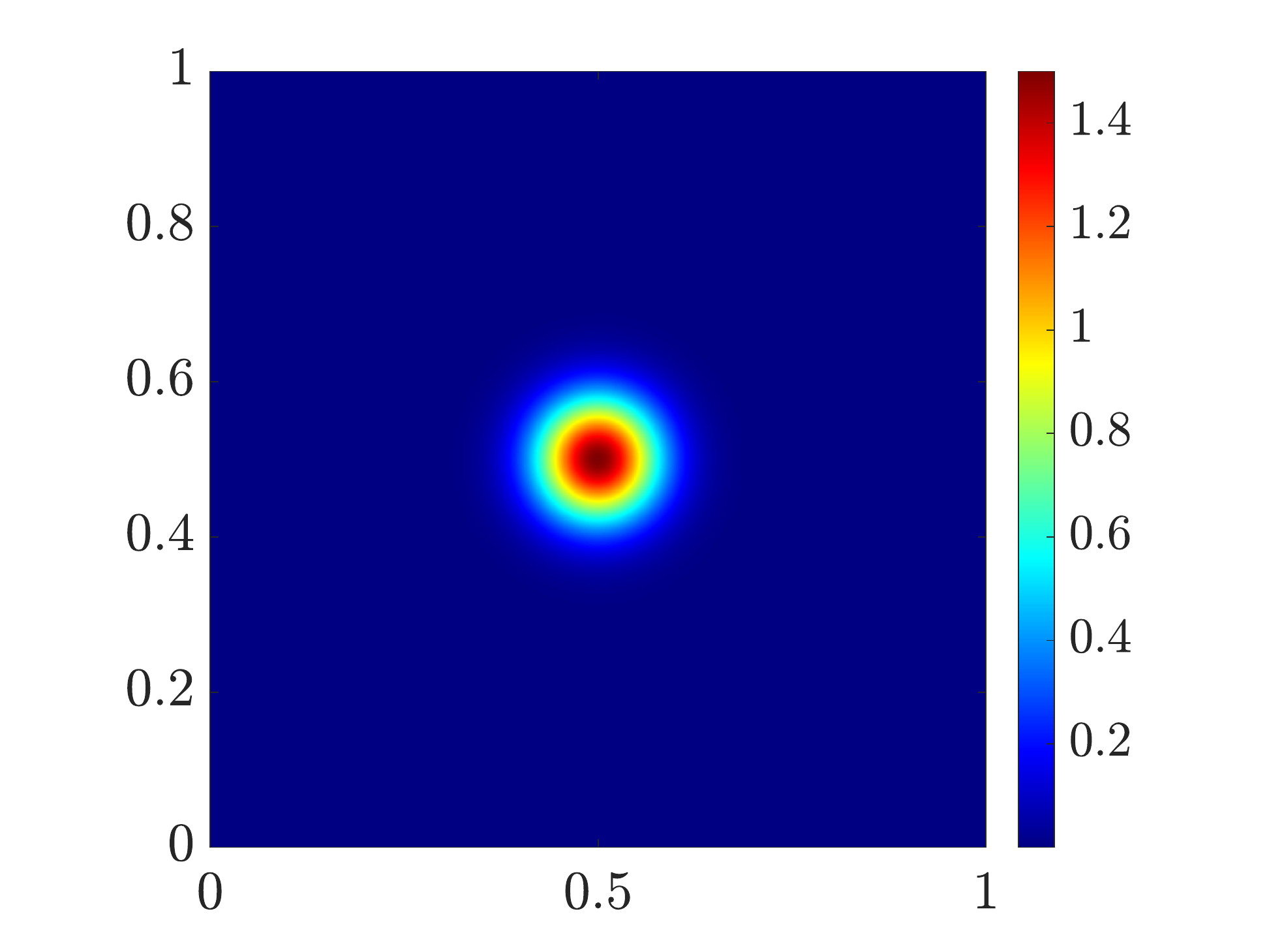}
  \includegraphics[width=70mm]{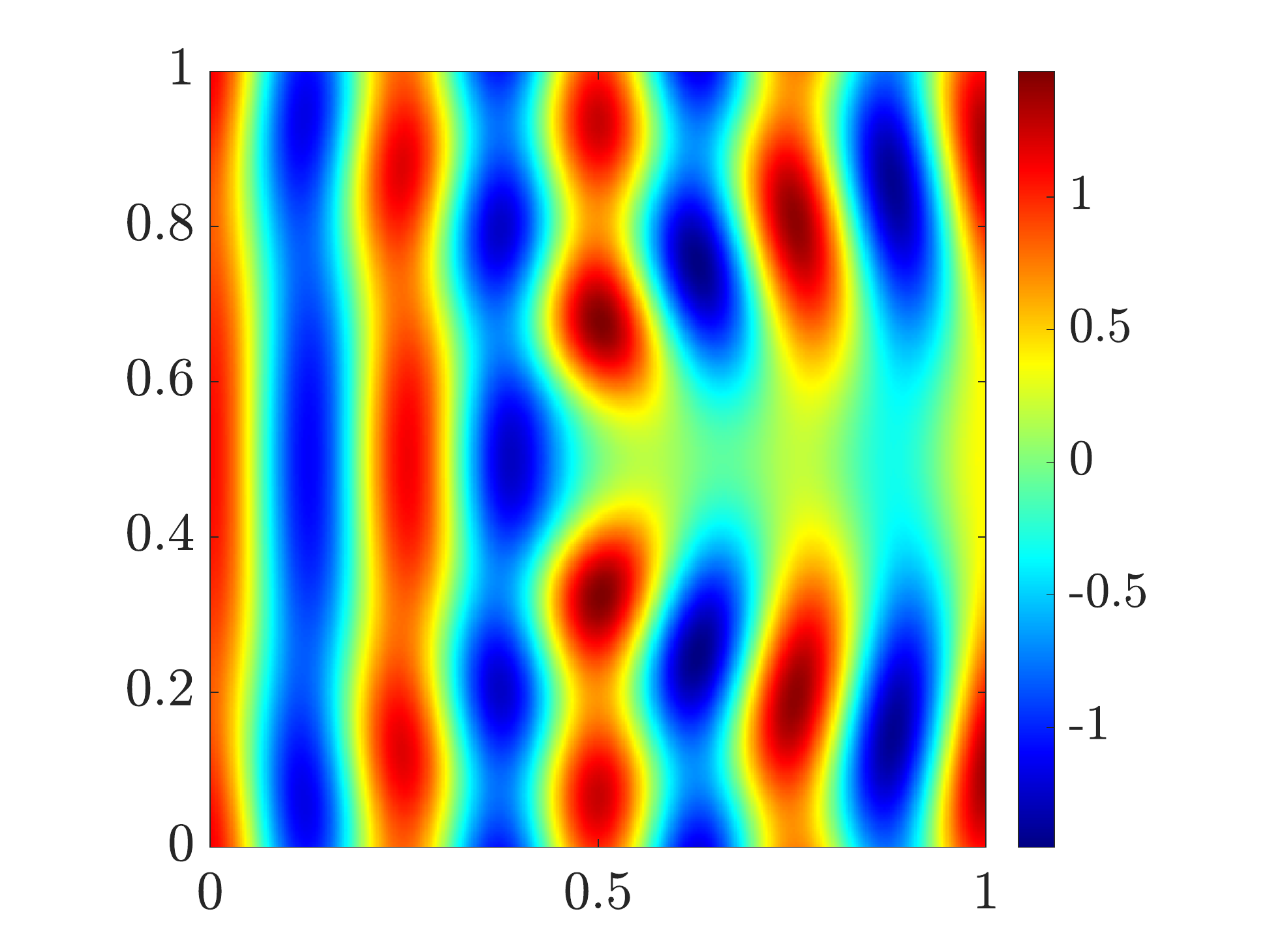}
  \caption{Left: Scattering potential for Gaussian bump in Section
  \ref{s:gaussian}. Right: Total field for Section \ref{s:gaussian}.}
  \label{fig:gaussian}
\end{figure}

\begin{table}[H]
  \centering
  \begin{tabular}{|ccccccc|}
    \hline
    $N$ & $h$ & $T_{\rm skel}$ & $T_{\rm build}$ & $T_{\rm apply}$ & mem &
    res\\
    \hline
    6400 & 0.0125 & 1.35 & 0.40 & 0.05 & 0.12 & 7.74e-05 \\
    25600 & 0.00625 & 5.51 & 1.77 & 0.12 & 0.60 & 5.71e-05 \\
    102400 & 0.003125 & 28.78 & 9.07 & 0.48 & 2.94 & 9.75e-05 \\
    409600 & 0.0015625 & 128.50 & 39.35 & 2.00 & 12.14 & 1.56e-04 \\
    1638400 & 0.00078125 & 680.04 & 165.15 & 10.53 & 48.59 & 5.77e-04 \\
    \hline
  \end{tabular}
  \caption{Data for Section \ref{s:gaussian} with compression tolerance
  $10^{-3}$.}
  \label{t:gaussian1}
\end{table}

\begin{table}[H]
  \centering
  \begin{tabular}{|ccccccc|}
    \hline
    $N$ & $h$ & $T_{\rm skel}$ & $T_{\rm build}$ & $T_{\rm apply}$ & mem &
    res\\
    \hline
    6400 & 0.0125 & 2.30 & 0.86 & 0.07 & 0.27 & 9.54e-09 \\
    25600 & 0.00625 & 11.57 & 4.43 & 0.19 & 1.42 & 7.13e-08 \\
    102400 & 0.003125 & 56.32 & 22.93 & 0.76 & 6.86 & 4.15e-08 \\
    409600 & 0.0015625 & 328.92 & 114.94 & 6.58 & 31.33 & 4.39e-08 \\
    1638400 & 0.00078125 & 2201.34 & 590.61 & 33.50 & 145.71 & 3.14e-07 \\
    \hline
  \end{tabular}
  \caption{Data for Section \ref{s:gaussian} with compression tolerance
  $10^{-6}$.}
  \label{t:gaussian2}
\end{table}

\begin{table}[H]
  \centering
  \begin{tabular}{|ccccccc|}
    \hline
    $N$ & $h$ & $T_{\rm skel}$ & $T_{\rm build}$ & $T_{\rm apply}$ & mem &
    res\\
    \hline
    6400 & 0.0125 & 2.82 & 1.22 & 0.07 & 0.36 & 1.57e-12 \\
    25600 & 0.00625 & 16.48 & 6.91 & 0.24 & 1.98 & 3.37e-12 \\
    102400 & 0.003125 & 80.08 & 37.49 & 1.04 & 10.01 & 1.45e-11 \\
    409600 & 0.0015625 & 475.07 & 200.59 & 6.82 & 48.22 & 3.48e-09 \\
    1638400 & 0.00078125 & 3139.07 & 1023.81 & 67.17 & 225.26 & 2.99e-09 \\
    \hline
  \end{tabular}
  \caption{Data for Section \ref{s:gaussian} with compression tolerance
  $10^{-9}$.}
  \label{t:gaussian3}
\end{table}

\begin{table}[H]
  \centering
  \begin{tabular}{|ccccccc|}
    \hline
    $N$ & $h$ & $T_{\rm skel}$ & $T_{\rm build}$ & $T_{\rm apply}$ & mem &
    res\\
    \hline
    6400 & 0.0125 & 2.99 & 1.41 & 0.07 & 0.38 & 1.87e-15 \\
    25600 & 0.00625 & 15.90 & 7.80 & 0.22 & 2.11 & 3.80e-15 \\
    102400 & 0.003125 & 83.22 & 42.56 & 1.00 & 10.82 & 6.94e-15 \\
    409600 & 0.0015625 & 529.54 & 237.64 & 5.98 & 52.79 & 1.71e-14 \\
    1638400 & 0.00078125 & 3474.55 & 1304.09 & 47.13 & 250.27 & 5.07e-14 \\
    \hline
  \end{tabular}
  \caption{Data for Section \ref{s:gaussian} with compression tolerance
  $10^{-12}$.}
  \label{t:gaussian4}
\end{table}

Figure \ref{fig:gaussian_scaling} provides plots that compare the observed
timings with the estimates of the asymptotic complexity in Section \ref{s:complexity}.
The empirical results indicate slightly better scaling than the predicted
$O(N^{3/2})$ complexity, but this effect is merely due to the fact that the
true asymptotic has not yet had a chance to establish itself at these problem
sizes. Looking at the prefactors in the scalings, we see that compression is
significantly slower than inversion. As the compression tolerance $\epsilon$
goes to zero, the prefactors grow at an observed rate of between $O(\log(1/\epsilon))$
and $O((\log(1/\epsilon))^{2})$.

\begin{remark}
It is interesting to note that in Tables \ref{t:gaussian1}--\ref{t:gaussian4}
the residual errors tend to grow with the number of degrees of freedom.
This is due to the accumulation of errors in compression as the number of
levels in the tree increase, cf.~Remark \ref{remark:howpicklocaltolerance}.
Nonetheless, the residuals here are all close to the HBS compression tolerance.
\end{remark}

\begin{figure}
  \centering
  \includegraphics[scale=0.39]{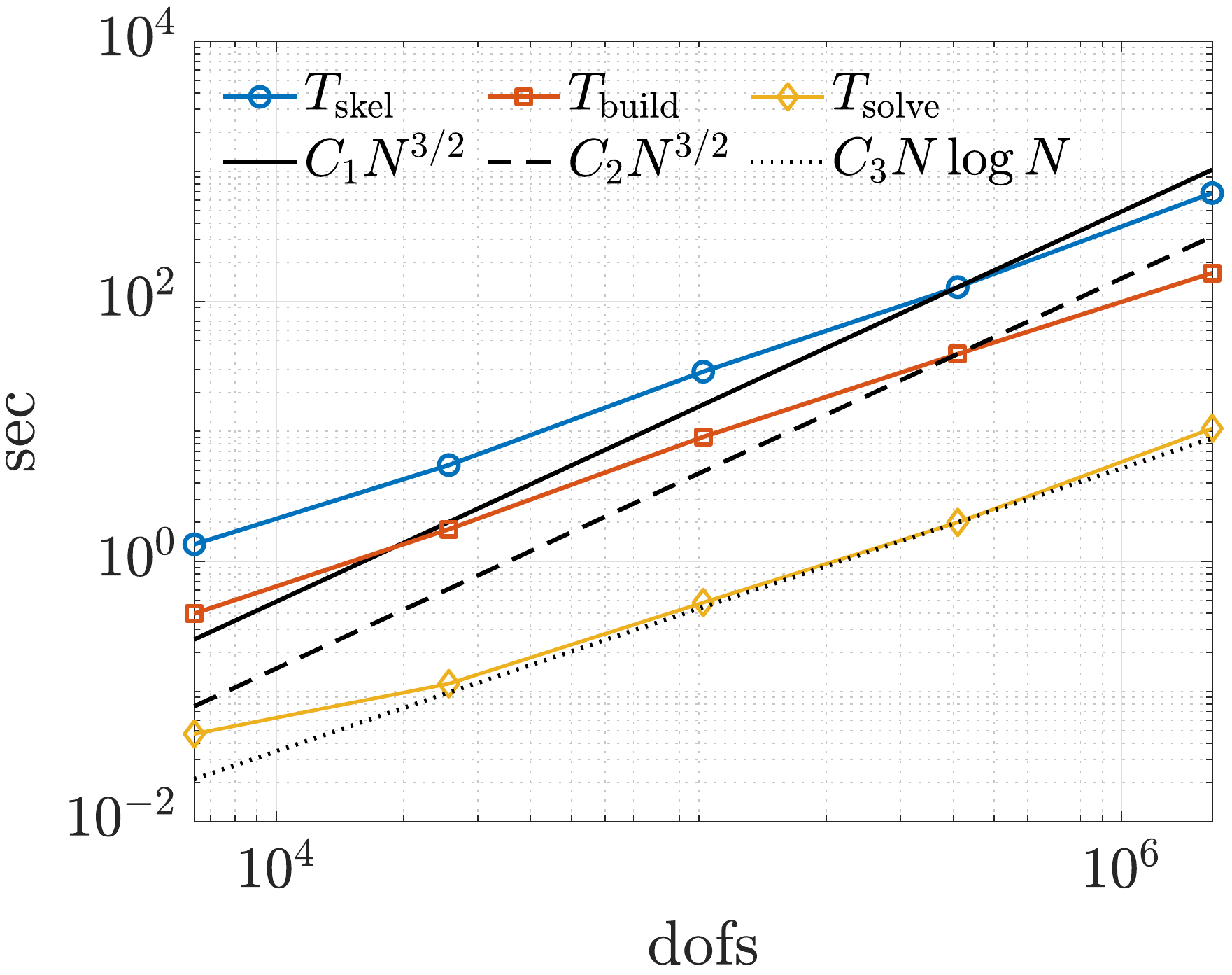} \hspace{2em}
  \includegraphics[scale=0.39]{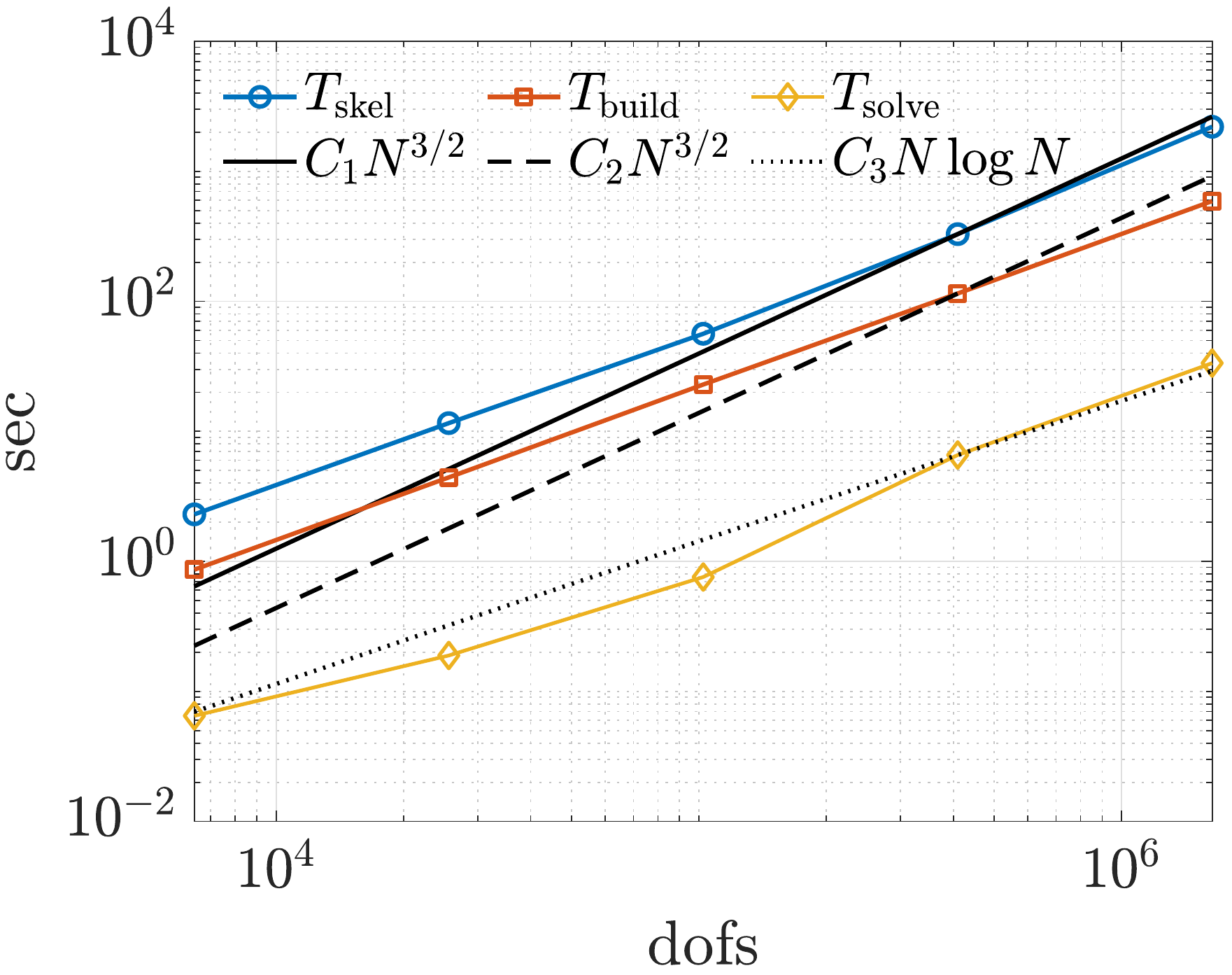} \\[1em]
  \includegraphics[scale=0.39]{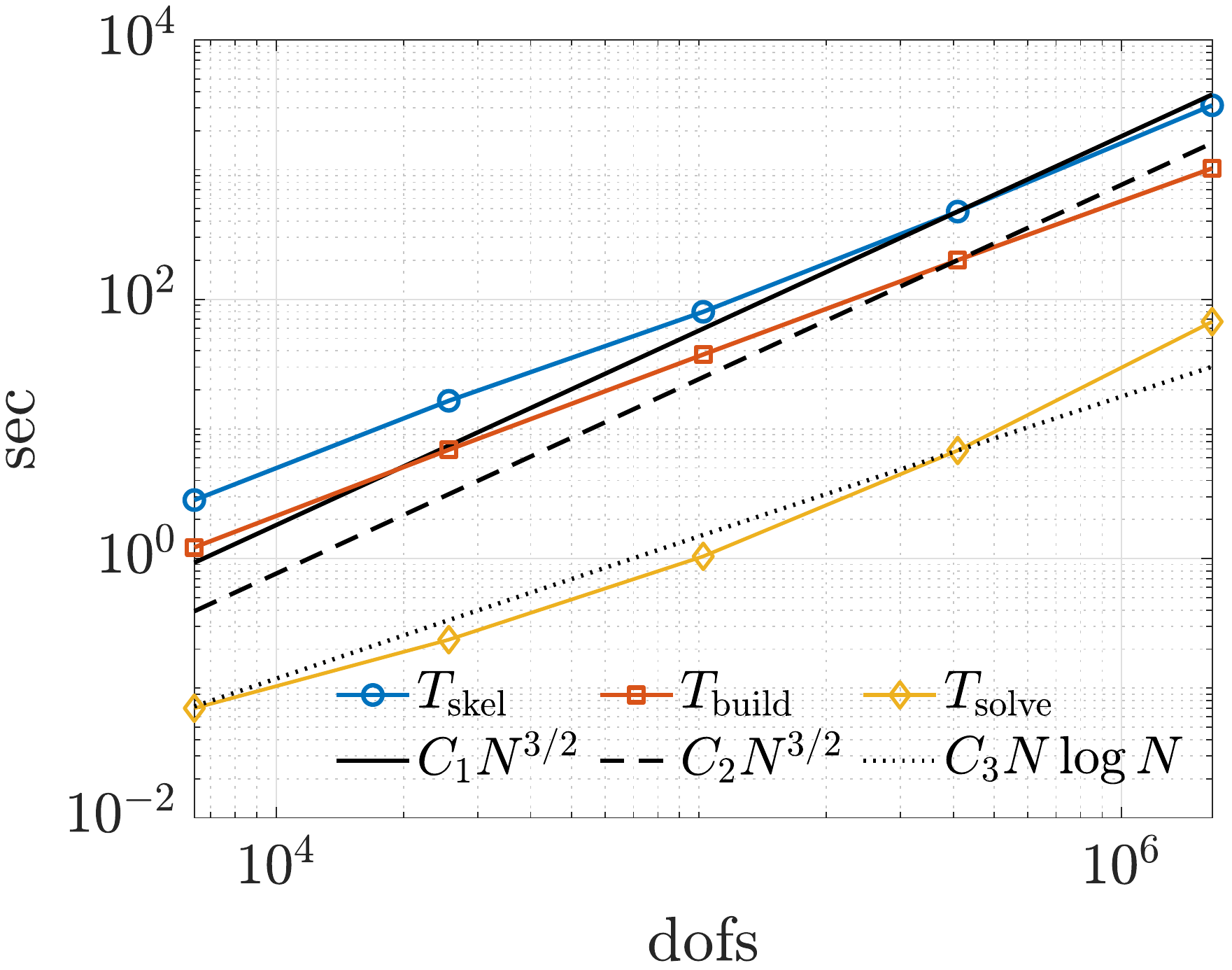} \hspace{2em}
  \includegraphics[scale=0.39]{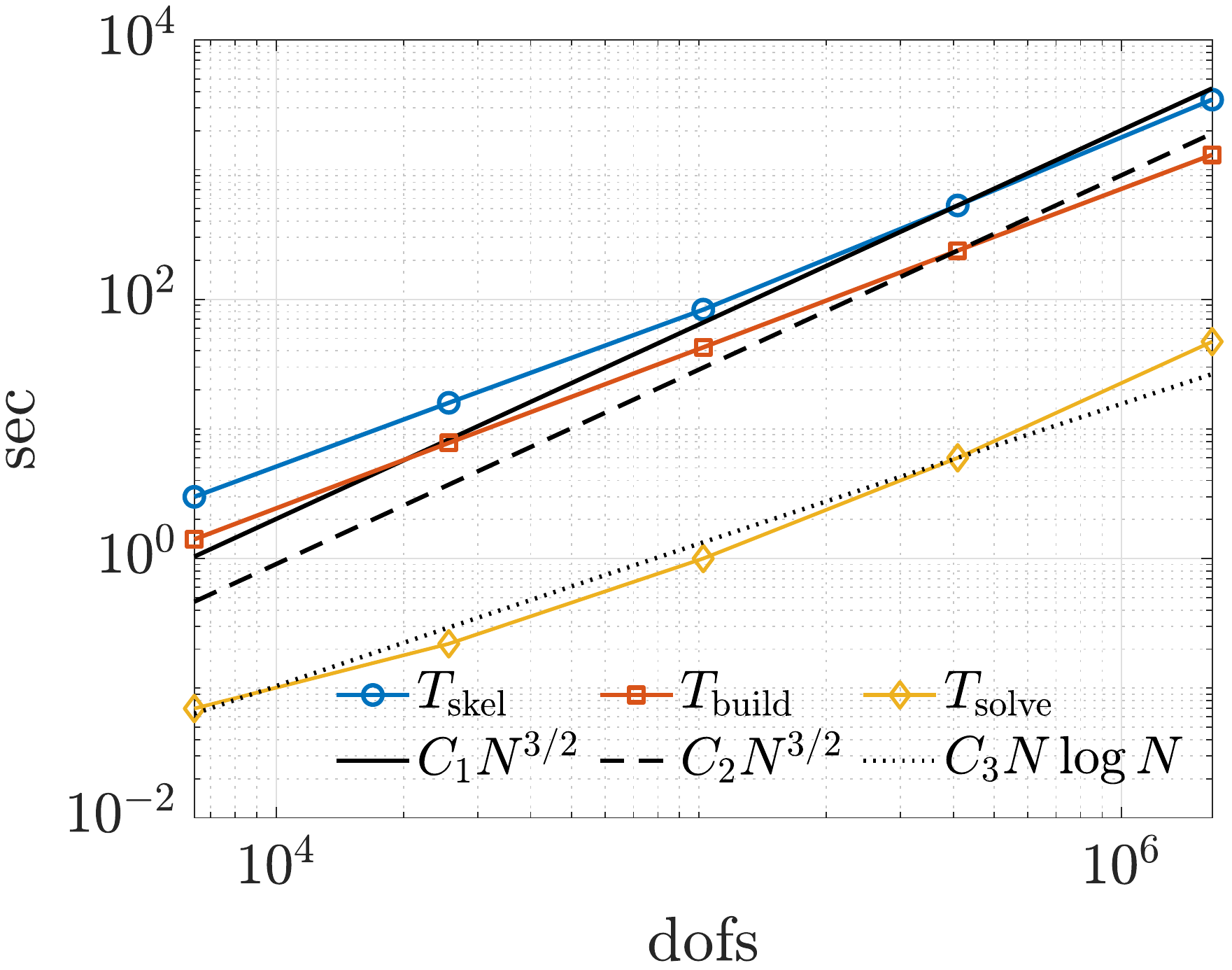}
  \caption{Scaling of HBS compression, inversion, and inverse application
    for experiment in Section \ref{s:gaussian} with compression tolerances
    $10^{-3}$ (top left), $10^{-6}$ (top right), $10^{-9}$ (bottom left),
    and $10^{-12}$ (bottom right).}
  \label{fig:gaussian_scaling}
\end{figure}

\subsection{The direct solver applied to a cavity problem}
\label{s:cavity}

Solutions to the problem considered in Section \ref{s:gaussian} do not
involve any backscattering; a fact that makes the problem highly amenable
to iterative methods.
Indeed, it was found in \cite{duan2009high} that for
such a potential even an iterative method without preconditioning converges
in just a few iterations, although a direct solver might still be
better-suited when one is interested in scattering for multiple incident
fields.

We now consider a cavity potential given in polar
coordinates by
\[
  b(r,\ta) = (1-\sin(0.5\ta)^{500}) \exp(-2000(0.1-r^2)^2)
\]
(see the left image in Figure \ref{fig:cavity1}). Without preconditioning, the number
of GMRES iterations required to solve even relatively low frequency
scattering to a couple of digits is unfeasibly high, due to the
near-resonant behavior of the problem.
We repeat the experiments of the previous section for this scattering
potential when $\ka = 50.27$, so that $\Om$ is about 8 wavelengths
across. The results are given in Tables \ref{t:cavity1}--\ref{t:cavity4}
and Figure \ref{fig:cavity_scaling}. The total field is plotted in the
right of Figure \ref{fig:cavity1}.

\begin{figure}
  \centering
  \includegraphics[width=58mm]{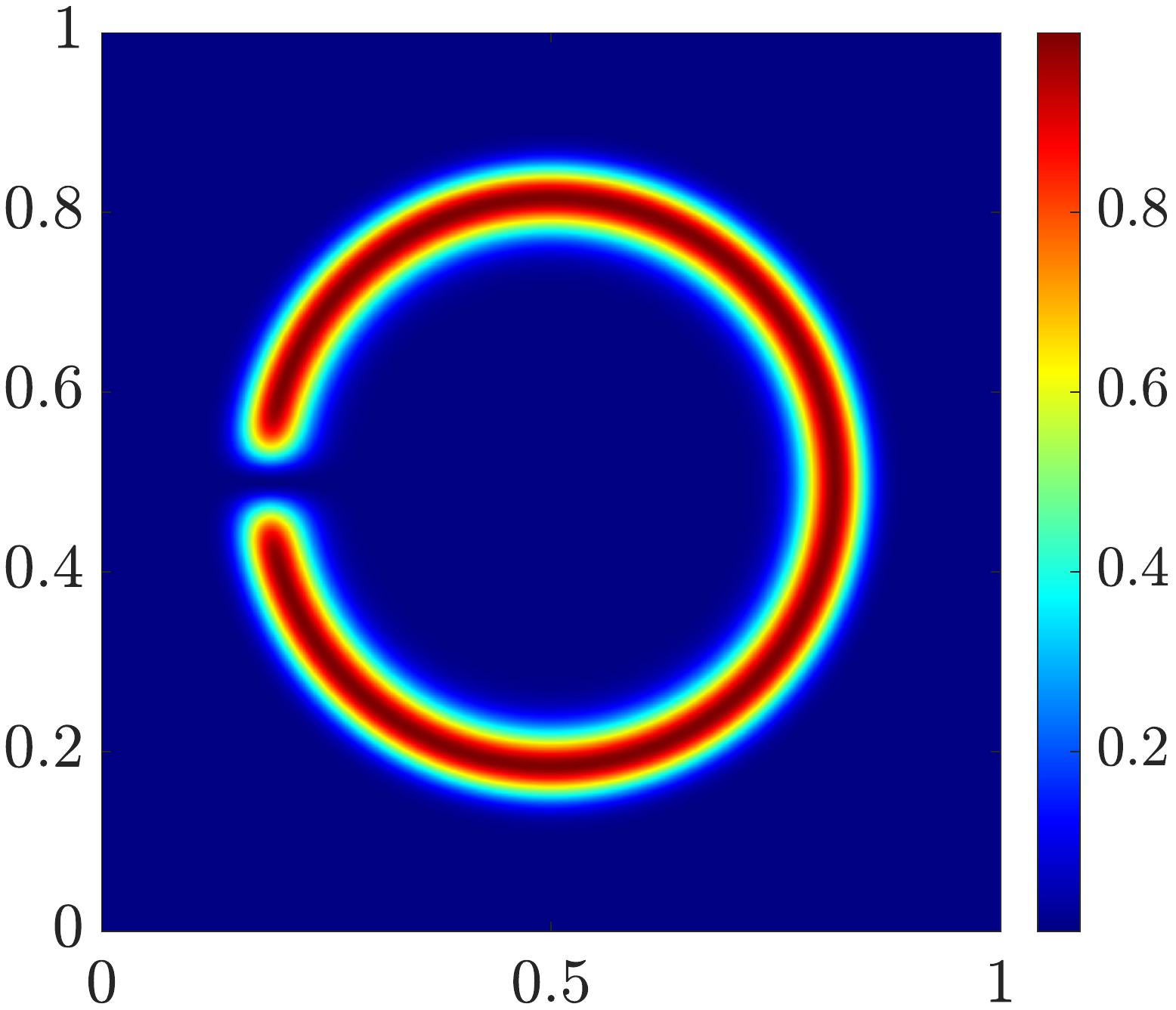} \hspace{2em}
  \includegraphics[width=56mm]{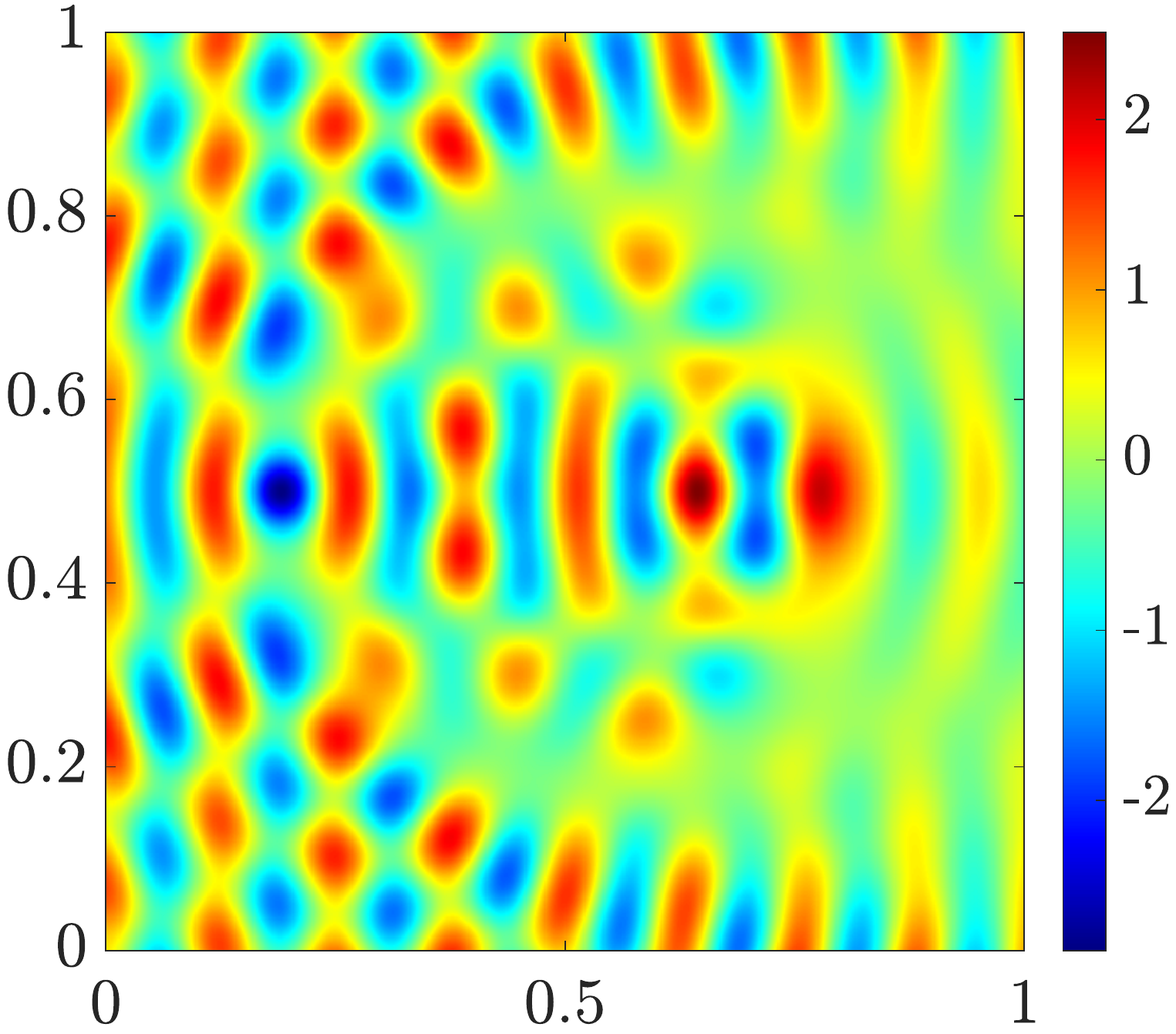}
  \caption{Left: Scattering potential for Section \ref{s:cavity}.
  Right: Total field for $\ka = 50.27$.}
  \label{fig:cavity1}
\end{figure}

\begin{table}[H]
  \centering
  \begin{tabular}{|ccccccc|}
    \hline
    $N$ & $h$ & $T_{\rm skel}$ & $T_{\rm build}$ & $T_{\rm apply}$ & mem &
    res\\
    \hline
    6400 & 0.0125 & 2.21 & 0.46 & 0.05 & 0.15 & 6.42e-05 \\
    25600 & 0.00625 & 5.36 & 2.16 & 0.14 & 0.67 & 2.82e-04 \\
    102400 & 0.003125 & 30.70 & 9.63 & 0.77 & 3.08 & 3.68e-04 \\
    409600 & 0.0015625 & 167.84 & 46.80 & 3.09 & 14.52 & 6.86e-04 \\
    1638400 & 0.00078125 & 940.77 & 222.96 & 13.16 & 65.10 & 1.14e-03 \\
    \hline
  \end{tabular}
  \caption{Data for Section \ref{s:cavity} with compression tolerance
  $10^{-3}$.}
  \label{t:cavity1}
\end{table}

\begin{table}[H]
  \centering
  \begin{tabular}{|ccccccc|}
    \hline
    $N$ & $h$ & $T_{\rm skel}$ & $T_{\rm build}$ & $T_{\rm apply}$ & mem &
    res\\
    \hline
    6400 & 0.0125 & 2.42 & 0.94 & 0.07 & 0.30 & 9.52e-08 \\
    25600 & 0.00625 & 12.59 & 4.97 & 0.42 & 1.57 & 5.89e-08 \\
    102400 & 0.003125 & 60.27 & 25.71 & 1.12 & 7.64 & 5.54e-07 \\
    409600 & 0.0015625 & 360.95 & 130.26 & 6.52 & 35.64 & 5.75e-07 \\
    1638400 & 0.00078125 & 2280.63 & 645.20 & 25.84 & 158.14 & 1.09e-06 \\
    \hline
  \end{tabular}
  \caption{Data for Section \ref{s:cavity} with compression tolerance
  $10^{-6}$.}
  \label{t:cavity2}
\end{table}

\begin{table}[H]
  \centering
  \begin{tabular}{|ccccccc|}
    \hline
    $N$ & $h$ & $T_{\rm skel}$ & $T_{\rm build}$ & $T_{\rm apply}$ & mem &
    res\\
    \hline
    6400 & 0.0125 & 2.61 & 1.11 & 0.07 & 0.36 & 4.23e-11 \\
    25600 & 0.00625 & 15.81 & 6.54 & 0.21 & 2.01 & 7.40e-11 \\
    102400 & 0.003125 & 79.86 & 36.36 & 1.06 & 10.38 & 2.67e-10 \\
    409600 & 0.0015625 & 491.72 & 206.02 & 9.55 & 50.42 & 4.59e-10 \\
    1638400 & 0.00078125 & 3289.37 & 1109.28 & 25.72 & 234.48 & 1.02e-08 \\
    \hline
  \end{tabular}
  \caption{Data for Section \ref{s:cavity} with compression tolerance
  $10^{-9}$.}
  \label{t:cavity3}
\end{table}

\begin{table}[H]
  \centering
  \begin{tabular}{|ccccccc|}
    \hline
    $N$ & $h$ & $T_{\rm skel}$ & $T_{\rm build}$ & $T_{\rm apply}$ & mem &
    res\\
    \hline
    6400 & 0.0125 & 2.72 & 1.23 & 0.06 & 0.39 & 3.28e-14 \\
    25600 & 0.00625 & 16.34 & 7.20 & 0.21 & 2.15 & 1.03e-13 \\
    102400 & 0.003125 & 82.90 & 40.68 & 1.16 & 10.95 & 6.29e-13 \\
    409600 & 0.0015625 & 515.89 & 227.17 & 5.08 & 53.34 & 3.36e-12 \\
    1638400 & 0.00078125 & 3495.09 & 1304.42 & 48.21 & 252.23 & 1.02e-11 \\
    \hline
  \end{tabular}
  \caption{Data for Section \ref{s:cavity} with compression tolerance
  $10^{-12}$.}
  \label{t:cavity4}
\end{table}

Besides a marginal increase in cost due to the slightly higher wavenumber,
the results are practically identical to those for the Gaussian bump. This
illustrates the fact that the direct solver, unlike iterative solvers, is
agnostic to features of the scattering potential.

\begin{figure}
  \centering
  \includegraphics[scale=0.39]{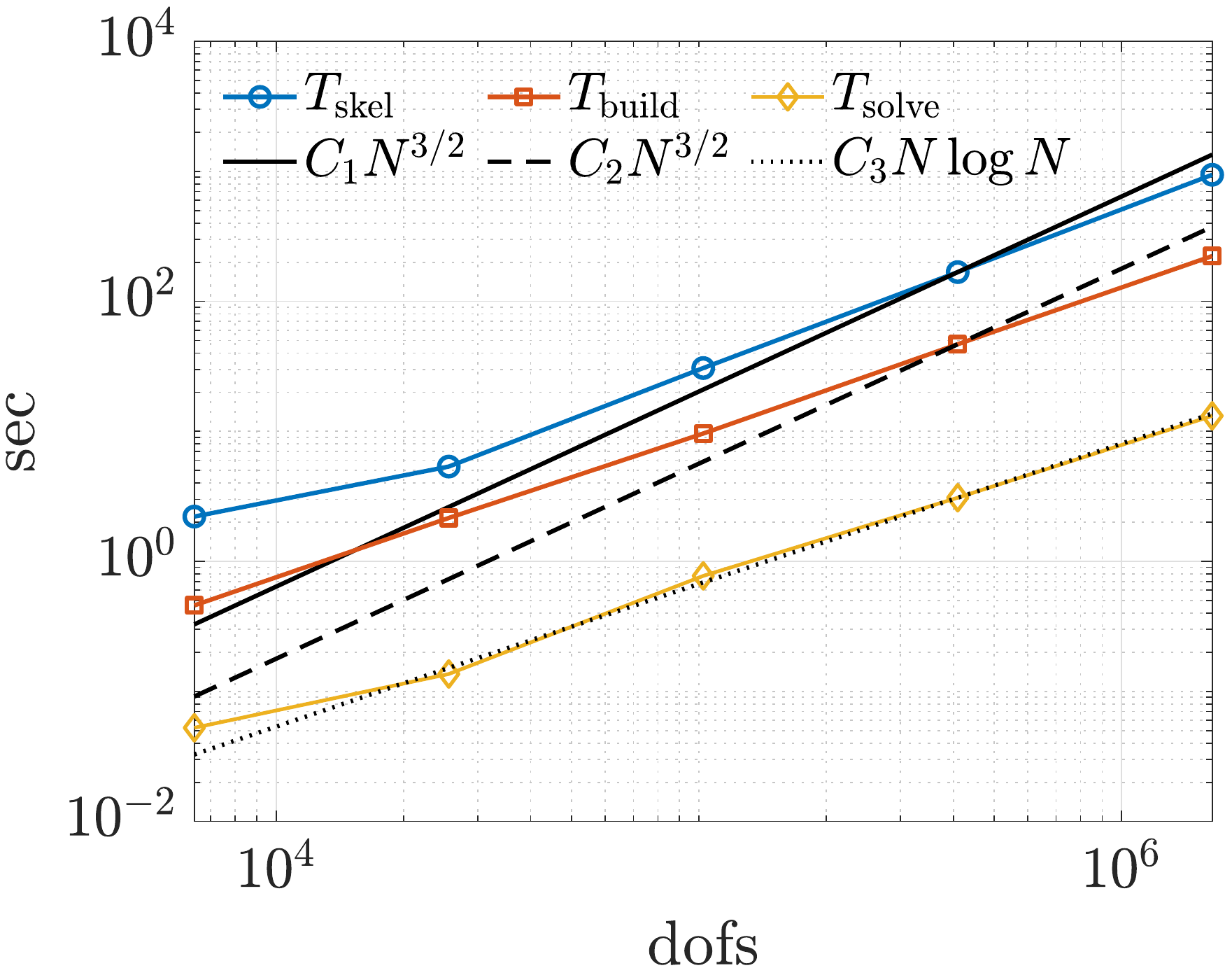} \hspace{2em}
  \includegraphics[scale=0.39]{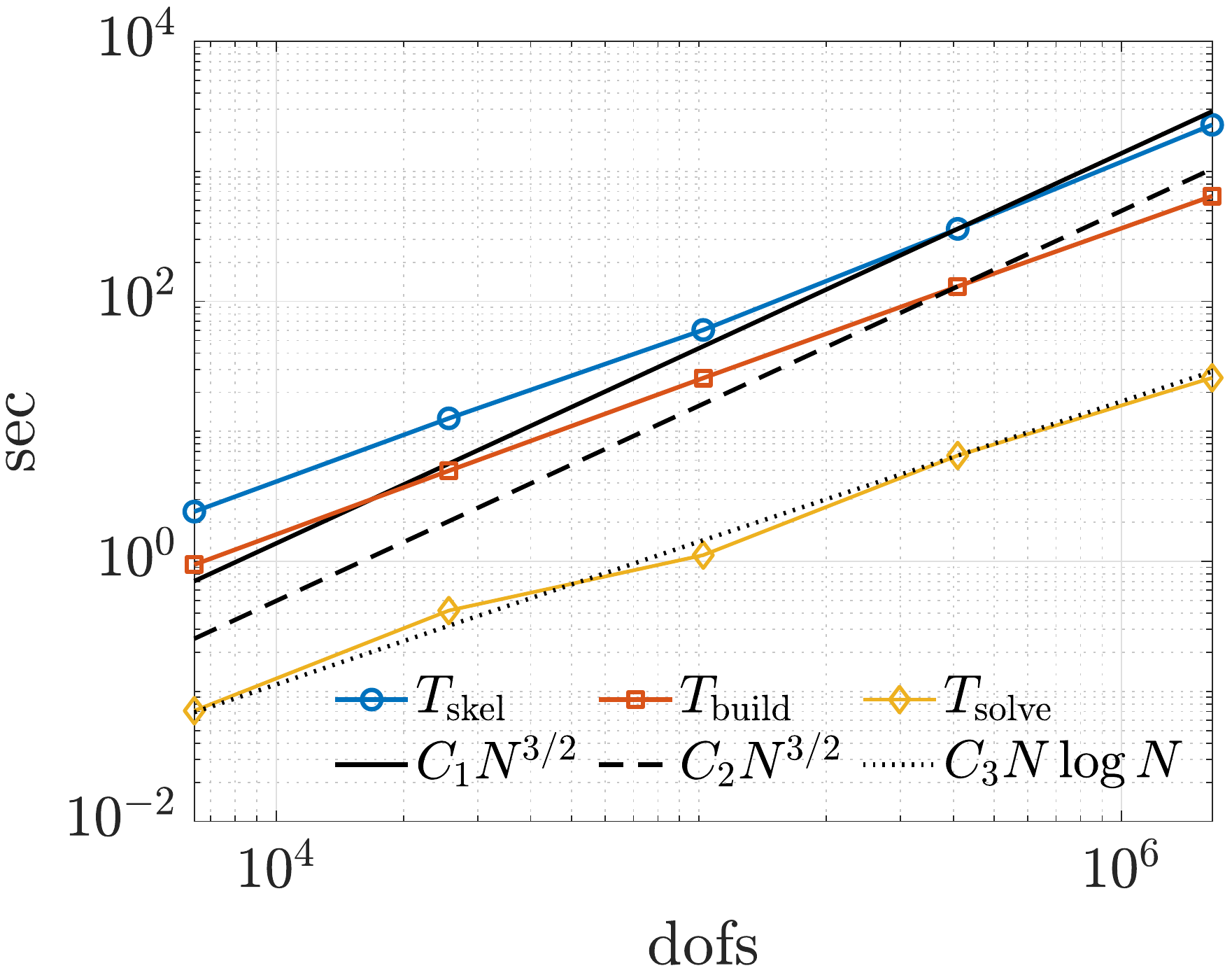} \\[1em]
  \includegraphics[scale=0.39]{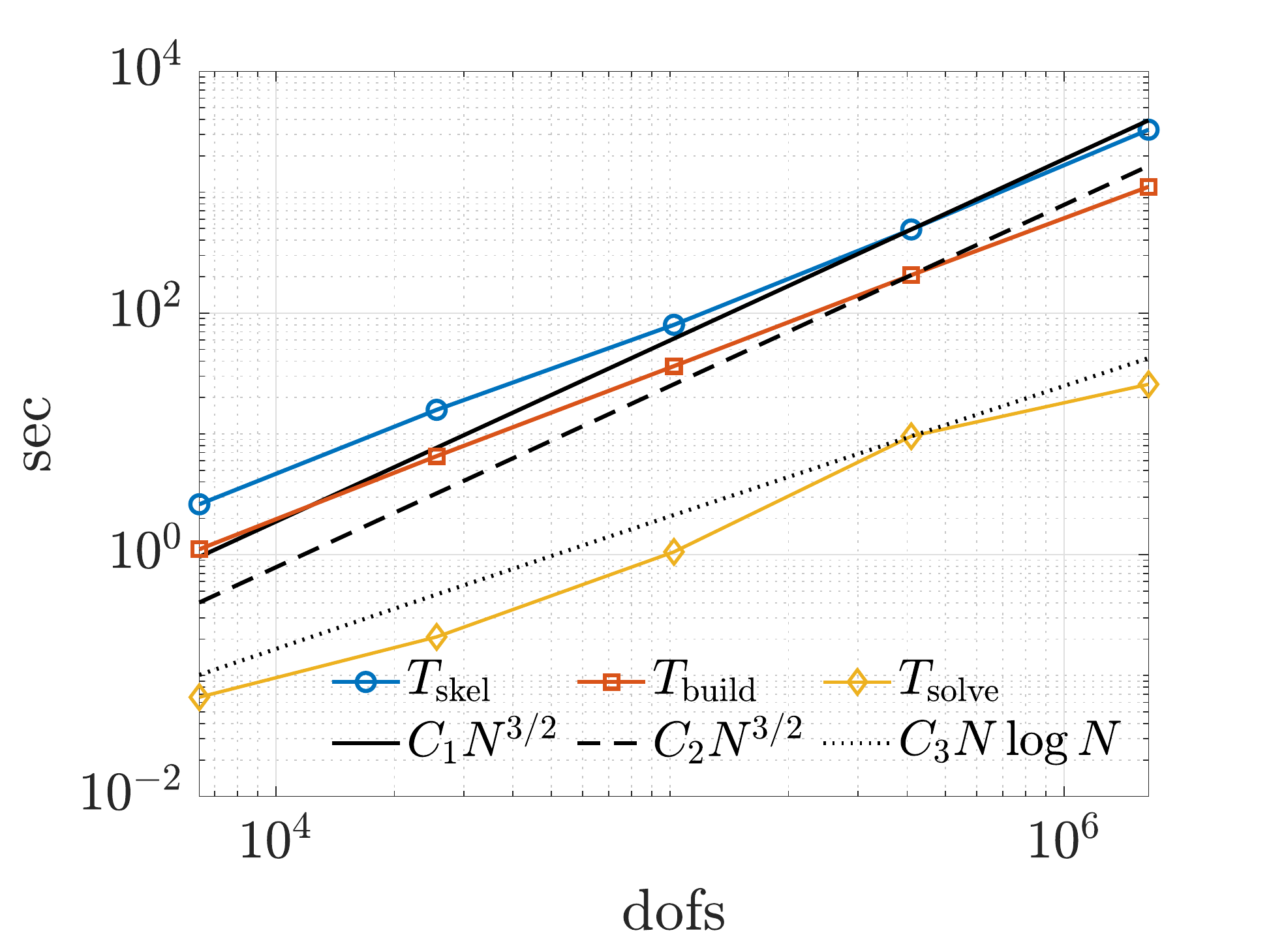} \hspace{1em}
  \includegraphics[scale=0.39]{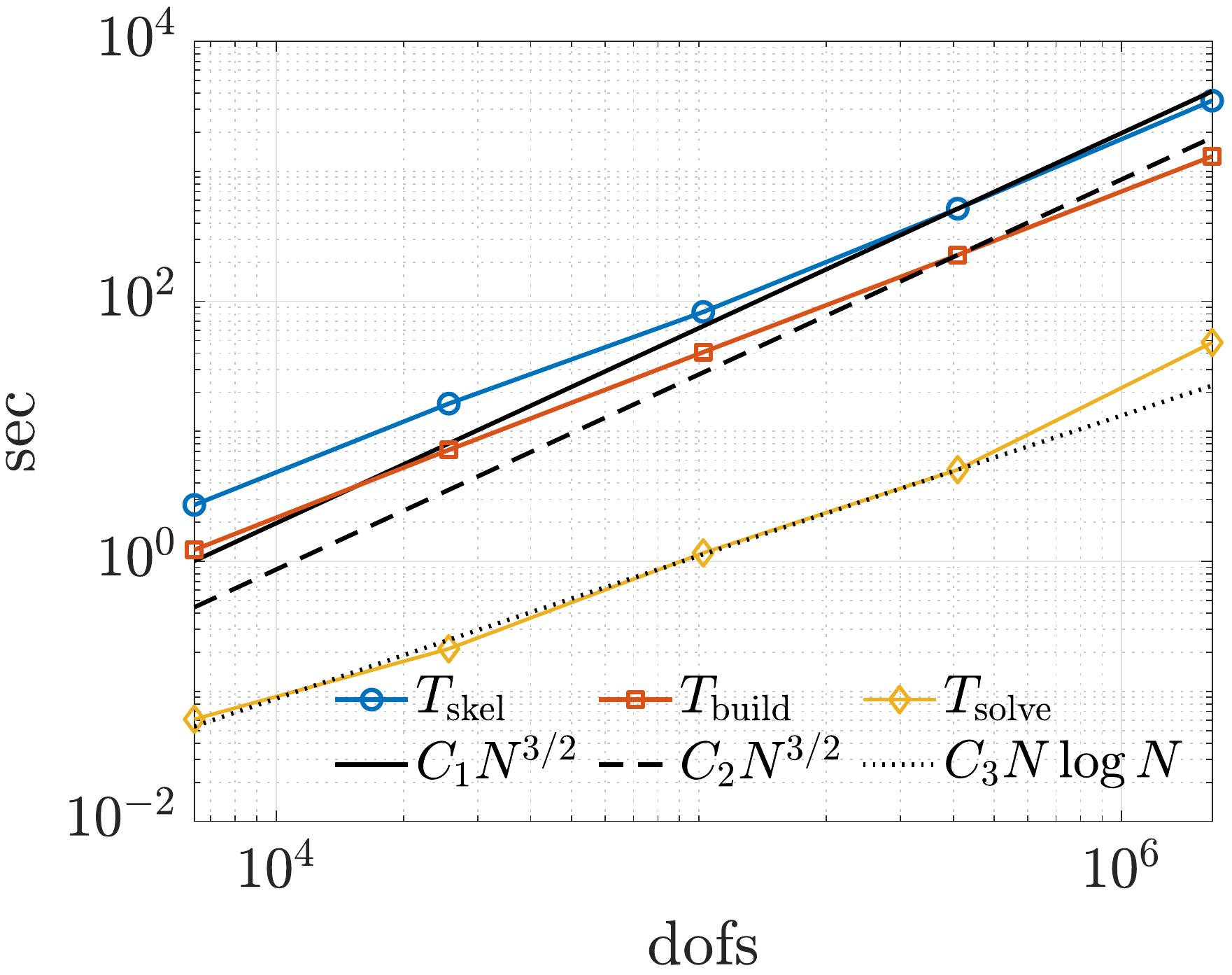}
  \caption{Scaling of HBS compression, inversion, and inverse application
    for direct solver experiment in Section \ref{s:cavity} with compression
    tolerances $10^{-3}$ (top left), $10^{-6}$ (top right), $10^{-9}$
    (bottom left), and $10^{-12}$ (bottom right).}
  \label{fig:cavity_scaling}
\end{figure}

\subsection{Using an approximate inverse as a preconditioner}
\label{s:hps_paper}

We will next demonstrate the effectiveness of using the direct
solver as a preconditioner.
We recall from Section \ref{s:precon} that the idea is to use
a low-order quadrature with a small correction stencil in order to improve compressibility
of the off-diagonal blocks of the coefficient matrix, and
a lower compression tolerance to obtain a leaner representation.
The resulting inverse is too rough to use as a direct solver for
the highly ill-conditioned problems under consideration, but is an
excellent preconditioner.

For an initial experiment, we stick with the cavity problem described
in Section \ref{s:cavity}. We use the a compression tolerance of $10^{-4}$ and
a $4^{\rm th}$-order quadrature corrections to build the preconditioner for
the GMRES iteration. The matrix-vector product is executed using the coefficient
matrix with $10^{\rm th}$-order corrections applied using the FFT as
discussed in Section \ref{s:quadcorr}. We apply this approach to the cavity
scattering potential where the number of points per wavelength is kept
fixed at 10 and the wavenumber is increased, with GMRES set to terminate
when the norm of the relative residual is below $10^{-10}$. In addition to
the relevant previously defined quantities, we also examine:

\begin{tabular}{p{4cm}p{10cm}}
   $T_{\rm gmres}$ & Total time spent by GMRES. \\
   iter & The number of GMRES iterations required to reach a relative
   residual with norm below $10^{-10}$. \\
\end{tabular}

\noindent Table \ref{t:cavity} reports these results for computational domains
ranging from approximately 8 to 512 wavelengths across. The total field
when $\ka = 201.06$ and $\ka = 804.25$ are displayed in the left and
right plots of Figure \ref{fig:cavity2}, respectively.

\begin{table}[H]
  \centering
  \begin{tabular}{|cccccccc|}
    \hline
    $N$ & $\ka$ & $T_{\rm skel}$ & $T_{\rm build}$ & $T_{\rm gmres}$ & mem
    & iter
    & res \\
    \hline
    6400 & 50.27 & 0.23 & 0.24 & 0.20 & 0.04 & 4 & 6.97e-11\\
    25600 & 100.53 & 0.65 & 0.99 & 0.62 & 0.21 & 5 & 6.16e-12 \\
    102400 & 201.06 & 2.26 & 4.36 & 2.49 & 1.01 & 6 & 1.04e-12 \\
    409600 & 402.12 & 14.91 & 20.06 & 9.78 & 4.67 & 6 & 3.23e-11 \\
    1638400 & 804.25 & 99.01 & 91.37 & 56.13 & 21.16 & 9 & 8.12e-12 \\
    6553600 & 1608.50 & 430.60 & 398.88 & 330.91 & 94.63 & 13 & 3.93e-11 \\
    26214400 & 3216.99 & 3102.09 & 2024.16 & 2698.53 & 418.37 & 22 &
    3.30e-11 \\
    \hline
  \end{tabular}
  \caption{Data for preconditioned iterative solver for the cavity problem
  in Section \ref{s:hps_paper}.}
  \label{t:cavity}
\end{table}

\begin{figure}
  \centering
  \includegraphics[width=58mm]{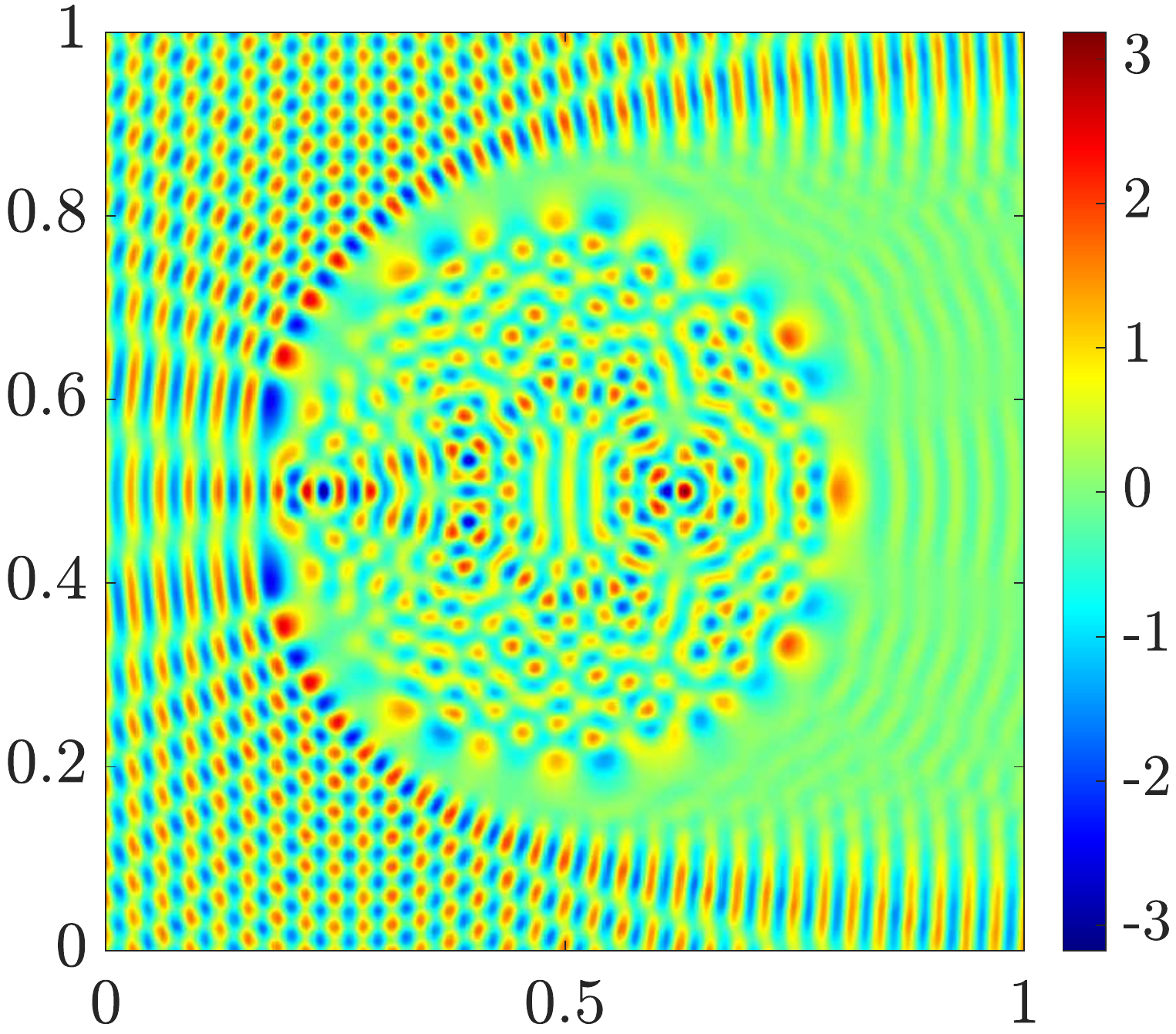} \hspace{2em}
  \includegraphics[width=58mm]{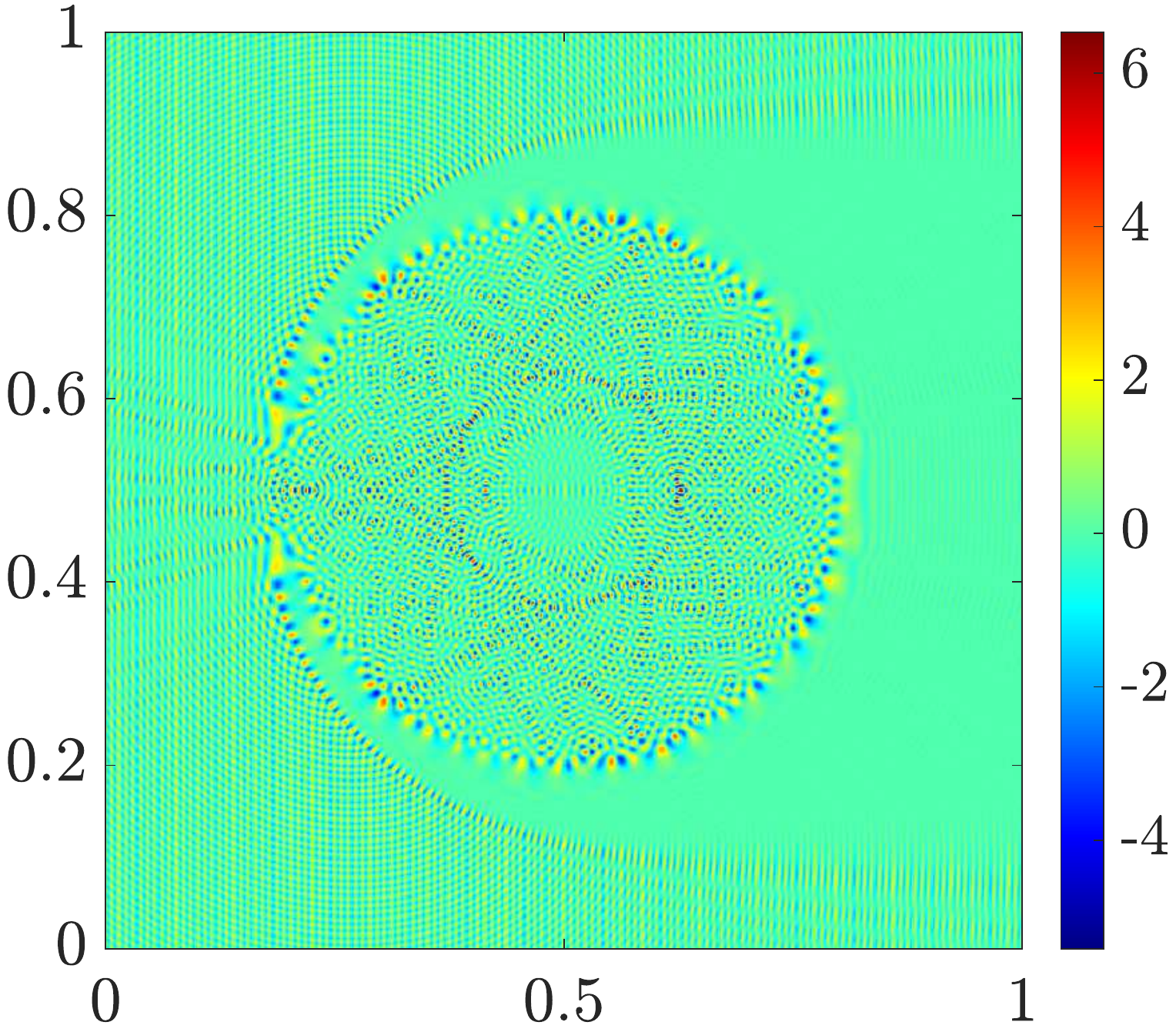}
  \caption{Total field for $\ka = 201.06$ (left) and $\ka = 804.25$ (right)
  for cavity potential in Figure \ref{fig:cavity1}.}
  \label{fig:cavity2}
\end{figure}

Without preconditioning, the problems in the larger domains in Table
\ref{t:cavity} would require thousands of GMRES iterations to achieve
convergence to a tolerance that yields accuracy for these intrinsically
ill-conditioned problems. To investigate whether the excellent performance
we observed for the cavity problem was a fluke or is the universal behavior
of the preconditioner, we next consider four challenging test problems that
were introduced in \cite{gillman2015spectrally}:

\begin{tabular}{p{3.5cm}p{11cm}}
    \textit{Lens:} & A vertically-graded lens given by $b(\pxx) = 4(x_2
    -0.1)[1-\operatorname{erf}(25(\|\pxx\|-0.3))]$ with $\ka = 300$ (Figure
    \ref{fig:hps_paper_pots}, top left). The
    incident field is given by a plane wave $u^{\rm inc}(\pxx) = \exp(i \ka
    (x_1 - 0.5))$. The computational domain is about 47.7 vacuum wavelengths
    across.\\
    \textit{Random bumps:} & The sum of 200 Gaussian bumps randomly placed in
    $\Om$ and rolled off to zero in a smooth fashion with $\ka = 160$
    (Figure \ref{fig:hps_paper_pots}, top right). The
    incident field is given by a plane wave $u^{\rm inc}(\pxx) = \exp(i \ka
    (x_1 - 0.5))$. The computational domain is about 25.5 vacuum wavelengths
    across. \\
    \textit{Photonic crystal:} & A 20 by 20 square grid of Gaussian bumps with a
    channel removed with $\ka = 77.1$ (the first complete band gap) (Figure
    \ref{fig:hps_paper_pots}, bottom). The incident field is given by a
    plane wave $u^{\rm inc}(\pxx) = \exp(i \ka (x_1 - 0.5))$. The
    computational domain is about 12.2 vacuum wavelengths across.\\
   \textit{Photonic crystal 2:} & Same geometry and wavenumber as above. The
incident field is now given by a plane wave $u^{\rm inc}(\pxx) = \exp(i \ka
(-x_1/\sqrt{2} + x_2/\sqrt{2}-1))$.
\end{tabular}

As in the previous section, for each problem we use our direct solver as a
preconditioner for GMRES with $4^{\rm th}$-order quadrature corrections and
compression tolerance set to $10^{-4}$, and study the behavior as the mesh
is refined. The iterative solver again uses the $10^{\rm th}$-order
quadrature corrections and matrix-vector multiplications are done using the
FFT. GMRES is set to terminate when the norm of the relative residual
falls below $10^{-10}$ or 100 iterations is reached. The two previously
undefined quantities are also reported in Table \ref{t:hps_paper}:
\begin{center}
\begin{tabular}{p{4cm}p{10cm}}
  $E_{\rm near}$ & The absolute error between the real part of the
    computed solution at the point $(.75,.5)$ and the reference solution
    given in \cite{gillman2015spectrally}. \\
  $E_{\rm far}$ & The absolute error between the real part of the
    computed solution at the point $(1.5,1)$ and the reference solution
    given in \cite{gillman2015spectrally}.
\end{tabular}
\end{center}

\noindent The real parts of the total field for each experiment is plotted in Figure
\ref{fig:hps_paper_fields}. Table \ref{t:hps_paper} clearly demonstrates
how powerful the preconditioner is. For problems involving 26M degrees of
freedom, we obtain nine or ten correct digits with between 9 and 23
iterations across the different examples.  Let us stress that the
quantities $E_{\rm near}$ and $E_{\rm far}$ report errors in solving the
PDE, not just in solving the linear system.

\begin{figure}
  \centering
  \includegraphics[scale=0.3]{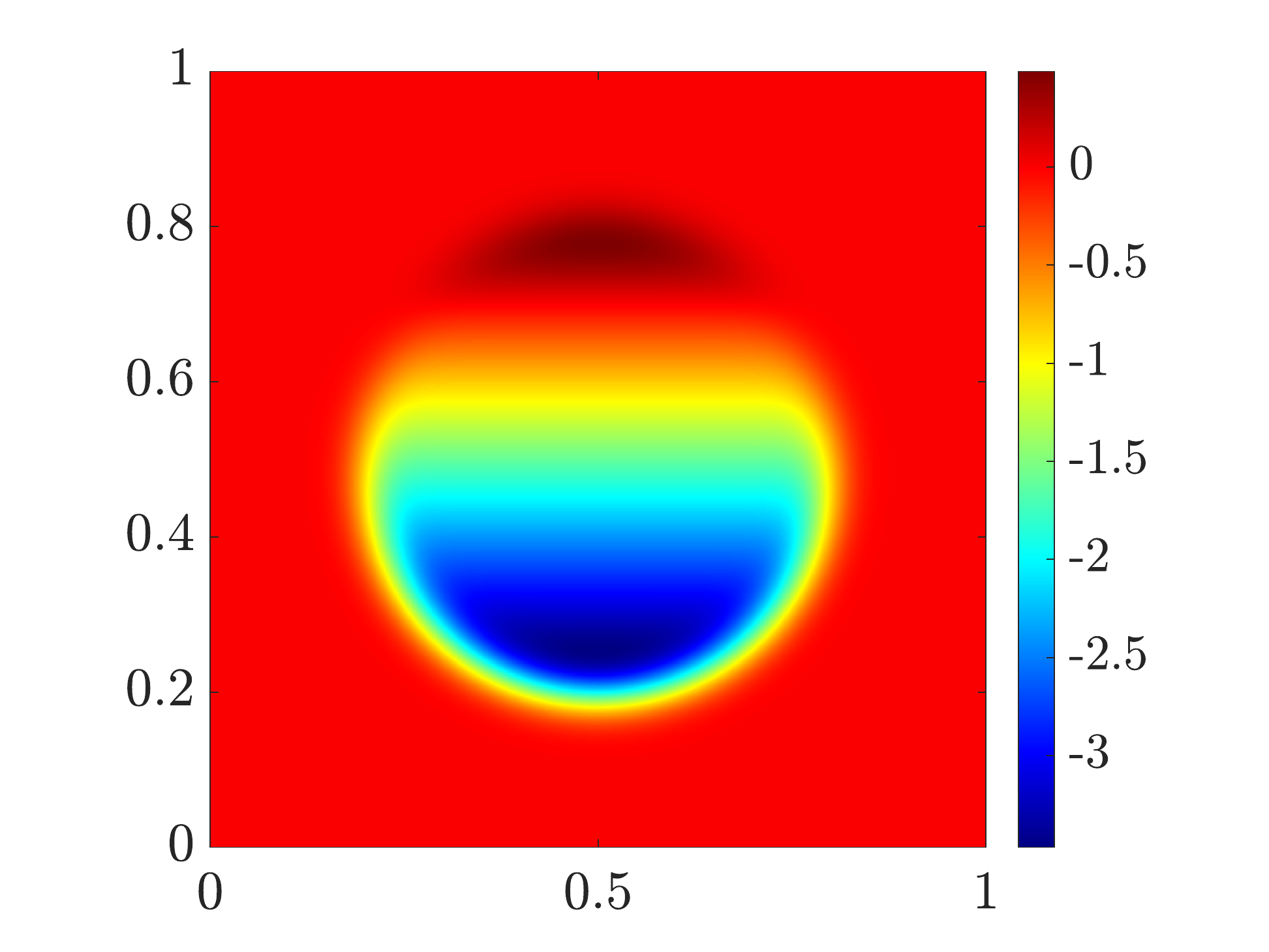}
  \includegraphics[scale=0.3]{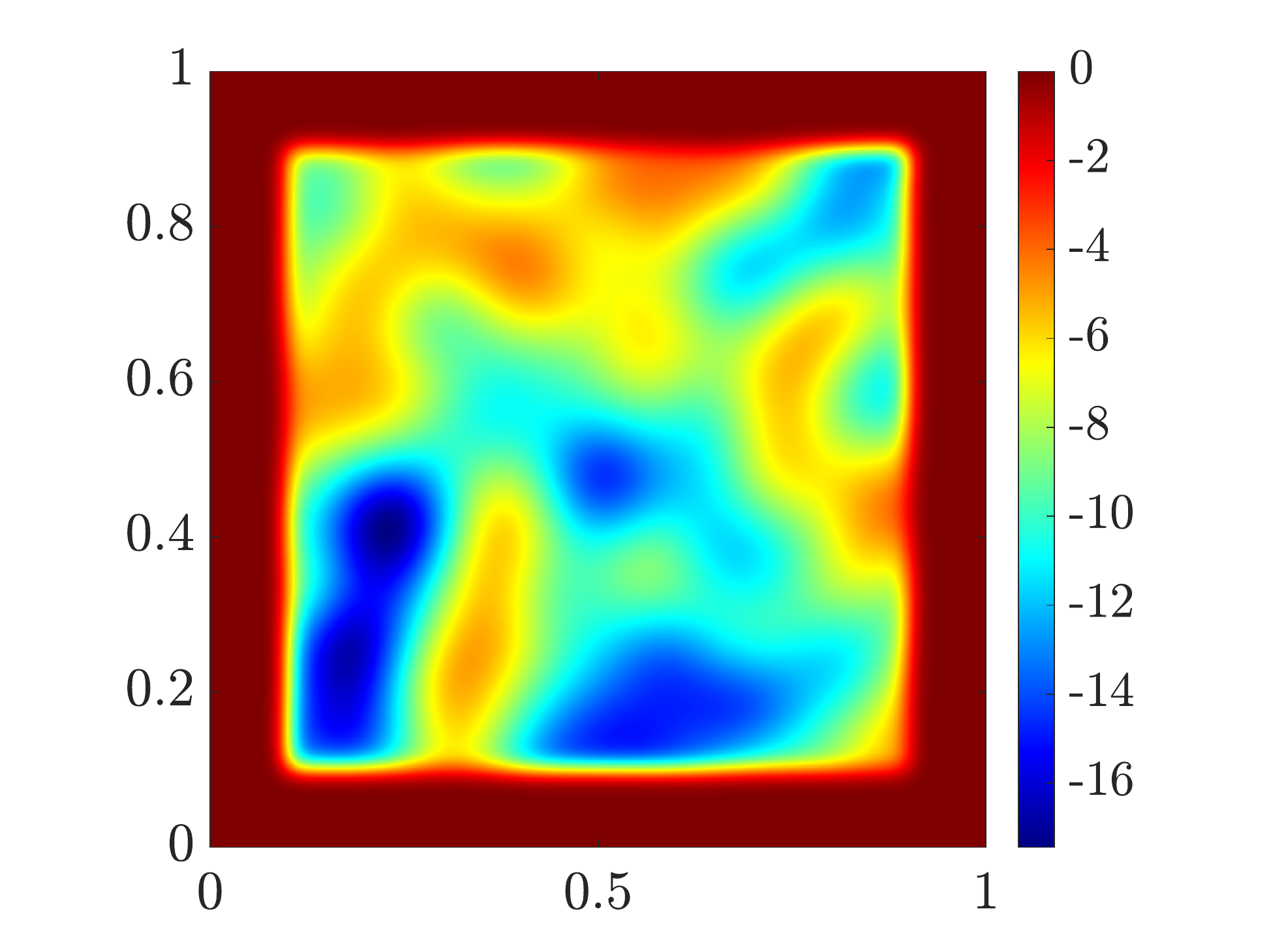}
  \begin{center}
    \includegraphics[scale=0.3]{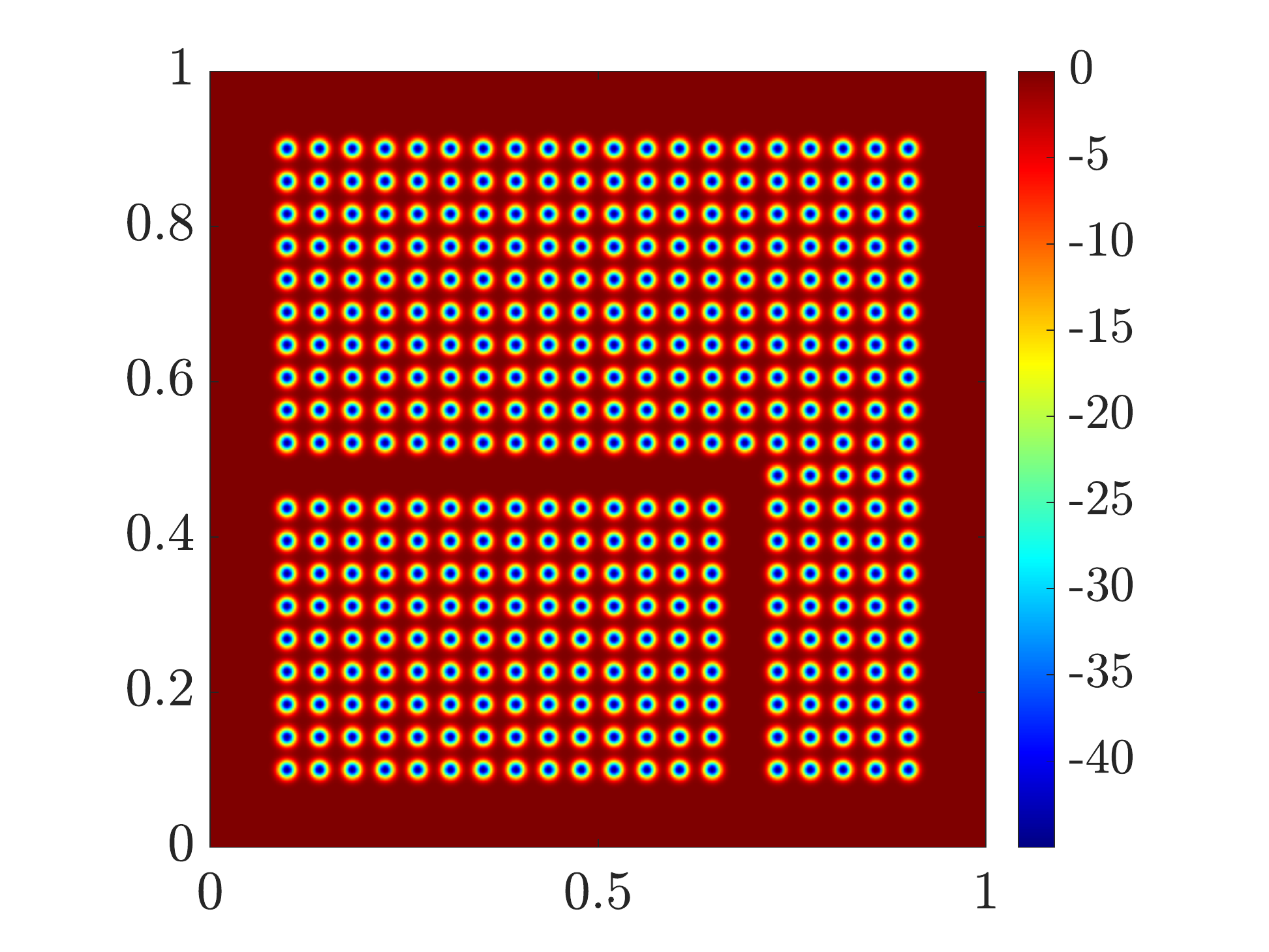}
  \end{center}
  \caption{Lens (top left), random bumps (top right), and photonic crystal
    (bottom)
  scattering potentials.}
  \label{fig:hps_paper_pots}
\end{figure}

\begin{table}[H]
  \centering
  \begin{tabular}{|c|cccccc|}
    \hline
    & $N$ & $h$ & iter & res & $E_{\rm near}$ & $E_{\rm far}$ \\ \hline
    Lens & 102400 & 0.003125 &  51 & 6.38e-11 & 1.27e-02 & 2.06e-03 \\
    & 409600 & 0.015625 & 9 & 6.18e-11 & 1.05e-05 & 6.77e-07 \\
    & 1638400 & 0.00078125 & 7 & 4.86e-12 & 1.14e-08 & 5.22e-10 \\
    & 6553600 & 0.000390625 & 7 & 7.64e-12 & 8.29e-11 & 2.81e-10 \\
    & 26214400 & 0.0001953125 & 9 & 5.41e-12 & 7.85e-11 & 2.79e-10 \\
    \hline Bumps & 102400 & 0.003125 & $100$ & 3.69e-02 & 3.14e-01 &
    2.53e-02\\
    & 409600 & 0.015625 & 22 & 9.84e-11 & 3.90e-04 & 2.82e-04 \\
    & 1638400 & 0.00078125 & 10 & 8.05e-12 & 4.17e-07 & 3.01e-07 \\
    & 6553600 & 0.000390625 & 13 & 9.24e-11  & 5.29e-10  & 6.83e-10 \\
    & 26214400 & 0.0001953125 & 23 & 3.71e-11 & 1.20e-10 & 9.90e-10 \\
    \hline Crystal & 102400 & 0.003125 & 24 & 6.49e-11 & 1.67e-03 & 1.09e-03 \\
    & 409600 & 0.015625 & 8 & 3.70e-12 & 2.96e-06 & 1.47e-06 \\
    & 1638400 & 0.00078125 & 7 & 2.30e-12 & 3.68e-09 & 1.29e-09 \\
    & 6553600 & 0.000390625 & 15  & 2.24e-11 & 4.81e-10  & 2.94e-10 \\
    & 26214400 & 0.0001953125 & 11 & 5.23e-11 & 4.57e-10 & 3.00e-10 \\
    \hline Crystal 2 & 102400 & 0.003125 & 24 & 9.99e-11 & 5.15e-03 & 4.01e-03  \\
    & 409600 & 0.015625 & 8 & 2.36e-12 & 7.12e-06 & 5.97e-06 \\
    & 1638400 & 0.00078125 & 7 & 1.82e-12 & 7.88e-09 & 1.15e-08 \\
    & 6553600 & 0.000390625  & 14 & 8.33e-11  & 1.67e-10 & 5.01e-09 \\
    & 26214400 & 0.0001953125 & 11 & 4.75e-11 & 1.27e-10 & 5.00e-09 \\
    \hline
  \end{tabular}
  \caption{Data for problems from \cite{gillman2015spectrally} in Section \ref{s:hps_paper}.}
  \label{t:hps_paper}
\end{table}

\begin{figure}
  \centering
  \includegraphics[width=70mm]{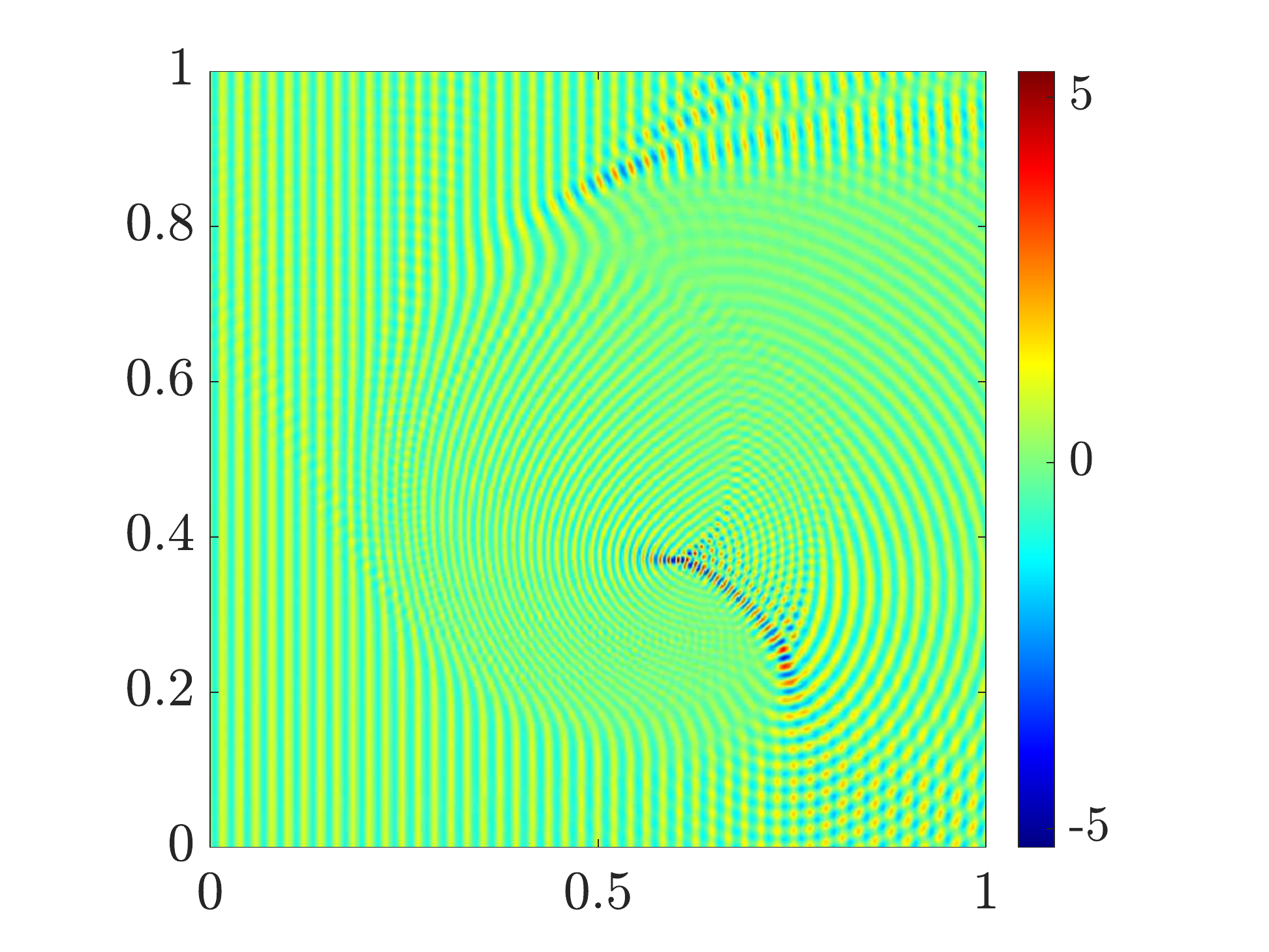}
  \includegraphics[width=70mm]{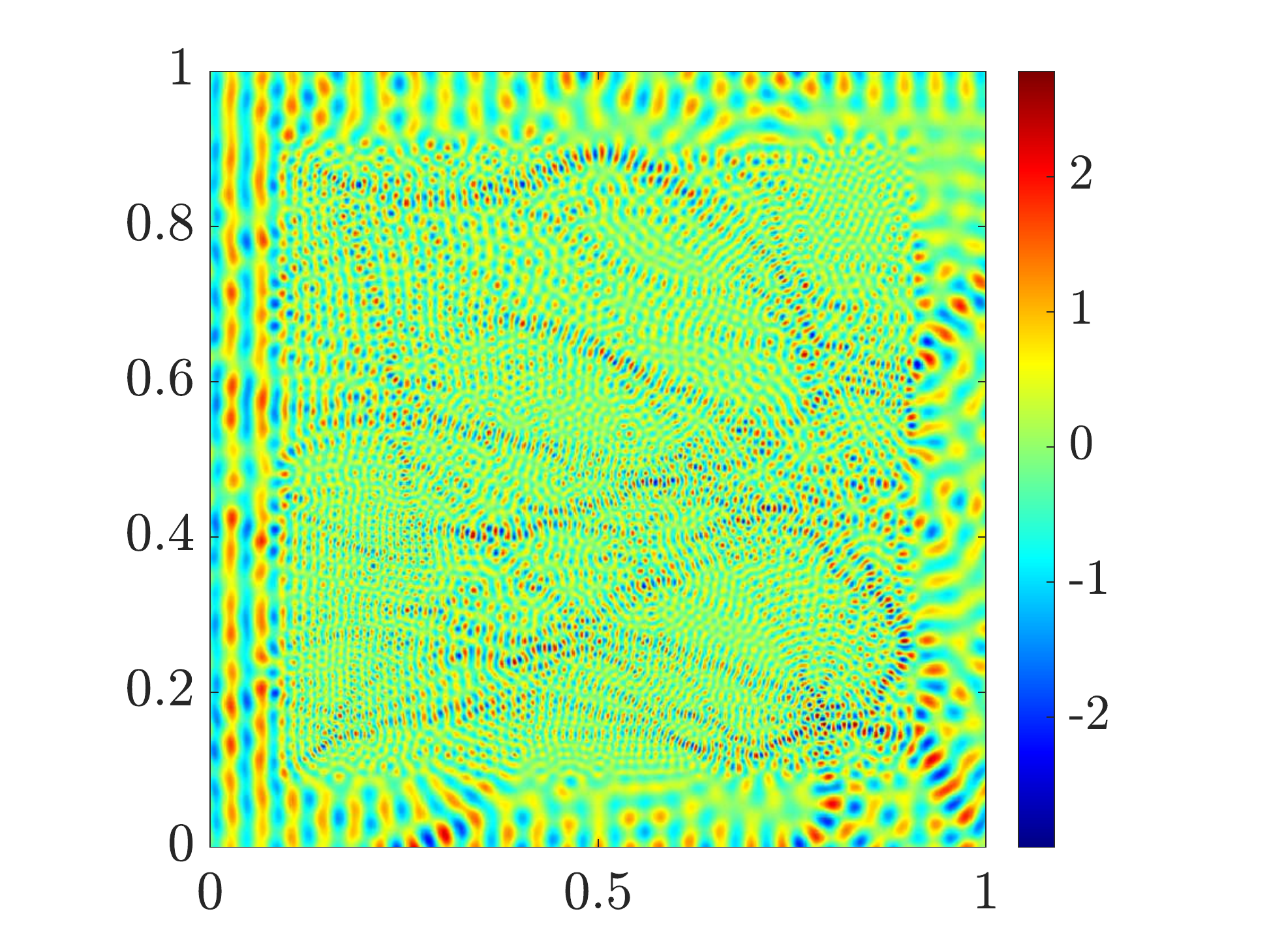}  \\
  \includegraphics[width=70mm]{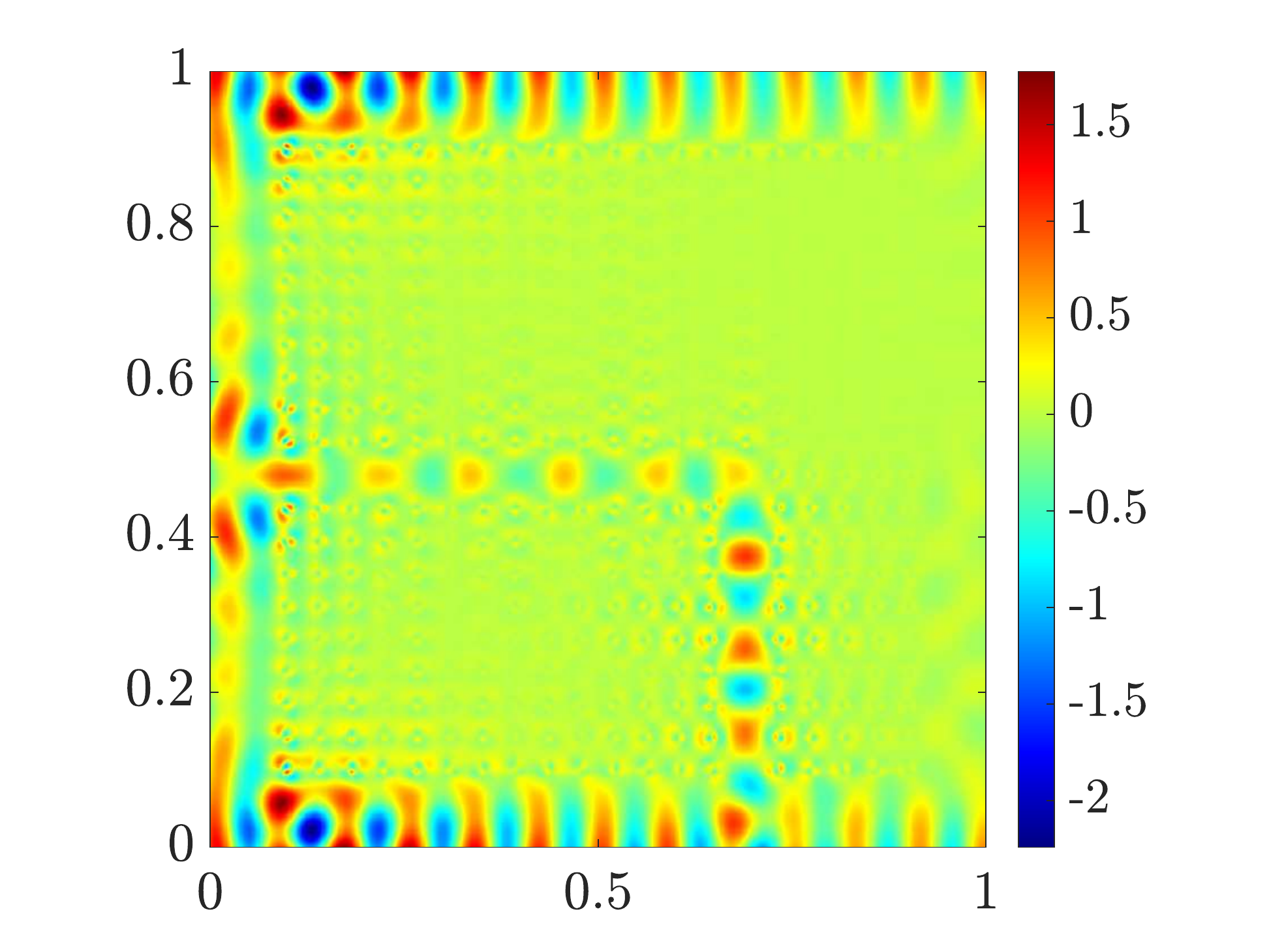}
  \includegraphics[width=70mm]{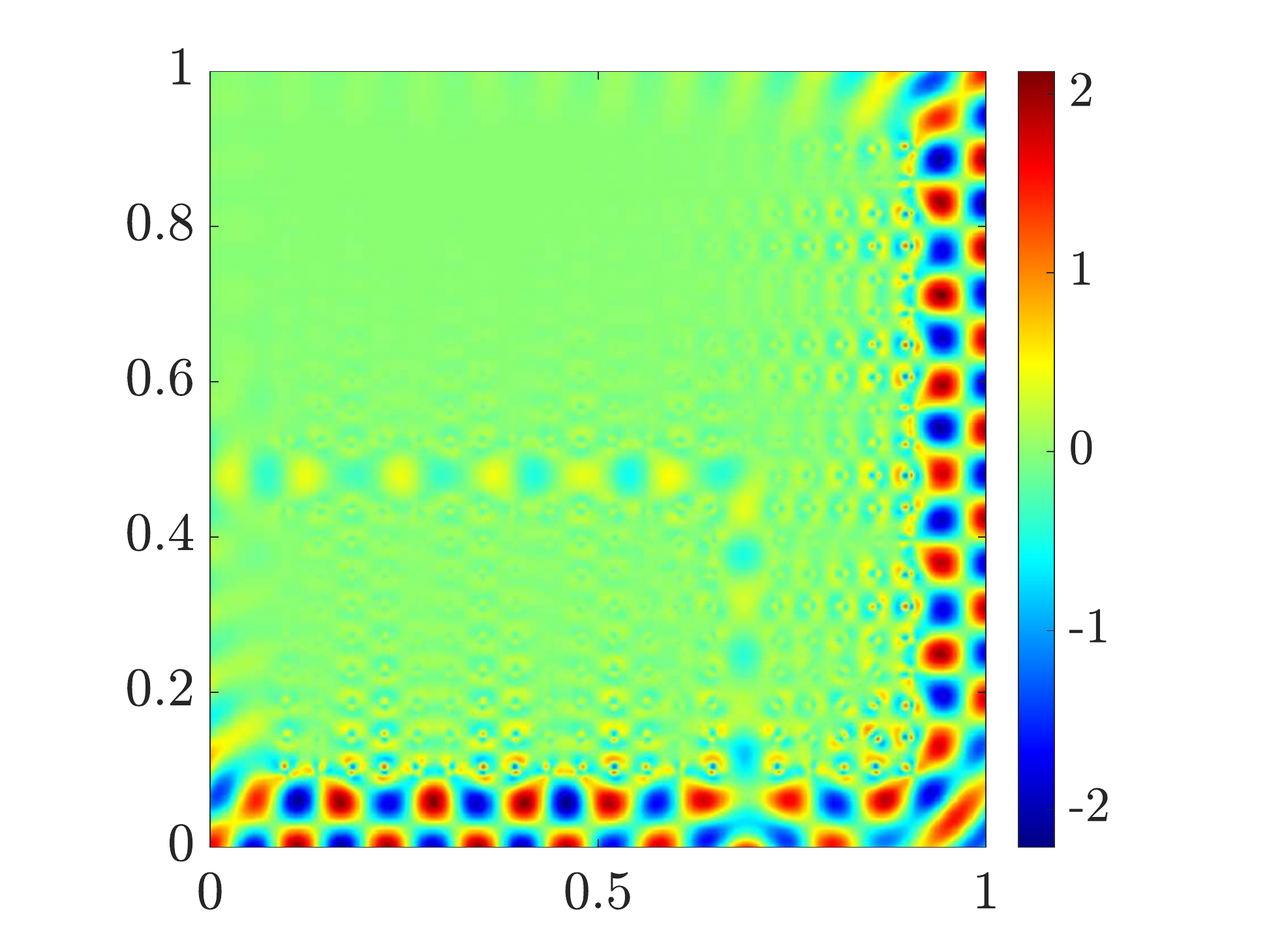}
  \caption{The real part of the total fields for the challenging problems
    from \cite{gillman2015spectrally} in Section
  \ref{s:hps_paper}.}
  \label{fig:hps_paper_fields}
\end{figure}

\section{Conclusions}
This paper describes an accelerated, high-order direct solver for the
Lippmann-Schwinger equation for acoustic scattering.
The solver is based on the Hierarchically Block Separable format, and
specifically the formulation based on scattering matrices described in
\cite{bremer2015high} and \cite[Ch.~18]{martinsson2019book}.
We demonstrate that significant acceleration can be achieved by taking
advantage of the particular form of the  kernel function in this application,
and of simplifications that arise from the use of uniform grids.
The algorithm for computing an approximate inverse has asymptotic complexity
$\mc{O}(N^{3/2})$, while applying the inverse has complexity $\mc{O}(N \log N)$.

Numerical experiments substantiate the claims in regards to the asymptotic
complexity and demonstrate very favorable practical performance.
A particularly effective method results from using the direct solver as a
preconditioner for GMRES, with numerical experiments showing how a highly
ill-conditioned cavity problem that is 500 wavelengths across can be solved
to ten accurate digits in a couple of hours on a single workstation, using
26M degrees of freedom.

The specific direct solver described here is designed primarily for handling
highly ill-conditioned problems in the high-frequency scattering regime. In
this environment, the asymptotic complexities reported match those of existing
state of the art solvers when the number of points per wavelength is fixed
as the problem size $N$ is increased. Ongoing work is concerned with accelerating
the solver to attain linear complexity in the situation where the elliptic PDE
either has non-oscillatory solutions, or where the PDE is kept fixed as the problem
size is increased, cf.~Remark \ref{rmk:fast}.

\subsubsection*{Acknowledgments}

Vladimir Rokhlin has generously shared his insights about the problems
under consideration, and we gratefully acknowledge his contributions.
We thank Ran Duan for sharing codes to compute the quadrature weights,
and Alex Barnett for sharing codes to compute the photonic crystal.
The work of PGM was funded by the Office of Naval Research (award N00014-18-1-2354),
by the National Science Foundation (awards 1952735 and 1620472), and by Nvidia Corp.

\bibliography{p}{}
\bibliographystyle{plain}

\end{document}